\documentclass[twoside,11pt]{article}

\usepackage[american]{babel}

\usepackage{url}
\usepackage{epsfig} 
\usepackage{times} 
\usepackage{amsmath}
\usepackage{amsfonts}
\usepackage{amssymb}  
\usepackage{mathtools}

\usepackage{xcolor}
\usepackage{transparent}
\usepackage{graphicx}
\graphicspath{{media/}}
\usepackage{tikz}
\usepackage{algpseudocode}
\usepackage{algorithm}
\usepackage{tabularx}

\usepackage{epstopdf}
\usepackage{units}
\usepackage{pgfplots}
\usepgfplotslibrary{groupplots}

\usepackage{pstool}

\graphicspath{{media/}}
\newcommand{\executeiffilenewer}[3]{%
 \ifnum\pdfstrcmp{\pdffilemoddate{#1}}%
 {\pdffilemoddate{#2}}>0%
 {\immediate\write18{#3}}\fi%
}

\usepackage{arydshln}
\usepackage{perpage}
\MakePerPage{footnote}

\newtheorem{claim}{Claim}
\newtheorem{assumption}{Assumption}
\newtheorem{postulate}{Postulate}

\pgfplotsset{every tick label/.append style={font=\footnotesize}}

\usetikzlibrary{arrows}
\usetikzlibrary{shapes.geometric}
\usetikzlibrary{arrows.meta}
\usetikzlibrary{matrix}

\newcommand{\argmin}{\mathop{\mathrm{argmin}}}
\newcommand{\argmax}{\mathop{\mathrm{argmax}}}

\newcommand{\TT}{\ensuremath{\mathsf{\tiny{T}}}}
\newcommand{\T}{^{\TT}}

\newcommand{\diff}[1][]{\mathrm{d}#1}
\newcommand{\dt}{\diff t }

\newcommand{\hchange}[1]{#1}
\newcommand{\hchangeII}[1]{#1}

\usepackage[preprint]{jmlr2e}
\usepackage{lastpage}

\jmlrheading{23}{2022}{1-\pageref{LastPage}}{7/21; Revised 3/22}{8/22}{21-0798}{Michael Muehlebach and Michael I. Jordan} \ShortHeadings{On Constraints in First-Order Optimization}{Muehlebach and Jordan}

\firstpageno{1}

\begin{document}

\title{On Constraints in First-Order Optimization: A View from Non-Smooth Dynamical Systems}

\author{\name Michael Muehlebach \email michaelm@tuebingen.mpg.de \\
       \addr  Learning and Dynamical Systems Group \\ 
       Max Planck Institute for Intelligent Systems \\
       72076 T\"ubingen, Germany
       \AND
       \name Michael I.\ Jordan \email jordan@cs.berkeley.edu \\
       \addr Department of Electrical Engineering and Computer Sciences\\
       Department of Statistics\\
       University of California\\
       Berkeley, CA 94720, USA}

\editor{Prateek Jain}

\maketitle

\begin{abstract}
We introduce a class of first-order methods for smooth constrained optimization that are based on an analogy to non-smooth dynamical systems. Two distinctive features of our approach are that (i) projections or optimizations over the entire feasible set are avoided, in stark contrast to projected gradient methods or the Frank-Wolfe method, and (ii) iterates are allowed to become infeasible, which differs from active set or feasible direction methods, where the descent motion stops as soon as a new constraint is encountered. The resulting algorithmic procedure is simple to implement even when constraints are nonlinear, and is suitable for large-scale constrained optimization problems in which the feasible set fails to have a simple structure.  The key underlying idea is that constraints are expressed in terms of velocities instead of positions, which has the algorithmic consequence that optimizations over feasible sets at each iteration are replaced with optimizations over local, sparse convex approximations. \hchange{In particular, this means that at each iteration only constraints that are violated are taken into account.} The result is a simplified suite of algorithms and an expanded range of possible applications in machine learning.
\end{abstract}


\section{Introduction}
Optimization has played an essential role in machine learning in recent years, providing a conceptual and practical platform on which algorithms, systems, and datasets can be brought together at unprecedented scales. This joint platform has led to high-impact applications, the discovery of new phenomena, and the development of new theory.  One of the major themes that have catalyzed the interplay between optimization and learning is that ``simple is good.''  Whereas classical optimization has tended to focus on relatively complex schemes for determining update directions and step sizes, the recent focus of research at the learning/optimization interface
has been on algorithms that use simple, stochastic approximations to first-order operators and employ step sizes that are set via simple averaging schemes, or even use constant step sizes.  The simplifications have worked well in practice and have triggered the development of commodity software systems that are increasingly general and robust.  They have also, appealingly, created new challenges for theoreticians, who have begun to develop new tools to fill in the gaps that the absence of strong assumptions has opened up.

Somewhat overlooked in all of these developments is the treatment of \emph{constraints} in machine-learning problems.  Machine-learning practitioners often handle constraints on parameters and predictions via simple, adhoc reparameterizations.  This reflects the ``simple is good'' dictum, but it also creates a need to develop special-case reparameterizations in many cases and it poses additional challenges for theory, as convergence rates can be affected by the reparameterizations.  More significantly, it overlooks the broader potential role that constrained optimization can play in machine learning.  Moving beyond pattern recognition, emerging problems involving decision-making in real-world, multi-agent settings often involve contextual-driven
constraints.  Control-theoretic problems generally involve interactions with physical, biological, and social systems, whose laws are often expressed in terms of fundamental constraints.  Mathematically, constraints can simplify statements of existence and uniqueness, simplify the specification of sets of solutions, and allow duality principles to be brought to bear.

There is a nascent thread of research on constrained optimization in machine learning that has aimed to build on the success of first-order methods.  It has focused primarily on \emph{projected gradient algorithms} and the \emph{Frank-Wolfe method}.  Both of these methods involve an inner loop that is nested inside of the overall procedure---in the former case the optimization of a quadratic function and in the latter case the
optimization of a linear function.  In both cases the optimization is over the entire feasible set.  From a theoretical point of view, these are relatively simple methods, providing hooks such that convergence analyses from the unconstrained case can be readily brought to bear.  Moreover, they can be easy to implement when the feasible set has a simple structure, such as a norm ball or a low-dimensional hyperplane.  In these cases it is often possible to obtain closed-form expressions for the inner loop.  This simplicity can disappear entirely, however, when the feasible set fails to have a simple structure.  In such cases, optimizing a quadratic or linear function over the entire feasible set becomes prohibitive, and the ``simple is good'' dictum provides no clear path forward.

\hchange{Important machine learning examples where nonlinear constraints are key includes reinforcement learning, where autonomous agents are often required to plan trajectories that avoid obstacles and satisfy the laws of physics \citep{Frazzoli}. Obstacle avoidance constraints are generally nonconvex and cannot be easily handled with projections. Similarly, minimax problems that arise in generative adversarial networks or robust learning problems, for example, can be reformulated as constrained optimization problems, leading to semi-infinite and nonconvex constraints~\citep{AdversarialRobustness}. Machine learning applications in chemistry and physics often benefit from incorporating prior knowledge, for example in the form of symmetries and invariants. While much of the recent work has focused on reparametrizing convolutional layers in neural networks~\citep{SchNet,weiler20183d}, these symmetries and invariants are described by nonconvex and nonlinear constraints.}

When the structure of the feasible set fails to enable closed-form projections or closed-form solutions for Frank-Wolfe updates, optimization theorists often turn to interior point or sequential quadratic programming methods.  The idea of interior point methods is to reduce the constrained optimization problem to an unconstrained one by using barrier functions that assign a high cost to points close to the boundary of the feasible set.  In sequential quadratic programming, the underlying nonlinear problem is approximated by a series of quadratic programs.  While both classes of methods have been proposed for applications in machine learning~\citep[see, e.g.,][]{BoydL1,InteriorPointSupportVector,Zeilinger}, they are significantly more complex than the stochastic-gradient methods that have been so successful in unconstrained machine learning.  There remains a need for a learning-friendly approach to constrained optimization.

In the current paper, we present a class of first-order methods that are applicable to a wide range of problems in machine learning.  A notable simplification of these methods, relative to classical constrained optimization methods, including projection methods and Frank-Wolfe, is that our methods rely exclusively on \emph{local approximations of the feasible set}.  These local approximations are a natural generalization of Clarke's tangent cone and are well defined for feasible and infeasible points.  Moreover, as we will show, they make possible a key algorithmic simplification---they yield algorithms that converge even with a constant step size.  Technically, they handle the case when the iterates become infeasible.  This makes the resulting algorithmic procedure simple to implement and also ensures that the descent motion is not necessarily stopped as soon as a new constraint is violated.  Finally, while the entire feasible set might be described by a very large (or even infinite) number of \emph{nonlinear} constraints, the local approximation typically only includes a small number of \emph{linear} constraints, which substantially reduces the amount of computation required for a single iteration.

We believe that these simplifications make our approach a natural candidate for large-scale constrained machine-learning problems.  Our main goal in the current paper is to provide a theoretical foundation to support such a claim.  We also present results from a preliminary set of numerical experiments, which include, for example, randomly generated high-dimensional quadratic programs.  Comparing the new methods to the interior point solver CVXOPT of \citet{CVXOPT}, we find that the complexity of the new methods scales roughly with $n^2$ (where $n$ is the problem dimension), whereas the complexity of the interior point solver scales with $n^3$.  When $n$ is large, this may lead to speedups of several orders of magnitude.

As our discussion has hinted, while our methods are relatively simple to specify and deploy, their analysis brings new challenges.  Our treatment builds on recent progress in using continuous-time dynamical systems tools to analyze discrete-time algorithms in gradient-based optimization~\citep{SuAcc, WibisonoVariational, Diakonikolas2, KricheneAcc, Gui, Betancourt, ourWork, ourWork3, ourWork2}.  Much of the work in this vein is focused on understanding accelerated first-order optimization methods, such as Nesterov's algorithm, where the understanding arises by exposing links between differential and symplectic geometry, dynamical systems, and mechanics.  These links, which supply a mechanical interpretation of accelerated methods and provide a rigorous interpretation of concepts such as ``momentum,'' are often easiest to derive in continuous time, making use of variational, Hamiltonian, and control-theoretic perspectives.  Indeed, the most complex part of these analyses often arises in the conversion from continuous time to discrete time.

In line with this recent literature, our treatment of constrained optimization also straddles the boundary between continuous time and discrete time.  As in the unconstrained setting, the continuous case is relatively straightforward and the major challenges arise in the conversion to discrete time.  Indeed, the key novelty is that in our constrained setting, the discrete-time function that maps one iterate to the next is \emph{discontinuous}.  Thus, tools such as smooth Lyapunov functions or the theory of monotone operators that have been widely employed in the unconstrained setting are not applicable in our setting, and a new analysis framework is needed.  We develop such a framework by making use of ideas from \emph{non-smooth mechanics}.  Indeed, as we will discuss in the following section, the closest point of contact with existing literature is the notion of \emph{Moreau time-stepping} in non-smooth mechanics.

\textbf{Related work:}
In the following paragraphs we highlight some of the connections of our approach to the existing literature.  Due to the wealth of work on constrained optimization over the last several decades, a comprehensive summary seems out of reach.  We will therefore focus on ideas that are most closely related to our approach and refer to the textbooks of \citet{Bertsekas}, \citet{NesterovIntro}, \citet{Wright}, or \citet{LuenbergerBook} for a broader overview.

Our approach is in the spirit of projected gradient methodology.  The basic idea of the projected gradient method is to compute a step along the negative gradient of the objective function and to project the resulting point back to the feasible set~\citep[see, e.g.,][Ch.~2.3]{Bertsekas}.  From a theoretical point of view, the analysis of projected gradients strongly parallels that of unconstrained gradient descent.  Indeed, by generalizing the notion of the gradient to the ``gradient mapping''~\citep[][p.~86]{NesterovIntro}, arguments can be readily translated from the unconstrained to the constrained case.  More generally, projected gradients can be viewed as an instance of a proximal point algorithm~\citep{ProxAlg}, which itself can be elegantly described with the theory of monotone operators~\citep{ConvexAnalysisBauschke,Rockafellar}.

The key difference between our approach and classical projected gradients is that our approach is based on a local approximation of the feasible set.  This local approximation includes only the active constraints\footnote{We say that the $i$th constraint is \emph{active} at the iterate $x_k$ if $g_i(x_k)\leq 0$, where the smooth function $g:\mathbb{R}^n \rightarrow \mathbb{R}^{n_\text{g}}$ describes the feasible set as $\{ x\in \mathbb{R}^n~|~g(x)\geq 0\}$.  It is important to note that this definition of active constraints does not require the corresponding dual multipliers to be nonzero.} and is guaranteed to be a convex cone even if the underlying set is nonconvex.  Our approach can be viewed as an inexact projected gradient method, and as such has similarities to the work of \citet{inexactPG} and \citet{inexactPG2}.  However, in contrast to this work, we do not impose a monotone decrease of the cost function by an appropriate line search.  In fact, our approach converges even with a constant step size, whereby the objective function fails to be monotonically decreasing (in general).

While projected gradient approaches have been successfully applied in various machine learning problems~\citep[see, e.g.,][]{SignalProcessing,SVMwPG}, an even simpler algorithm---the Frank-Wolfe algorithm---has also received considerable attention in recent years~\citep{Jaeggi}.  At each iteration of the Frank-Wolfe algorithm, a feasible descent direction is computed by maximizing the inner product with the negative gradient.  This reduces to the minimization of a linear objective function over the feasible set, which, compared to projected gradients, can lead to considerable simplification.  The simplification is in accord with the ``simple is good'' dictum of machine learning, and indeed it has been found that the Frank-Wolfe algorithm provides a unified theoretical framework for many greedy machine learning algorithms, including support vector machines, online estimation of mixtures of probability densities, and boosting~\citep{Clarkson}.  Recent results extend the Frank-Wolfe algorithm to the stochastic setting~\citep{PfOnlineLearning,OneSampleStochasticFW}, or improve on its relatively slow convergence rate~\citep{BoostingFrankWolfe,FasterRates}.

\hchange{In some cases, constraints can be handled very effectively by mirror descent~\citep[Ch.~3]{ProblemComplexity}. The underlying idea of mirror descent is to introduce a non-Euclidean metric for adapting gradient descent to the specific type of objective function or the specific type of constraints at hand~\citep{BeckMirror}. While mirror descent relies on projections on the feasible set, the non-Euclidean metric can improve on problem-specific constants and lead to algorithms whose complexity scales mildly in the number of decision variables. A prominent example is the optimization of linear functions over the unit simplex, which has important applications in online learning and online decision making~\citep[see ][Ch.~5]{BubeckBandits}.}

As we have already discussed, alternatives to projected gradients, \hchange{mirror descent}, and Frank-Wolfe include interior point methods and sequential quadratic programming.  Interior point methods provide practical solutions to many problems in constrained optimization, and they are guaranteed to return approximate solutions to many convex nonlinear programming problems in polynomial time~\citep{NesterovInterior}.  They can be particularly efficient if the underlying Karush-Kuhn-Tucker system is sparse, which can be exploited for simplifying the Newton updates~\citep{Zeilinger}.  Similarly, in sequential quadratic programming, the underlying Karush-Kuhn-Tucker system resembles the Newton update of interior point methods.  There are many different flavors of sequential quadratic programming, depending on the type of line search, whether only approximate second order information is used, or whether equality constraints are eliminated.  An implementation that is widely used to solve complex optimal control and planning problems is presented in \citet{SNOPT}.  Recent advances in sequential quadratic programming share some similarity with our approach; see, for example, \citet{Torrisi} and \citet{Haeberle}.  Both of these methods involve linearizing both the active and inactive constraints.  The fact that all constraints are taken into account at each iteration enables the algorithms to anticipate constraint violations and distinguishes these approaches from the methods that will be discussed herein.

Finally, a main goal of the current paper is to bring to the fore an analogy between constrained optimization and non-smooth mechanics.  Indeed, from a certain point of view, finding stationary points of a constrained optimization problem is equivalent to computing equilibria of a corresponding non-smooth mechanical system.  The classical approach to simulating such systems is \emph{event-based integration}, which is a relatively complex algorithm that switches between smooth and non-smooth motion.  An alternative is Moreau time-stepping \citep{Moreau}, which is based on the discretization of a measure-differential inclusion that captures the smooth and non-smooth parts of the motion.  Moreau's algorithm can handle multiple (or even an infinite number of) discontinuities that may all happen within one time step. Further background can be found in the texts of \citet{Glocker} and \citet{Studer}.  Recent work in this area includes extensions to continuum mechanics \citep{Eugster} and higher-order integration schemes \citep{Acary}. 

Although we will exploit analogies to the simulation of physical systems, the focus of our theoretical analysis is in developing algorithms that efficiently compute approximate local minima of constrained nonlinear programming problems.  In this setting, it will be crucial to consider large time steps, to handle constraint violations (which are often ignored when simulating non-smooth mechanical systems), and to provide convergence guarantees in discrete time. \hchange{Our continuous-time analysis is also related to the theory of gradient inclusions, which are gradient flow dynamics on nonsmooth convex functionals~\citep[see, e.g.,][Ch.~3]{AubinDifferentialInclusions}. These have been extensively studied in the mathematical community due to their numerous applications, for example in the calculus of variations~\citep{Arrigo}. Our gradient flow formulation relies on local approximations of the feasible set. These approximations evolve over time, and as such, the dynamics can be viewed as a generalization of a sweeping process~\citep{MoreauSweeping}.}

Compared to classical treatments of constrained optimization, our treatment exhibits a key feature that arises directly from the physical analogy.  Rather than expressing constraints in the language of positions or configurations, as is standard in optimization, our constraints will be expressed in terms of velocities.  Thus, we will distinguish between constraints on the ``position level'' and constraints on the ``velocity level.''  Our focus on the latter will be seen to lead directly to a local, convex approximation of the feasible set.  By a constraint on velocity level, we mean a constraint on the forward increment $\lim_{\dt \downarrow 0} (x(t+\dt)-x(t))/\dt$ in continuous time or the difference $(x_{k+1}-x_k)/T$ in discrete time, where $T$ is the step size.  In continuous time, a given position constraint can (in most cases) be reformulated as an equivalent velocity constraint.  However, this equivalence breaks down in discrete time, which necessitates a careful analysis of the resulting discrete-time algorithms. We also note that there are (many) mechanical systems that have velocity constraints which cannot be formulated as position constraints. For example, while ice skater can move to any position in a skating rink, their velocity is constrained to lie parallel to the blades of the skates.

\textbf{Notation:} We follow standard notation from convex analysis. In particular, $\mathbb{R}$ denotes the real numbers, $\mathbb{R}_{\geq 0}$ the nonnegative real numbers, $\mathbb{R}_{\leq 0}$ the nonpositive real numbers, and $\mathbb{Z}$ the set of all integers. The notation $|\cdot|$ is reserved for the Euclidean norm or the cardinality of a set. The gradient of a function $h: \mathbb{R}^n \rightarrow \mathbb{R}^m$ is denoted by $\nabla h: \mathbb{R}^n \rightarrow \mathbb{R}^{n\times m}$ and the indicator function of the set $C$ is referred to as $\psi_C:\mathbb{R}^n \rightarrow \mathbb{R} \cup \{\infty\}$, that is, $\psi_C(x)$ takes the value zero for $x\in C$ and $\infty$ otherwise. The subgradient of a convex function $g: \mathbb{R}^n \rightarrow \mathbb{R}$ evaluated at $x\in \mathbb{R}^n$ is denoted by $\partial g(x)$ and is defined as the set $\{v\in \mathbb{R}^n~|~v\T (y-x) \leq g(y)-g(x), ~\forall y\in \mathbb{R}^n\}$. The tangent cone (in the sense of Clarke) at any point $x\in C$ is referred to as $T_C(x)$, that is, $\delta x\in T_C(x)$ if there exist two sequences $x_j\rightarrow x$, $x_j\in C$, $t_j\downarrow 0$, such that $(x_j-x)/t_j \rightarrow \delta x$. The corresponding normal cone is denoted by $N_C(x):=\{ \lambda\in \mathbb{R}^n ~|~\lambda\T \delta x \leq 0, \forall \delta x\in T_C(x) \}$. 
\hchange{We will consider trajectories $x: \mathbb{R}_{\geq 0} \rightarrow \mathbb{R}^n$ that are absolutely continuous and have a piecewise continuous derivative. Absolute continuity means that $x(t)-x(0)$ can be expressed as the Lebesgue integral over the velocity $\dot{x}$; that is, $x(t)=x(0)+\int_{0}^{t} \dot{x}(\tau) \diff \tau$ for all $t\geq 0$. The assumption that $\dot{x}$ is piecewise continuous means that on any finite interval, $\dot{x}$ is continuous except at a finite number of points, where left and right limits, denoted by $\dot{x}(t_0)^-$ and $\dot{x}(t_0)^+$, are well-defined. The value $\dot{x}(t_0)$ at the discontinuity $t_0$ is of no interest and may or may not exist. }
Finally, we use subscripts to denote both single components of a vector and the iteration number of a discrete algorithm. The distinction will be made from context (we usually reserve the subscript $k$ for the iteration number).
\section{Overview of the Results}\label{Sec:Summary}
We consider the following optimization problem:
\begin{equation}
\min_{x\in C} f(x),\quad \text{where}\quad C:=\{ x\in \mathbb{R}^n~|~ g(x)\geq 0,~h(x)=0\}, \label{eq:fundProb}
\end{equation}
\hchange{and where $f:\mathbb{R}^{n}\rightarrow \mathbb{R}$ defines the objective function. The functions $g:\mathbb{R}^{n} \rightarrow \mathbb{R}^{n_\text{g}}$ and $h:\mathbb{R}^{n} \rightarrow \mathbb{R}^{n_\text{h}}$ define the constraints, and $n$, $n_\text{g}$, and $n_\text{h}$ are positive integers. The functions $f$, $g$, and $h$ are continuously differentiable and have a Lipschitz continuous gradient. Moreover, $f$ is assumed to be such that $f(x) \rightarrow \infty$ for $|x|\rightarrow \infty$ and $C$ is assumed to be non-empty and bounded, which guarantees that the minimum in \eqref{eq:fundProb} is attained.}

\paragraph{Brief summary of the main contributions:} In mathematical optimization constraints are typically treated by direct reference to positions, meaning that $x_k$ or $x(t)$ are constrained to lie in $C$ for all $k\geq 0$ or all $t\geq 0$, respectively. We adopt a fundamentally different point of view---instead of constraining $x(t)$ or $x_k$, we constrain the forward velocity $\dot{x}(t)^+=\lim_{\diff t \downarrow 0} (x(t+\diff t)-x(t))/\diff t$ or forward increments $(x_{k+1}-x_k)/T$. At a given position $x\in \mathbb{R}^n$, the set of all admissible velocities will be denoted by $V_\alpha(x)\subset \mathbb{R}^n$. When $x\in C$, the set $V_\alpha(x)$ corresponds to the tangent cone of the set $C$ at $x$. We will introduce an appropriate generalization of $V_\alpha(x)$ in order to also capture cases in which $x\not \in C$. The two different point of views on constraints are illustrated in Figure~\ref{Fig:illustration1}.

In continuous time, the resulting velocity constraint is equivalent to the original position constraint, assuming constraint qualification. However, this equivalence breaks down in discrete time, and may lead to infeasible iterates over the course of the optimization. One of our main results is a guarantee that the resulting discrete algorithm nonetheless converges to stationary points, despite the possibility of infeasible iterates and despite the discontinuous nature of the map from $x_k$ to $x_{k+1}$. In addition to providing such a guarantee, we derive rates of convergence and we show that a formulation of constraints on the velocity level can lead to computational advantages. In particular, we show that at each iteration, only a \emph{linear} and \emph{convex} approximation of the original \emph{nonlinear} and \emph{nonconvex} feasible set needs to be considered. Moreover, the linear approximation includes only the constraints that are active at $x(t)$ or $x_k$. On randomly generated dense quadratic programs, for example, the complexity of the proposed method scales with $n^2$ (empirically), which contrasts with  state-of-the-art implementations of an interior-point method, which scale with $n^3$. Moreover, in many practical problems (for example, support vector machines) the proposed algorithm greatly reduces the number of constraints that must be considered at each iteration. 

\hchange{In summary, the purpose of this article is twofold: (i) we highlight that ``position" constraints can be reformulated as ``velocity" constraints, which leads to a new perspective on constrained optimization that has connections to non-smooth mechanics, and (ii) we exemplify and showcase this new point of view on gradient flow and gradient descent, which is arguably one of the simplest, but also most relevant, use cases for machine learning. This is done by providing formal convergence guarantees in continuous and in discrete time, deriving rates, and studying the behavior of the algorithms in numerical experiments.}

\begin{figure}
    \centering
    \scalebox{0.75}{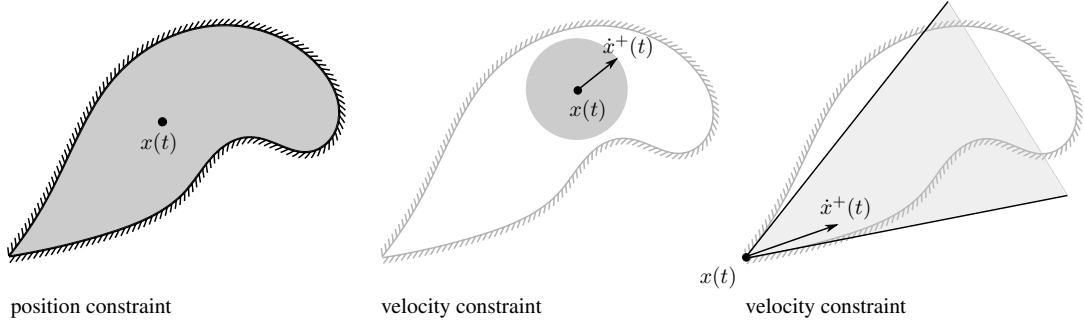}
    \caption{The figure contrasts position constraints with velocity constraints. The leftmost sketch illustrates the position constraint, where $x(t)$ is constrained to the feasible set as indicated by the shaded region. The center and right figures illustrate the induced constraints on the velocity $\dot{x}(t)^+$ (which will be precisely defined below). If $x(t)$ is in the interior of the feasible set, there are no restrictions on the forward velocity, as indicated with the shaded ball without border (center). The figure on the right illustrates the case where $x(t)$ lies on the boundary of the feasible set. As a result, $\dot{x}(t)^+$ is constrained to lie in the cone indicated by the shaded region. In the discrete-time case $\dot{x}(t)^+$ is replaced with $(x_{k+1}-x_k)/T$.}
    \label{Fig:illustration1}
\end{figure}

\paragraph{Detailed summary of the main contributions:} In order to discuss the results in greater detail, we introduce the following definition and assumption, which will hold throughout the remainder of the article.
\begin{definition} The point $x\in \mathbb{R}^n$ satisfies the \emph{Mangasarian-Fromovitz constraint qualification} if the columns of $\nabla h(x)$ are linearly independent and if there exists a vector $w\in \mathbb{R}^n$ such that $\nabla h(x)\T w =0$ and $\nabla g_{i}(x)\T w >0$ for all $i\in I_x$, where $I_x$ denotes the set of active inequality constraints at $x$, i.e., $I_x:=\{i\in \mathbb{Z}~|~g_i(x)\leq 0\}$.\footnote{We would like to emphasize that our definition of active constraints does not require constraints to have corresponding dual multipliers that are nonzero. }
\end{definition}
\begin{assumption}\label{Ass:MF} \textbf{\upshape{(standing)}}~~
The Mangasarian-Fromovitz constraint qualification is satisfied for all $x\in \mathbb{R}^n$.
\end{assumption}
From the definition of $T_C(x)$ it follows that every $\delta x\in T_C(x)$ satisfies $\nabla h(x)\delta x=0$ and $\nabla g_i(x)\T \delta x\geq 0$, for all $i\in I_{x}$. Assumption~\ref{Ass:MF} ensures that the converse is also true, which guarantees that all stationary points of \eqref{eq:fundProb} satisfy the corresponding Karush-Kuhn-Tucker conditions.
We further introduce the set
\begin{equation}
V_\alpha(x):=\{ v\in \mathbb{R}^n~|~\nabla h(x)\T v + \alpha h(x)=0,\quad \nabla g_{i}(x)\T v + \alpha g_i(x)\geq 0,~~\forall i\in I_x\}, \label{def:Va}
\end{equation}
where $\alpha\geq 0$ is a positive scalar. The role of $\alpha$ will be discussed below. As a result of the constraint qualification, the set $V_\alpha(x)$ reduces to the tangent cone $T_C(x)$ of the set $C$ for any $x\in C$. Moreover, for a fixed $x\in \mathbb{R}^n$, $V_\alpha(x)$ is a convex polyhedral set, involving only the active constraints $I_x$.

\hchange{For some of the results we will require convexity:}
\begin{assumption}\label{Ass:Conv}
\hchange{Let $C$ be convex and $f$ strongly convex with strong convexity constant $\mu>0$.}
\end{assumption}
\hchange{We will explicitly state when Assumption~\ref{Ass:Conv} will be needed.}

With the notation in place, we are ready to state our main results. We start with a general framework based on a continuous-time gradient flow which will be used as a starting point for our discrete algorithm. \hchange{The following proposition highlights that trajectories satisfying the continuous-time gradient flow converge to stationary points, even when the objective function $f$ or the set $C$ are nonconvex, or when the initial condition $x(0)$ is infeasible. The discrete algorithm that we investigate subsequently will be a simple Euler discretization of the continuous-time gradient flow.  Thus the analysis of the continuous-time flow will be important for understanding the algorithm's behavior for small step sizes in addition to its intrinsic interest.}

\begin{proposition} \label{Prop:GF} (constrained gradient flow)
Let $x:[0,\infty) \rightarrow \mathbb{R}^n$ be an absolutely continuous trajectory with a piecewise continuous derivative. Then, for any $x(0)\in C$, the following are equivalent:
\begin{align}
\dot{x}(t)&=- \nabla f(x(t)) + R(t), \quad -R(t) \in N_C(x(t)), ~\qquad \qquad \forall t\in [0,\infty)~\text{a.e.}, \label{eq:posLeveltmp}\\
\dot{x}(t)^+&=- \nabla f(x(t)) + R(t), \quad -R(t) \in \partial \psi_{V_\alpha(x(t))}(\dot{x}(t)^+), \quad \forall t\in [0,\infty),\label{eq:velLeveltmp}\\
\dot{x}(t)^+&= \argmin_{v \in V_\alpha(x(t))} \frac{1}{2} |v+\nabla f(x(t))|^2, \qquad \qquad \qquad \qquad \qquad~ \forall t\in [0,\infty), \label{eq:velLeveltmp2}
\end{align}
where $\dot{x}(t)^+$ denotes the right-hand derivative of $x$ at $t$. 

For any $x(0)\in \mathbb{R}^{n}$, \eqref{eq:velLeveltmp} and \eqref{eq:velLeveltmp2} are equivalent and lead to a unique trajectory $x(t)$ (if it exists) which is guaranteed to converge to the set of stationary points of \eqref{eq:fundProb} (for $\alpha>0$); that is, $x(t) \rightarrow C$ as $t\rightarrow \infty$ and
\begin{equation*}
\lim_{t\rightarrow \infty} |-\nabla f(x(t)) + R(t)| =0, 
\end{equation*} 
where $R(t)$ is defined in \eqref{eq:velLeveltmp}. Moreover, if the stationary points are isolated, the trajectory $x(t)$ converges to a single stationary point.

\hchange{If Assumption~\ref{Ass:Conv} (convexity) is satisfied} and $\alpha \leq 2\mu$, the trajectory satisfying \eqref{eq:velLeveltmp} and \eqref{eq:velLeveltmp2} converges exponentially:
\begin{equation}
(h(x(0)),\min\{0,g(x(0))\})\T \lambda^* e^{-\alpha t} \leq f(x(t))-f^*\leq (f(x(0))-f^*) e^{-2 \mu t},\label{eq:boundinfeas}
\end{equation}
for all $x(0)\in \mathbb{R}^n$, where $f^*$ is the value of the minimizer in \eqref{eq:fundProb} and $\lambda^*$ is a multiplier that satisfies the Karush-Kuhn-Tucker conditions \hchange{of \eqref{eq:fundProb}}.
\end{proposition}
We make the following remarks:
\begin{itemize}
\item 
\hchange{We note that the differential inclusion \eqref{eq:posLeveltmp} is also known as a \emph{gradient inclusion}, since its right-hand side amounts to the negative subgradient of $f+\psi_C$. It has been extensively studied in the mathematical community \citep[see, e.g.,][Ch.~3]{AubinDifferentialInclusions}. 
If $C$ is convex, the subgradient of $f+\psi_C$ is maximally montone \citep[see, e.g.,][p.~354, Theorem~20.25]{ConvexAnalysisBauschke} and as a result, existence and uniqueness of absolutely continuous trajectories satisfying \eqref{eq:posLeveltmp} can be guaranteed \citep[see, e.g.,] [Theorem~10.3.1, p. 399]{AubinSetValued}.} 
\item The additional assumptions on $\dot{x}$ are used for establishing the equivalence between \eqref{eq:posLeveltmp} and \eqref{eq:velLeveltmp}. Convergence results for \eqref{eq:velLeveltmp} and \eqref{eq:velLeveltmp2} similar to those of Proposition~\ref{Prop:GF} can still be obtained when the restrictions on $\dot{x}$ are relaxed. We also note that by applying the theory of \citet{Filippov}, \eqref{eq:velLeveltmp} and \eqref{eq:velLeveltmp2} can be extended to a differential inclusion that is guaranteed to have an absolutely continuous solution. We refer the reader who is interested in existence results to the work of \citet{Filippov} and \citet{AubinDifferentialInclusions}. The equivalence between \eqref{eq:posLeveltmp} and \eqref{eq:velLeveltmp} under weaker assumptions on $\dot{x}$ is discussed in \citet{Brogliato}, which also provides a short existence proof (requiring, however, that $C$ is convex).
\item The variable $R(t)$ in \eqref{eq:posLeveltmp} can be regarded as a reaction force that imposes the constraint $x(t) \in C$ for all $t\in [0,\infty)$ (by definition, the normal cone is empty if $x(t)\not \in C$). We therefore say that \eqref{eq:posLeveltmp} includes the constraint on the position level. In contrast, the reaction force $R(t)$ in \eqref{eq:velLeveltmp} enforces $\dot{x}(t)^+ \in V_{\alpha}(x(t))$ for all $t\in [0,\infty)$, which reduces to $\dot{x}(t)^+ \in T_C(x(t))$ for $x(t)\in C$. The condition $\dot{x}(t)^+ \in V_{\alpha}(x(t))$ can be viewed as an extension of $\dot{x}(t)^+ \in T_C(x(t))$ to allow also for $x(t)\not\in C$. Interpreting \eqref{eq:velLeveltmp} as a stationarity condition for $\dot{x}(t)^+$ yields \eqref{eq:velLeveltmp2}. We therefore say that \eqref{eq:velLeveltmp} and \eqref{eq:velLeveltmp2} impose the constraints on the velocity level.
\item The intuition behind the equivalence of \eqref{eq:posLeveltmp}, \eqref{eq:velLeveltmp}, and \eqref{eq:velLeveltmp2} can be summarized in the following way. For an absolutely continuous trajectory $x(t)$, the constraint $x(t)\in C$ for all $t\in [0,\infty)$ is equivalent to $\dot{x}(t)^+ \in V_\alpha(x(t))$ for all $t\in [0,\infty)$, $x(0)\in C$~\citep[Remark 2.5]{Moreau}.\footnote{Constraint qualification is needed for the equivalence to hold.} If we think of $x(t)$ as the position of a point mass, and $\dot{x}(t)^+$ as its velocity, this can be stated as follows: A constraint on the position of the point mass induces a constraint on its velocity. Conversely, the constraint $\dot{x}(t)^+\in V_\alpha(x(t))$ on the velocity ensures that the position constraint is satisfied for all times $t\geq 0$, provided that $x(0)\in C$.
\item The reformulation \eqref{eq:velLeveltmp2} emphasizes that at each point in time, the velocity is chosen to match unconstrained gradient flow as closely as possible, subject to the velocity constraint $\dot{x}(t)^+\in V_\alpha(x(t))$.
This can be seen as an analogue of the principle of least constraint in mechanics \citep[Ch.~9]{Glocker}.
\item 
Imposing $\dot{x}(t)^+\in V_\alpha(x(t))$ for all $t\in [0,\infty)$, yields, by definition of the set $V_\alpha(x)$ and by applying Gr\"onwall's inequality,
\begin{equation}
g_i(x(t))\geq g_i(x(0)) e^{-\alpha t}, \quad i\in I_{x(0)}, \quad h(x(t))=h(x(0)) e^{-\alpha t}, \label{eq:evolConstr}
\end{equation}
for all $t\in [0,\infty)$. Consequently, the constant $\alpha$ controls how quickly the constraint violations decay. \hchange{We note that there are two competing objectives: reducing the objective function and converging to the feasible set. The variable $\alpha$ controls the tradeoff between these two objectives; for small $\alpha$, the emphasis is on reducing the objective function, for large $\alpha$, the emphasis is on converging to the feasible set. We will also see that in discrete time, $\alpha$ is required to satisfy $\alpha T \leq 1$ for guaranteeing convergence ($T$ is the step size).}
\item By reformulating the constraint on the velocity level as in \eqref{eq:velLeveltmp} and \eqref{eq:velLeveltmp2}, the velocity $\dot{x}(t)^+$ can by computed by relying on a local and linear approximation of the set $C$ via $\dot{x}(t)^+\in V_\alpha(x(t))$, which includes only the active constraints $I_{x(t)}$. Hence, even for a nonconvex optimization problem such as \eqref{eq:fundProb}, the optimization given by the right-hand side of \eqref{eq:velLeveltmp2} is convex. 
\end{itemize}

By replacing $x(t)^+$ with $(x_{k+1}-x_{k})/T$ and $x(t)$ with $x_k$ in \eqref{eq:velLeveltmp} or \eqref{eq:velLeveltmp2}, we obtain the following discrete algorithm:
\begin{equation}
x_{k+1}=x_{k} - T \nabla f(x_k) + T R_k, \quad -R_k \in \partial \psi_{V_\alpha(x_k)}((x_{k+1}-x_k)/T), \quad k=0,1,2,\dots, \label{eq:disAlgorithm}
\end{equation}
which for any $x_0 \in \mathbb{R}^n$, leads to well-defined (unique) iterates, as long as the Mangasarian-Fromovitz constraint qualification is satisfied for all $x\in \mathbb{R}^n$. As in the continuous-time setting, the discrete algorithm relies on a local approximation of the feasible set at each iteration, which includes only the active constraints $I_{x_k}$. Projections or optimization over the entire feasible set $C$ (at each iteration) are therefore avoided. While this reduces computation, it also complicates the analysis. 

It is important to note that \eqref{eq:disAlgorithm} can be reformulated in a number of equivalent ways.  The choice made in Algorithm~\ref{Alg:one} is particularly suitable for numerical implementation.

\begin{algorithm}
\begin{algorithmic}
\Require $x_0 \in \mathbb{R}^n$, TOL, MAXITER, $T>0$, $\alpha T \in (0,1]$
\State $k=0$
\While{$k<\text{MAXITER}$}
\State Determine the set of closed constraints $I_{x_k}$ 
\vspace{-2pt}\State Define $W_k:=(\nabla h(x_k), \nabla g_i(x_k)_{i\in I_{x_k}})$ and $D_k:=\mathbb{R}^{n_\text{h}} \times \mathbb{R}_{\geq 0}^{|I_{x_k}|}$
\State Define $\bar{g}_k:=(h(x_k),g_i(x_k)_{i\in I_{x_k}})$
\State Find $\lambda_k\in D_k$ such that $-\lambda_k\in \partial \psi_{D_k}(W_k\T W_k \lambda_k - W_k\T \nabla f(x_k) + \alpha \bar{g}_k)$ (see Section~\ref{Sec:FixedPoint}, \eqref{eq:statdual3})
\State Perform the update $x_{k+1} = x_k - T\nabla f(x_k) + T W_k \lambda_k$
\If{$|x_{k+1}-x_k| \leq T\cdot \text{TOL}$} 
\State \Return $x_{k+1}$
\EndIf
\State $k\leftarrow k+1$
\EndWhile
\end{algorithmic}
\caption{Implementation of the gradient descent scheme \eqref{eq:disAlgorithm}.}
\label{Alg:one}
\end{algorithm}

The following definitions will be useful for characterizing the behavior and the convergence rate of \eqref{eq:disAlgorithm}. We start by introducing the function $v:\mathbb{R}^n \rightarrow \mathbb{R}^n$, which assigns the velocity $v(x)$ to each $x\in \mathbb{R}^n$:
\begin{equation}
v(x):=\argmin_{v\in V_\alpha(x)} \frac{1}{2} |v+\nabla f(x)|^2. \label{eq:vdef}
\end{equation}
Clearly, in continuous time, \eqref{eq:velLeveltmp} and \eqref{eq:velLeveltmp2} evolve as $\dot{x}(t)^+=v(x(t))$, whereas in discrete time, \eqref{eq:disAlgorithm} imposes $(x_{k+1}-x_k)/T=v(x_k)$. As a result of the constraint qualification, strong duality holds \hchange{(see Lemma~\ref{Lemma:SD} in Section~\ref{Sec:Continuous-Time} for details)} and we obtain the following dual \hchange{(the dual corresponds to \eqref{eq:vdef} and is not directly related to \eqref{eq:fundProb})}:
\begin{equation}
d(x):=\max_{\lambda \in D_x} l(x,\lambda) - \frac{1}{2 \alpha} |\nabla_x l(x,\lambda)|^2,  \label{eq:dual2}
\end{equation}
where $\nabla_x$ denotes the gradient with respect to $x$, and the Lagrangian $l: \mathbb{R}^n \times (\mathbb{R}^{n_\text{h}} \times \mathbb{R}^{n_\text{g}}_{\geq 0}) \rightarrow \mathbb{R}$ is defined as
\begin{equation}
l(x,\lambda):= f(x)- \lambda\T \bar{g}(x), \label{eq:Lag}
\end{equation}
with $\bar{g}(x):=(h(x),g(x))$. The set $D_x$ in \eqref{eq:dual2} is given by 
\begin{equation*}
D_x:=\{ \lambda \in \mathbb{R}^{n_\text{h}} \times \mathbb{R}^{n_\text{g}}_{\geq 0} ~|~\lambda_{n_\text{h}+i}=0, \forall i\not \in I_x\},
\end{equation*}
and includes only multipliers $\lambda_i \neq 0$ that correspond to equality constraints or active inequality constraints, defined by $i\in I_x$. The multipliers $\lambda_{n_\text{h}+i}$, which correspond to inactive inequality constraints, i.e., $i\not\in I_x$, are set to zero, and can therefore be eliminated from the outset when solving \eqref{eq:dual2} (as is done in Algorithm~\ref{Alg:one}). In general, there might be multiple $\lambda \in D_x$ that attain the maximum in \eqref{eq:dual2}. We will denote any one of them by $\lambda(x)$. As a consequence of Lagrange duality, $\lambda(x)$ is related to the minimizer of \eqref{eq:vdef} by
\begin{equation}
v(x)=-\nabla_x l(x,\lambda(x))=-\nabla f(x) + \nabla\bar{g}(x)\lambda(x).
\end{equation}
We note that the variable $R(t)$ in \eqref{eq:velLeveltmp} or $R_k$ in \eqref{eq:disAlgorithm} can therefore be expressed as $\nabla \bar{g}(x(t)) \lambda(x(t))$ and $\nabla \bar{g}(x_k) \lambda(x_k)$, respectively. 

\hchange{In general, the multipliers $\lambda(x)$ that result from \eqref{eq:dual2} are different than the multipliers $\lambda^*$ that arise from the Karush-Kuhn-Tucker conditions of \eqref{eq:fundProb} and only agree when $x=x^*$.}

\hchange{The function $d$ as defined in \eqref{eq:dual2} will be important for the analysis of \eqref{eq:disAlgorithm} and it will be shown that under suitable assumptions, $d(x_k)$ is monotonically increasing in $k$ and converges to $f^*$. Moreover, if Assumption~\ref{Ass:Conv} (convexity) holds, $f^*$ is an upper bound on $d(x)$ and $f^*-d(x)$ bounds $|x-x^*|^2$, the distance of $x$ to the optimizer of \eqref{eq:fundProb}. A proof of this fact is included in Appendix~\ref{App:d} along with other properties of $d$. This makes $d$ a natural choice for evaluating the progress of \eqref{eq:disAlgorithm}. We note that neither $f(x_k)$ nor $\bar{g}(x_k)$ alone are suitable, since these are not monotonic in $k$. Indeed, since we allow for infeasible iterates, $f(x_k)$ might increase over the course of the optimization and different constraints in $\bar{g}(x_k)$ might turn on and off.}

The maximum curvature of $f$ (the Lipschitz constant of $\nabla f$) limits the maximum admissible step size of gradient descent in the unconstrained case. We will see that the maximum curvature of $l(\cdot,\lambda)$ (for a fixed $\lambda$) will play a similar role for \eqref{eq:disAlgorithm}. \hchange{We denote by $\bar{\mu}_l(\lambda)$ and $\bar{L}_l(\lambda)$ the strong convexity and smoothness constant of $l(\cdot,\lambda): \mathbb{R}^n \rightarrow \mathbb{R}^n$ (for a fixed $\lambda$).} In case $C$ is convex and $f$ is strongly convex, the strong convexity constant $\mu$ of $f$ is a natural lower bound for $\bar{\mu}(\lambda)$, $\lambda \in \mathbb{R}^{n_\text{h}}\times \mathbb{R}^{n_\text{g}}_{\geq 0}$, which is attained for $\lambda=0$.

With this notation in place, we are now ready to state the main results that characterize \eqref{eq:disAlgorithm}.

\begin{proposition}\label{Prop:GD} (constrained gradient descent)
\hchangeII{Let Assumption~\ref{Ass:Conv} (convexity) be satisfied.} Then, for any $x_0\in \mathbb{R}^n$, the iterates $x_k$ of \eqref{eq:disAlgorithm} are well-defined (unique) and guaranteed to converge to the minimizer of \eqref{eq:fundProb} for 
\begin{equation*}
    T \leq \frac{2}{L_l+\mu}, \quad \alpha < \mu,
\end{equation*}
where $L_l$ is such that $L_l\geq \bar{L}_l(\lambda(x))$ for all $x\in \mathbb{R}^n$.\footnote{In case $g$ is affine, $L_l=L$, where $L$ is the smoothness constant of $f$.} The sequence $d(x_k)$ is monotonically increasing in $k$ and converges to $f^*$.

The velocity $(x_{k+1}-x_k)/T$ converges with
\begin{align*}
    \min_{j\in \{0,1,\dots,k\}} |-\nabla f(x_j)+R_j|^2 \leq \frac{f^*-d(x_0)}{c_1 (k+1)}, \quad \forall k\geq 0, \quad \forall x_0\in \mathbb{R}^n,
\end{align*}
where $c_1=T(\mu/\alpha -1) (1-\mu T/2)>0$ is constant, and for every $x_0 \in \mathbb{R}^n$ there exists a constant $N$ large enough such that
\begin{equation*}
    |-\nabla f(x_k) + R_k|^2\leq \frac{2}{c_1} (1-c_2 T)^{k-N} (f^*-d(x_{N})),
\end{equation*}
\hchange{where $c_2=2 \alpha (1-\mu T/2) (\mu-\alpha)/(L_l-\alpha)>0$ is constant.} \hchange{Similar bounds hold for the iterates $x_k$, that is,}
\begin{equation*}
    \min_{j\in \{0,1,\dots,k\}} |x^*-x_j|^2 \leq \frac{L_l/\alpha -1}{c_1 (\mu-\alpha)} ~  \frac{f^*-d(x_0)}{k+1}, \quad \forall k\geq N,
\end{equation*}
\hchange{and}
\begin{equation*}
    |x^*-x_k|^2 \leq \frac{2 (L_\text{l}/\alpha -1)}{c_1 (\mu-\alpha)} (1-c_2 T)^{k-N} (f^*-d(x_N)), \quad \forall k\geq N.
\end{equation*}
\end{proposition}
The following remarks are important:
\begin{itemize}
\item Algorithm \eqref{eq:disAlgorithm} does not anticipate any constraints that could potentially be violated at future iterations \hchangeII{(since $V_\alpha(x)$ is a local approximation of $C$, see Figure~\ref{Fig:illustration1}, that involves only $i\in I_{x}$)}. Unlike in the continuous-time case, where constraint violations decrease exponentially over time (see \eqref{eq:evolConstr}), a constraint may therefore open up, and close again a few iterations later. Nevertheless, the algorithm is guaranteed to converge at nearly a linear rate, which we find remarkable.
\item The convergence rate is dimension independent, which distinguishes the algorithm from interior-point methods, for example, where $\mathcal{O}(\sqrt{n_\text{g}})$ Newton-iterations are required to decrease the value of the objective function by a constant factor.
\item In the important special case where constraints are affine, all the above results hold for $L_l=L$, where $L$ is the smoothness constant of $f$. The constant $L_l$ is related to the maximum curvature of the Lagrangian, which seems a natural generalization from the unconstrained to the constrained case. 
\item Another important special case is given for a single nonlinear inequality constraint ($n_\text{g}=1$, $n_\text{h}=0$). We then obtain
\begin{equation*}
\lambda(x)=\begin{cases} \frac{\nabla g(x)\T \nabla f(x)-\alpha g(x)}{|\nabla g(x)|^2} &\text{for}~g(x)\leq 0,~~\nabla g(x)\T \nabla f(x)-\alpha g(x)\geq 0,\\
0 &\text{otherwise}.\end{cases}
\end{equation*}
In this case, the constant $L_l$ is given by the largest eigenvalue of the Hessian $\diff^2 l/\diff x^2=\diff^2 f/\diff x^2 - \lambda(x)~\diff^2 g/\diff x^2$ over all $x\in \mathbb{R}^n$.
\item The restriction $\alpha < \mu$ on the constant $\alpha$ is likely to be conservative. We observed in numerical experiments that a choice $\alpha T$ close to unity yields faster convergence. The restriction $\alpha T\leq 1$ is, however, necessary for convergence.
\item The convergence analysis will point to immediate extensions and variants of \eqref{eq:disAlgorithm}, which include line-search strategies, or alternations between gradient updates of the Lagrangian with fixed multipliers (which are computationally inexpensive) and updates of the multipliers according to \eqref{eq:dual2}. These extensions will be discussed in Section~\ref{Sec:Extension}.
\item \hchange{The results show that velocity and position converge at an exponential rate for large $k$. In fact, more detailed exponential convergence results that also apply to small $k$ will be derived in Section~\ref{Sec:Discrete}. However, as a result of the discontinuities in \eqref{eq:disAlgorithm} (constraints are not anticipated), these are more complicated to state and will require additional notation.}

\item \hchange{When choosing $\alpha=\mu/2$, for example, the exponential rate of convergence, $c_2 T$, scales with $(\mu/L_{l})^2$ for large $L_{l}/\mu$. This is in contrast to projected gradient descent, where the rate scales with $\mu/L$ for large $L/\mu$, where $L$ is the smoothness constant of $f$. This is, however, an artefact of the analysis and a simple argument (see Appendix~\ref{App:AsymptoticRate}) provides a tighter asymptotic rate of convergence of \eqref{eq:disAlgorithm}, which scales in fact with $\mu/L_{l}$ for large $L_{l}/\mu$.}

\item \hchange{By following the analysis of Appendix~\ref{App:ProofLambda}, we conclude that for a single strongly concave constraint function $g$, the constant $L_{l}$ is of the form $L (1+L_g/\mu_g~\text{const})$, where $\mu_g$ denotes the strong convexity constant of $-g$, $L_g$ the smoothness constant of $-g$, and $\text{const}$ is independent of $L, \mu, \mu_g$, and $L_g$. This means that $L_{l}$ is affected by the ratio $L_g/\mu_g$,} \hchangeII{which matches our intuition, since for larger $L_g/\mu_g$, the approximation quality of our local, sparse convex approximations of the feasible set deteriorates.}
\end{itemize}

The remainder of the article is concerned with proving Proposition~\ref{Prop:GF} and Proposition~\ref{Prop:GD}, providing context for both algorithms, discussing a particular implementation of Algorithm~\ref{Alg:one}, and illustrating the algorithms with numerical examples.
\section{The Continuous-Time Case}\label{Sec:Continuous-Time}
\hchange{The following section is concerned with proving Proposition~\ref{Prop:GF}. This will be done in several smaller steps, which are presented in the following subsections. Each part will be important to understand the continuous-time gradient flow \eqref{eq:velLeveltmp2} and its discrete-time counterpart \eqref{eq:disAlgorithm}.}

\subsection{Equivalences between position and velocity constraints in continuous time}
As mentioned in Section~\ref{Sec:Summary}, the constraint $x(t)\in C$ for all $t\in [0,\infty)$ can be reformulated as a constraint on the velocity, i.e., $\dot{x}(t)^+\in V_\alpha(x(t))$. This forms the basis for the equivalence between \eqref{eq:posLeveltmp} and \eqref{eq:velLeveltmp}:
\begin{proposition}\label{Prop:velLevel} (Similar to \citet[Prop.~5.1]{Moreau}, \citet[Ch.~7]{Glocker}) Let $x: [0,\infty) \rightarrow \mathbb{R}^n$, $x(0) \in C$, be an absolutely continuous trajectory that has a piecewise continuous derivative. Then, $x(t)$ satisfies \eqref{eq:posLeveltmp} if and only if it satisfies \eqref{eq:velLeveltmp}:
\begin{equation*}
\dot{x}(t)^+=-\nabla f(x(t)) + R(t), \quad -R(t)\in \partial \psi_{V_\alpha(x(t))}(\dot{x}(t)^+), \quad \forall t\in [0,\infty).
\end{equation*}
\end{proposition}
\begin{proof}
The proof is adapted from \citet[Prop.~5.1]{Moreau}. We start by assuming that $x(t)$ satisfies \eqref{eq:velLeveltmp}. The fact that the subdifferential of the indicator function is non-empty implies that $\dot{x}(t)^+\in V_\alpha(x(t))$ for all $t\in [0,\infty)$. Combined with $x(0)\in C$, we therefore have $x(t)\in C$ for all $t\in [0,\infty)$, and $V_\alpha(x(t))=T_C(x(t))$. \hchange{This follows by contradiction: Let $t$ be such that $\bar{g}_i(x(t))<0$. Then, by continuity of $\bar{g}_i(x(t))$ and the fact that $x(0)\in C$, there exists $0\leq t_0 < t$ such that $\bar{g}_i(x(t_0))=0$ and $\bar{g}_i(x(s))<0$ for all $s\in (t_0,t]$. This also means that}
\begin{equation*}
    \bar{g}_i(x(t))=\int_{t_0}^{t} \nabla \bar{g}_i(x(s))\T \dot{x}(s)^+ \text{d} s \geq 0,
\end{equation*}
\hchange{since, by virtue of $\dot{x}(s)^+ \in V_\alpha(x(s))$ for all $s\in [t_0,t]$, the integrand is guaranteed to be non-negative. This leads to the desired contradiction and ensures that $x(t)\in C$ for all $t\in [0,\infty)$.}
In addition, it follows from the definition of the subdifferential that $-R(t)\T (v-\dot{x}(t)^+) \leq 0$ for all $v\in T_C(x(t))$. Due to the fact that $T_C(x(t))$ is a cone, this implies $-R(t)\T v \leq 0$ for all $v\in T_C(x(t))$ (otherwise we could derive a contradiction by scaling an appropriate $v\in T_C(x(t))$), or in other words, $-R(t)\in N_C(x(t))$. This shows that any $x(t)$ with $x(0)\in C$ satisfying \eqref{eq:velLeveltmp} also satisfies \eqref{eq:posLeveltmp}.

In order to show the converse we start by assuming that $x(t)$ satisfies \eqref{eq:posLeveltmp}. We consider any interval $(t_0,t_1)$ where $\dot{x}(t)$ is continuous. By definition of the tangent cone, we have $\lim_{\dt \rightarrow 0} (x(t+\dt)-x(t))/\dt =\dot{x}(t) \in T_C(x(t))$ and $\lim_{\dt \rightarrow 0} (x(t-\dt)-x(t))=-\dot{x}(t) \in T_C(x(t))$ for all $t\in (t_0,t_1)$. Thus, from $-R(t)\in N_C(x(t))$ it follows that $-R(t)\T \dot{x}(t)\leq 0$ and $R(t)\T \dot{x}(t)\leq 0$, which implies that $-R(t)\T \dot{x}(t)=0$ for all $t\in (t_0,t_1)$. In addition, by definition of the normal cone, it follows that $-R(t)\T v \leq 0$ for all $v\in T_C(x(t))$. Combining these two facts results in $-R(t)\T (v-\dot{x}(t)) \leq 0$ for all $v\in T_C(x(t))$ and all $t\in (t_0,t_1)$. Hence, $-R(t)\in \partial \psi_{T_C(x(t))}(\dot{x}(t))$ for all $t\in (t_0,t_1)$, which implies \eqref{eq:velLeveltmp} for any time interval where $\dot{x}(t)$ is continuous. By taking the right-limit $t \downarrow t_0$, we conclude that $-R(t_0)^+\in\partial \psi_{T_C(x(t_0))}(\dot{x}(t_0)^+)$, $\dot{x}(t_0)^+ = -\nabla f(x(t_0)) + R(t_0)^+$, since $x(t)$ is continuous. Thus, \eqref{eq:velLeveltmp} holds for $t=t_0$, and therefore also at any other time instant where $\dot{x}(t)$ is discontinuous.
\end{proof}
Three important points are worth mentioning:
\begin{itemize}
\item The piecewise continuity assumptions on $\dot{x}$ are only used for showing that \eqref{eq:posLeveltmp} implies \eqref{eq:velLeveltmp}; absolute continuity of $x$ is enough for the converse to hold (provided the constraint qualifications are satisfied). \item When the solution $x(t)$ slides along the boundary of the constraint ($\dot{x}(t)$ is continuous), the reaction force is necessarily orthogonal to the velocity. From the point of view of classical mechanics, this means that the constraint reaction forces are passive and do not exert any power (at almost every time instant). This directly implies that the function $f(x(t))$ necessarily decreases along the trajectories of \eqref{eq:posLeveltmp}.
\item The condition \eqref{eq:velLeveltmp} describes the forward evolution of $x(t)$ by prescribing the right-hand derivative of $\dot{x}$ at each point in time. An equivalent formulation for the backwards evolution also exists. We will concentrate on the forward evolution, since we are interested in minimizing $f$.
\end{itemize}

The above proposition proves the equivalence between \eqref{eq:posLeveltmp} and \eqref{eq:velLeveltmp} as stated in Proposition~\ref{Prop:GF}. \hchange{For proving the convergence results of Proposition~\ref{Prop:GF}, the following intermediate steps will be useful.}

\subsection{Intermediate results}
\hchange{The first result establishes strong duality between \eqref{eq:vdef} and \eqref{eq:dual2} and summarizes the stationarity condition of \eqref{eq:dual2}, while the second result points to an important property of the multipliers $\lambda(t)$.}

\begin{lemma}\label{Lemma:SD}
\hchange{Strong duality holds for \eqref{eq:vdef}. For $\alpha\geq 0$, the dual of \eqref{eq:vdef} can be restated as}
\begin{equation}
    \max_{\lambda \in D_{x}} -\frac{1}{2} |\nabla \bar{g} (x) \lambda -\nabla f(x)|^2 - \alpha \lambda\T \bar{g}(x), \label{eq:dual}
\end{equation}
\hchange{and as a result, $\lambda(x)$ satisfies the following stationarity conditions}:
\begin{equation}
\nabla \bar{g}(x)\T \nabla \bar{g}(x) \lambda(x) - \nabla \bar{g}(x)\T \nabla f(x) + \alpha \bar{g}(x) \in \partial \psi_{D_{x}}(\lambda(x)). \label{eq:statDual}
\end{equation}
\end{lemma}
\begin{proof}
\hchange{We start by showing that Slater's condition holds for $V_\alpha(x)$ as a consequence of the constraint qualification; i.e., for any $x\in \mathbb{R}^n$, there exists a $v\in \mathbb{R}^n$ such that $\nabla h(x)\T v+ \alpha h(x)=0$ and $\nabla g_{i}(x)\T v + \alpha g_{i}(x)>0$ for all $i\in I_x$. }

We pick a $\bar{v}\in \mathbb{R}^n$ such that $\nabla h(x)\T \bar{v}=-\alpha h(x)$. Due to the fact that the columns of $\nabla h(x)$ are linearly independent (see Assumption~\ref{Ass:MF}), such a $\bar{v}$ exists. Thus, for a sufficiently large constant $\xi>0$, we have that $\nabla h(x)\T (\bar{v}+\xi w) =-\alpha h(x)$, $\nabla g_i(x)\T (\bar{v}+ \xi w) > -\alpha g_{i}(x)$ for all $i\in I_x$, where $w\in \mathbb{R}^n$ satisfies $\nabla h(x)\T w = 0$ and $\nabla g_i(x)\T w > 0$ for all $i\in I_x$. By assumption (see Assumption~\ref{Ass:MF}) such a $w$ exists. Thus, $v=\bar{v}+ \xi w$ satisfies the required conditions.

\hchange{Strong duality follows from the fact that \eqref{eq:vdef} is convex and Slater's condition holds. The rest is immediate.}
\end{proof}
%
%
\hchange{The following lemma establishes that $\lambda(t)$ is a feasible candidate for the dual \eqref{eq:dual2} (or \eqref{eq:dual}) at time $t_0>t$ provided that $t$ is close enough to $t_0$. A similar discrete-time result will be derived in Section~\ref{Sec:Discrete}. These result are fundamental for the convergence analysis of our algorithms.}
\begin{lemma}\label{Lemma:cont1}
Let the assumptions of Proposition~\ref{Prop:GF} be satisfied. Then, for every $t_0>0$, there exists $\delta>0$ such that $\lambda(t) \in D_{x(t_0)}$ for all $t\in (t_0-\delta,t_0)$.
\end{lemma}
\begin{proof}
We fix $t_0>0$ and consider the set of inequality constraints that are inactive at $t_0$; that is, $g_i(x(t_0))>0$. Due to the continuity of $x$ and $g$ there exists an interval $ (t_0-\delta,t_0)$, where $\delta >0$ is small enough, such that $g_i(x(t))>0$ for all $t\in (t_0-\delta,t_0)$ and for all $i\not\in I_{x(t_0)}$. As a result, $I_{x(t)} \subset I_{x(t_0)}$ for all $t\in (t_0-\delta,t_0)$ and the result follows.
\end{proof}

\subsection{Convergence results}
\hchange{The following section provides the remaining statements of Proposition~\ref{Prop:GF}; i.e., showing the equivalence between \eqref{eq:velLeveltmp} and \eqref{eq:velLeveltmp2}, showing that the solutions of \eqref{eq:posLeveltmp}, \eqref{eq:velLeveltmp}, and \eqref{eq:velLeveltmp2} converge to stationary points of \eqref{eq:fundProb} and deriving convergence rates if $C$ is convex and $f$ is strongly convex.}

\begin{claim}\label{Claim:CTconv}
Let the assumptions of Proposition~\ref{Prop:GF} be satisfied. For any $x(0)\in \mathbb{R}^{n}$, \eqref{eq:velLeveltmp} and \eqref{eq:velLeveltmp2} are equivalent and lead to a unique trajectory $x(t)$, which is guaranteed to converge to the set of stationary points of \eqref{eq:fundProb} (for $\alpha>0$). Moreover, if the stationary points are isolated, the trajectory $x(t)$ converges to a single stationary point.
\end{claim}
\begin{proof}
The equivalence between \eqref{eq:velLeveltmp} and \eqref{eq:velLeveltmp2} follows from the fact that \eqref{eq:velLeveltmp} corresponds to the stationarity condition of \eqref{eq:velLeveltmp2}, which, by strong convexity and non-emptiness of $V_\alpha(x(t))$, uniquely defines $\dot{x}(t)^+$ for each $t\in (0,\infty)$. This implies that $x(t)$ is unique.

We argue next that $x(t)\rightarrow C$ for $t\rightarrow \infty$, and that, as a result, $x(t)$ and $\lambda(t)$ are bounded. According to \eqref{eq:evolConstr}, the constraint violations at time $t$ can be bounded by 
$g_i(x(t))\geq g_i(x(0)) e^{-\alpha t}$ for all $i\in I_{x(0)}$ and $|h(x(t))| \leq |h(x(0))| e^{-\alpha t}$. We therefore conclude that $x(t) \rightarrow C$ for $t\rightarrow \infty$. The fact that $C$ is bounded and $x$ is continuous implies that $x(t)$ is bounded for all $t\geq 0$. As a result, there exist bounded dual variables $\lambda(t)$ satisfying \eqref{eq:dual}.

The stationarity condition \eqref{eq:statDual} implies that
\begin{equation*}
\lambda(t)\T \nabla \bar{g}(x(t))\T \left[ \nabla \bar{g}(x(t)) \lambda(t) - \nabla f(x(t))\right] + \alpha \lambda(t)\T \bar{g}(x(t)) = 0,
\end{equation*}
due to complementary slackness. This can be restated as $-R(t)\T\dot{x}(t)^+=\alpha \lambda(t)\T \bar{g}(x(t))$, which, in view of \eqref{eq:velLeveltmp}, yields
\begin{equation}
\frac{\diff}{\dt} f(x(t))^+ = -|\dot{x}(t)^+|^2 - \alpha \lambda(t)\T \bar{g}(x(t)). \label{eq:dec2}
\end{equation}
We further note that $f(x(t))$ is bounded below, which, by taking the integral of the right-hand side of \eqref{eq:dec2}, implies
\begin{equation}
\int_{0}^{\infty} - |\dot{x}(t)^+|^2 - \alpha \lambda(t)\T \bar{g}(x(t)) \diff t > - \infty. \label{eq:cond3}
\end{equation}
We note that the integrand is closely related to the objective function in \eqref{eq:dual}, which we denote as $\xi_\text{d}(t)$:
\begin{equation*}
\xi_\text{d}(t):=-\frac{1}{2}|\dot{x}(t)^+|^2 - \alpha \lambda(t)\T \bar{g}(x(t)).
\end{equation*}
From the fact that $\lambda(t)$ is bounded and that $-\lambda(t)\T \bar{g}(x(t))$ decays exponentially, we conclude that $\limsup_{t \rightarrow \infty} \xi_\text{d}(t) \leq 0$. From \eqref{eq:cond3} it also follows that the integral of $\xi_\text{d}$ over $\mathbb{R}_{\geq 0}$ is bounded below.

We will now establish that $\lim_{t\rightarrow \infty} \xi_\text{d}(t) =0$ by applying a variant of Barbalat's lemma; see Lemma~\ref{Lemma:cont2} in Appendix~\ref{App:Barbalat}. We start by observing that $\lambda$ inherits the continuity properties of $\dot{x}^+$, due to the fact that $\nabla \bar{g}(x(t))\lambda(t)=\nabla f(x(t))+\dot{x}(t)^+$. This means that $\lambda$ is piecewise continuous, and for each time $t_0>0$, $\lambda(t_0)=\lim_{t\downarrow t_0} \lambda(t)$. The same applies for $\xi_\text{d}$. We now characterize the discontinuities of $\xi_\text{d}$ and provide a lower bound on its derivative, whenever it exists. We fix $t_0>0$. By virtue of Lemma~\ref{Lemma:cont1}, we conclude that $\lambda(t)$ is a feasible candidate for \eqref{eq:dual} at time $t_0$ as long as $t\in (t_0-\delta,t_0)$ for sufficiently small $\delta>0$. This means
\begin{align}
\xi_\text{d}(t_0)&\geq -\frac{1}{2} |\nabla \bar{g} (x(t_0)) \lambda(t) - \nabla f(x(t_0))|^2 - \alpha \bar{g}(x(t_0))\T \lambda(t), \nonumber\\
&\geq \xi_\text{d}(t) - r_1(t_0) |x(t_0)-x(t)| - r_2(t_0) |x(t_0)-x(t)|^2, \label{eq:proof1tmp}
\end{align}
for all $t\in (t_0-\delta,t_0)$, where $r_1(t_0)\geq 0$ and $r_2(t_0)\geq 0$ are related to the remainder terms of a first-order Taylor expansion of $\nabla_x l(x,\lambda(t))$ and $\lambda(t)\T \bar{g}(x)$ with respect to $x$ at $(x(t),\lambda(t))$. The fact that $x(t)$ and $\lambda(t)$ are bounded implies that $r_1(t_0)$ and $r_2(t_0)$ are likewise bounded (uniformly) for all $t_0>0$. Furthermore, $\dot{x}(t)^+$ is bounded, which implies the existence of a constant $\bar{r}_1>0$ (independent of $t_0$) such that
\begin{equation*}
\xi_\text{d}(t_0) \geq \xi_\text{d}(t) - \bar{r}_1 \delta,
\end{equation*}
for sufficiently small $\delta$ and all $t\in (t_0-\delta,t_0)$.
We can now distinguish two cases, depending on whether $\xi_\text{d}$ is continuous at $t_0$ or not. If $\xi_\text{d}$ is discontinuous at $t_0$, we obtain $\xi_\text{d}(t_0) \geq \xi_\text{d}(t_0)^-$. The other case yields $\xi_\text{d}(t_2)\geq \xi_\text{d}(t_1) - \bar{r}_1 (t_2-t_1)$, for all $t_2 \geq t_1$, as long as $\xi_\text{d}$ is continuous on $(t_1,t_2)$. 

We are now ready to apply Lemma~\ref{Lemma:cont2} (see Appendix~\ref{App:Barbalat}), which implies that $\lim_{t\rightarrow \infty} \xi_\text{d}(t)=0$. As a result of the exponential convergence of $\lambda(t)\T \bar{g}(x(t))$, we obtain $\lim_{t\rightarrow \infty} |\dot{x}(t)^+|=0$. \hchange{Let $\bar{x}\in C$ be an accumulation point of $x(t)$, which means that there exists a sequence $x(t_j)$, $j>0$ with $x(t_j) \rightarrow \bar{x}$. From the analysis of $\xi_\text{d}(t)$ we infer that $d(x(t_j))\rightarrow f(\bar{x})$, and from the fact that the function $d$ is upper semicontinuous \citep[see][Thm.~1.17, p.~16]{RockafellarWets} we conclude $f(\bar{x})=\lim_{j\rightarrow \infty} d(x(t_j))\leq d(\bar{x})=f(\bar{x})-|v(\bar{x})|^2/(2\alpha)$. This implies $v(\bar{x})=0$ and shows that $\bar{x}$ is a stationary point of \eqref{eq:fundProb}.}

It remains to show that $x(t)$ converges to a single stationary point in case that the stationary points are isolated. To that end, we consider the sequence $x(k)$, $k>0$. Due to the fact that $\dot{x}(t)^+$ converges, we can find, for every $\epsilon>0$, an integer $N>0$ such that $|x(k+1)-x(k)|<\epsilon$ for all $k>N$. Choosing $\epsilon$ small enough implies that $x(k)$ necessarily converges to a single stationary point, which we denote by $x_\text{s}$ (this would otherwise contradict the fact that the stationary points are isolated). Moreover, $|x(t)-x_\text{s}|\leq |x(t)-x(k_t)| + |x(k_t)-x_\text{s}|$, where $k_t$ is the largest integer such that $k_t<t$. We conclude $\lim_{t\rightarrow \infty} x(t)=x_\text{s}$ by observing that $|x(t)-x(k_t)|$ is bounded by the supremum of $|\dot{x}(\tau)^+|$ over $\tau \in (k_t,t)$, which becomes arbitrarily small for large $t$.
\end{proof}

\begin{claim}
Let the assumptions of Proposition 2 be satisfied including Assumption~\ref{Ass:Conv} (convexity) and let $\alpha\leq 2\mu$. Then the following holds:
\begin{equation*}
(h(x(0)),\min\{0,g(x(0))\})\T \lambda^* e^{-\alpha t} \leq f(x(t))-f^*\leq (f(x(0))-f^*) e^{-2 \mu t}, 
\end{equation*}
for all $x(0)\in \mathbb{R}^n$, where $x(t)$ satisfies \eqref{eq:velLeveltmp} and \eqref{eq:velLeveltmp2}, $f^*$ is the optimal cost in \eqref{eq:fundProb} and $\lambda^*$ is a corresponding multiplier that satisfies the Karush-Kuhn-Tucker conditions \hchange{of \eqref{eq:fundProb}}.
\end{claim}
\begin{proof} 
We will use \eqref{eq:dec2} as a starting point for deriving the upper bound. From \eqref{eq:importantUBd2} (see Lemma~\ref{Lem:discrete2}, Appendix~\ref{App:d}) we conclude that
\begin{equation*}
-|\dot{x}(t)^+|^2 \leq - 2\mu (f(x(t))-f^*) + 2 \mu \lambda(t)\T \bar{g}(t).
\end{equation*}
Thus, inserting the upper bound on $-|\dot{x}(t)^+|^2$ in \eqref{eq:dec2}, we obtain
\begin{equation*}
\frac{\diff}{\dt} f(x(t))^+ \leq - 2\mu (f(x(t))-f(x^*)) + (2\mu - \alpha) \lambda(t)\T \bar{g}(x(t)).
\end{equation*}
For $\alpha \leq 2\mu$, the term $(2\mu-\alpha) \lambda(t)\T \bar{g}(x(t))$ is certainly negative (or vanishes completely if $x(0)\in C$), which readily proves the upper bound.

The lower bound follows from a perturbation analysis. For a given $x(0)\in C$, we define 
\begin{equation*}
f^*(t):=\min_{z\in \mathbb{R}^n} f(z), \quad \text{s.t.} \quad  h(z)=h(x(0))e^{-\alpha t}, \quad g(z)\geq \min\{0, g(x(0))\} e^{-\alpha t},
\end{equation*} 
which is of the form \eqref{eq:fundProb}, with the sole difference that the right-hand side of the constraints has been replaced with the vector $(h(x(0)),\min\{0,g(x(0))\}) \exp(-\alpha t)$. The trajectory $x(t)$ is guaranteed to be feasible with respect to these modified constraints, which implies that $f^*(t)\leq f(x(t))$. The minimum is attained for all $t\in [0,\infty)$, due to the fact that $f$ is bounded below and the modified set of feasible points is closed. A multiplier $\lambda^*$ satisfying the Karush-Kuhn-Tucker conditions of \eqref{eq:fundProb} captures the sensitivity of the cost function with respect to perturbations of the right-hand side of the constraints. More precisely, $-\lambda^*$ is guaranteed to satisfy the following inequality~\citep[see, e.g.,][p.~277]{RockafellarConvex}:
\begin{equation*}
f^*(t)-f^*\geq (h(x(0)),\min\{0,g(x(0))\})\T \lambda^*  \exp(-\alpha t).
\end{equation*}
The lower bound of \eqref{eq:boundinfeas} in Proposition~\ref{Prop:GF} then follows from the fact that $f(x(t))\geq f^*(t)$ for all $t\in [0,\infty)$.
\end{proof}
\section{A First Example}\label{Sec:Example}
In this section we present an example that illustrates the behavior of \eqref{eq:velLeveltmp} and \eqref{eq:disAlgorithm}. We consider the following problem:
\begin{equation}
    \min_{x\in \mathbb{R}} \frac{1}{10} (x+1)^2, \quad \text{s.t.} \quad x\in [0,2], \label{eq:simpEx}
\end{equation}
which has the unique minimum $x^*=0$. The function $f$ is therefore given by $(x+1)^2/10$, whereas $g_1(x)=x$ and $g_2(x)=2-x$. It will be instructive to plot the function $\nabla_x l(x,\lambda(x))=\nabla f(x) - R(x)$, where the multiplier $\lambda(x)$ is obtained from \eqref{eq:dual}. This yields a continuous-time gradient flow that is given by $\dot{x}(t)^+=-\nabla_x l(x(t),\lambda(t))$, whereas the discrete-time version is given by $x_{k+1}-x_{k}=-T \nabla_x l(x_k,\lambda_k)$, where $\lambda(t)$ and $\lambda_k$ are implicitly dependent on $x(t)$ and $x_k$, respectively. Furthermore, we can interpret $\nabla_x l(x,\lambda(x))$ as the gradient of a continuous function $F_\alpha: \mathbb{R} \rightarrow \mathbb{R}$, with $F_\alpha(0)=f^*$. We also plot the function $d(x)$ as defined in \eqref{eq:dual2}.

\newlength{\figurewidth}
\newlength{\figureheight}

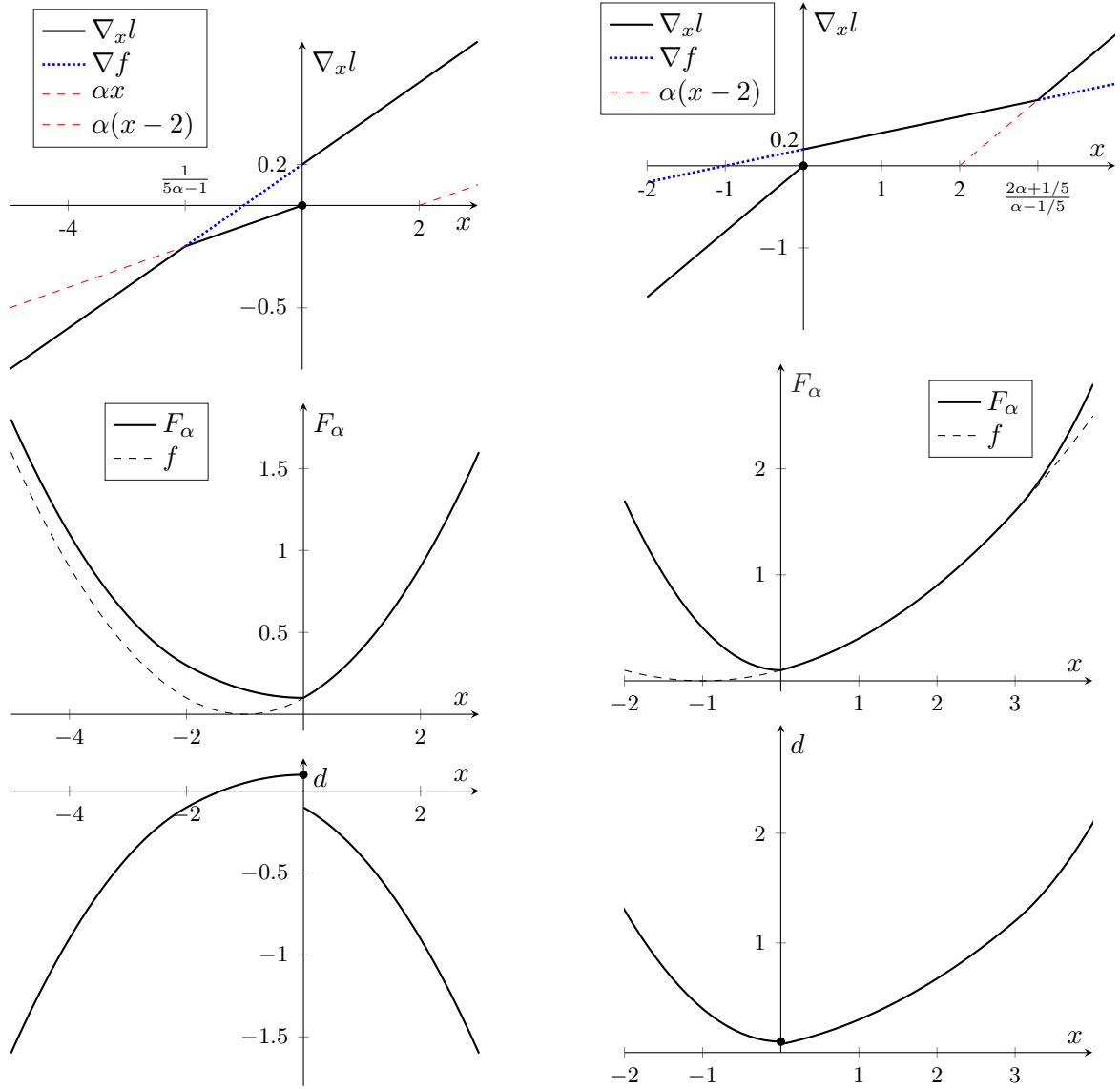
\begin{figure}
    
    \setlength{\figurewidth}{.45\columnwidth}
    \setlength{\figureheight}{.3\columnwidth}
    
    \begin{minipage}[l]{.45\columnwidth}
    \centering
%
%
\begin{tikzpicture}

\begin{axis}[%
width=0.951\figurewidth,
height=\figureheight,
at={(0\figurewidth,0\figureheight)},
scale only axis,
xmin=-5,
xmax=3,
xlabel style={font=\color{white!15!black}},
xtick={-4,2},
xticklabels={-4,2},
extra x ticks={-2},
extra x tick style={
    xticklabel style={yshift=0.5ex, anchor=south}
},
extra x tick labels={$\frac{1}{5\alpha -1}$},
ytick={-0.5,0.2},
xlabel={$x$},
xlabel near ticks,
ymin=-0.8,
ymax=0.8,
ylabel style={font=\color{white!15!black}},
ylabel={$\nabla_x l$},
ylabel near ticks,
axis background/.style={fill=white},
axis x line=middle,
axis y line=middle,
xlabel style={at={(axis description cs:0.97,0.49)},anchor=north},
legend style={legend cell align=left, align=left, draw=white!15!black},
legend style={at={(0.05,1.1)},anchor=north west}
]
\addplot [color=black,thick]
  table[row sep=crcr]{%
-5	-0.8\\
-4.9	-0.78\\
-4.8	-0.76\\
-4.7	-0.74\\
-4.6	-0.72\\
-4.5	-0.7\\
-4.4	-0.68\\
-4.3	-0.66\\
-4.2	-0.64\\
-4.1	-0.62\\
-4	-0.6\\
-3.9	-0.58\\
-3.8	-0.56\\
-3.7	-0.54\\
-3.6	-0.52\\
-3.5	-0.5\\
-3.4	-0.48\\
-3.3	-0.46\\
-3.2	-0.44\\
-3.1	-0.42\\
-3	-0.4\\
-2.9	-0.38\\
-2.8	-0.36\\
-2.7	-0.34\\
-2.6	-0.32\\
-2.5	-0.3\\
-2.4	-0.28\\
-2.3	-0.26\\
-2.2	-0.24\\
-2.1	-0.22\\
-2	-0.2\\
-1.9	-0.19\\
-1.8	-0.18\\
-1.7	-0.17\\
-1.6	-0.16\\
-1.5	-0.15\\
-1.4	-0.14\\
-1.3	-0.13\\
-1.2	-0.12\\
-1.1	-0.11\\
-1	-0.1\\
-0.9	-0.09\\
-0.8	-0.08\\
-0.7	-0.07\\
-0.6	-0.06\\
-0.5	-0.05\\
-0.4	-0.04\\
-0.3	-0.03\\
-0.2	-0.02\\
-0.1	-0.01\\
0	0\\
};
\addlegendentry{$\nabla_x l$}

\addplot [color=black, thick, forget plot]
  table[row sep=crcr]{%
0.001	0.2002\\
0.101	0.2202\\
0.201	0.2402\\
0.301	0.2602\\
0.401	0.2802\\
0.501	0.3002\\
0.601	0.3202\\
0.701	0.3402\\
0.801	0.3602\\
0.901	0.3802\\
1.001	0.4002\\
1.101	0.4202\\
1.201	0.4402\\
1.301	0.4602\\
1.401	0.4802\\
1.501	0.5002\\
1.601	0.5202\\
1.701	0.5402\\
1.801	0.5602\\
1.901	0.5802\\
2.001	0.6002\\
2.101	0.6202\\
2.201	0.6402\\
2.301	0.6602\\
2.401	0.6802\\
2.501	0.7002\\
2.601	0.7202\\
2.701	0.7402\\
2.801	0.7602\\
2.901	0.7802\\
3.001	0.8002\\
};

\addplot [color=black,only marks, mark size=1.5pt,forget plot]
 table[row sep=crcr]{%
 0 0\\
 };

\addplot [color=blue, line width=1pt, densely dotted]
  table[row sep=crcr]{%
-2	-0.2\\
-1.9	-0.18\\
-1.8	-0.16\\
-1.7	-0.14\\
-1.6	-0.12\\
-1.5	-0.1\\
-1.4	-0.08\\
-1.3	-0.06\\
-1.2	-0.04\\
-1.1	-0.02\\
-1	0\\
-0.9	0.02\\
-0.8	0.04\\
-0.7	0.06\\
-0.6	0.08\\
-0.5	0.1\\
-0.4	0.12\\
-0.3	0.14\\
-0.2	0.16\\
-0.1	0.18\\
0	0.2\\
};
\addlegendentry{$\nabla f$}

\addplot [color=red, dashed]
  table[row sep=crcr]{%
-5	-0.5\\
-4.9	-0.49\\
-4.8	-0.48\\
-4.7	-0.47\\
-4.6	-0.46\\
-4.5	-0.45\\
-4.4	-0.44\\
-4.3	-0.43\\
-4.2	-0.42\\
-4.1	-0.41\\
-4	-0.4\\
-3.9	-0.39\\
-3.8	-0.38\\
-3.7	-0.37\\
-3.6	-0.36\\
-3.5	-0.35\\
-3.4	-0.34\\
-3.3	-0.33\\
-3.2	-0.32\\
-3.1	-0.31\\
-3	-0.3\\
-2.9	-0.29\\
-2.8	-0.28\\
-2.7	-0.27\\
-2.6	-0.26\\
-2.5	-0.25\\
-2.4	-0.24\\
-2.3	-0.23\\
-2.2	-0.22\\
-2.1	-0.21\\
-2	-0.2\\
};
\addlegendentry{$\alpha x$}

\addplot [color=red, dashed]
  table[row sep=crcr]{%
2	0\\
2.1	0.01\\
2.2	0.02\\
2.3	0.03\\
2.4	0.04\\
2.5	0.05\\
2.6	0.06\\
2.7	0.07\\
2.8	0.08\\
2.9	0.09\\
3	0.1\\
3.1	0.11\\
};
\addlegendentry{$\alpha (x-2)$}

\end{axis}
\end{tikzpicture}%
\\[12pt]
%
%
\begin{tikzpicture}

\begin{axis}[%
width=0.951\figurewidth,
height=\figureheight,
at={(0\figurewidth,0\figureheight)},
scale only axis,
xmin=-5,
xmax=3,
xlabel style={font=\color{white!15!black}},
xlabel={$x$},
ymin=-.1,
ymax=1.9,
ylabel style={font=\color{white!15!black}},
ylabel={$F_\alpha$},
axis background/.style={fill=white},
axis x line=middle,
axis y line=middle,
legend style={legend cell align=left, align=left, draw=white!15!black},
legend style={at={(0.2,1.0)},anchor=north west},
]
\addplot [color=black,thick]
  table[row sep=crcr]{%
-5	1.8\\
-4.9	1.721\\
-4.8	1.644\\
-4.7	1.569\\
-4.6	1.496\\
-4.5	1.425\\
-4.4	1.356\\
-4.3	1.289\\
-4.2	1.224\\
-4.1	1.161\\
-4	1.1\\
-3.9	1.041\\
-3.8	0.984\\
-3.7	0.929\\
-3.6	0.876\\
-3.5	0.825\\
-3.4	0.776\\
-3.3	0.729\\
-3.2	0.684\\
-3.1	0.641\\
-3	0.6\\
-2.9	0.561\\
-2.8	0.524\\
-2.7	0.489\\
-2.6	0.456\\
-2.5	0.425\\
-2.4	0.396\\
-2.3	0.369\\
-2.2	0.344\\
-2.1	0.321\\
-2	0.3\\
-1.9	0.2805\\
-1.8	0.262\\
-1.7	0.2445\\
-1.6	0.228\\
-1.5	0.2125\\
-1.4	0.198\\
-1.3	0.1845\\
-1.2	0.172\\
-1.1	0.1605\\
-1	0.15\\
-0.9	0.1405\\
-0.8	0.132\\
-0.7	0.1245\\
-0.6	0.118\\
-0.5	0.1125\\
-0.4	0.108\\
-0.3	0.1045\\
-0.2	0.102\\
-0.1	0.1005\\
0	0.1\\
0.0999999999999996	0.121\\
0.2	0.144\\
0.3	0.169\\
0.4	0.196\\
0.5	0.225\\
0.6	0.256\\
0.7	0.289\\
0.8	0.324\\
0.9	0.361\\
1	0.4\\
1.1	0.441\\
1.2	0.484\\
1.3	0.529\\
1.4	0.576\\
1.5	0.625\\
1.6	0.676\\
1.7	0.729\\
1.8	0.784\\
1.9	0.841\\
2	0.9\\
2.1	0.961\\
2.2	1.024\\
2.3	1.089\\
2.4	1.156\\
2.5	1.225\\
2.6	1.296\\
2.7	1.369\\
2.8	1.444\\
2.9	1.521\\
3	1.6\\
};
\addlegendentry{$F_\alpha$}

\addplot [color=black, dashed]
  table[row sep=crcr]{%
-5	1.6\\
-4.9	1.521\\
-4.8	1.444\\
-4.7	1.369\\
-4.6	1.296\\
-4.5	1.225\\
-4.4	1.156\\
-4.3	1.089\\
-4.2	1.024\\
-4.1	0.961\\
-4	0.9\\
-3.9	0.841\\
-3.8	0.784\\
-3.7	0.729\\
-3.6	0.676\\
-3.5	0.625\\
-3.4	0.576\\
-3.3	0.529\\
-3.2	0.484\\
-3.1	0.441\\
-3	0.4\\
-2.9	0.361\\
-2.8	0.324\\
-2.7	0.289\\
-2.6	0.256\\
-2.5	0.225\\
-2.4	0.196\\
-2.3	0.169\\
-2.2	0.144\\
-2.1	0.121\\
-2	0.1\\
-1.9	0.081\\
-1.8	0.064\\
-1.7	0.049\\
-1.6	0.036\\
-1.5	0.025\\
-1.4	0.016\\
-1.3	0.009\\
-1.2	0.00400000000000001\\
-1.1	0.001\\
-1	0\\
-0.9	0.001\\
-0.8	0.004\\
-0.7	0.009\\
-0.6	0.016\\
-0.5	0.025\\
-0.4	0.036\\
-0.3	0.049\\
-0.2	0.064\\
-0.1	0.081\\
0	0.1\\
};
\addlegendentry{$f$}

\end{axis}
\end{tikzpicture}%
\\
%
%
\begin{tikzpicture}

\begin{axis}[%
width=0.951\figurewidth,
height=\figureheight,
at={(0\figurewidth,0\figureheight)},
scale only axis,
xmin=-5,
xmax=3,
xlabel style={font=\color{white!15!black}},
xlabel={$x$},
ymin=-1.8,
ymax=0.2,
ylabel style={font=\color{white!15!black}},
ylabel={$d$},
axis background/.style={fill=white},
axis x line=middle,
axis y line=middle,
legend style={legend cell align=left, align=left, draw=white!15!black}
]
\addplot [color=black,thick,forget plot]
  table[row sep=crcr]{%
-5	-1.6\\
-4.9	-1.521\\
-4.8	-1.444\\
-4.7	-1.369\\
-4.6	-1.296\\
-4.5	-1.225\\
-4.4	-1.156\\
-4.3	-1.089\\
-4.2	-1.024\\
-4.1	-0.961\\
-4	-0.9\\
-3.9	-0.841\\
-3.8	-0.784\\
-3.7	-0.729\\
-3.6	-0.676\\
-3.5	-0.625\\
-3.4	-0.576\\
-3.3	-0.529\\
-3.2	-0.484\\
-3.1	-0.441\\
-3	-0.4\\
-2.9	-0.361\\
-2.8	-0.324\\
-2.7	-0.289\\
-2.6	-0.256\\
-2.5	-0.225\\
-2.4	-0.196\\
-2.3	-0.169\\
-2.2	-0.144\\
-2.1	-0.121\\
-2	-0.1\\
-1.9	-0.0805000000000001\\
-1.8	-0.062\\
-1.7	-0.0445\\
-1.6	-0.028\\
-1.5	-0.0125\\
-1.4	0.00199999999999999\\
-1.3	0.0155\\
-1.2	0.028\\
-1.1	0.0395\\
-1	0.05\\
-0.9	0.0595\\
-0.8	0.068\\
-0.7	0.0755\\
-0.6	0.082\\
-0.5	0.0875\\
-0.4	0.092\\
-0.3	0.0955\\
-0.2	0.098\\
-0.1	0.0995\\
0	0.1\\
};

\addplot [color=black,only marks, mark size=1.5pt,forget plot]
 table[row sep=crcr]{%
 0 0.1\\
 };

\addplot [color=black,thick,forget plot]
  table[row sep=crcr]{%
0.001	-0.1002001\\
0.101	-0.1212201\\
0.201	-0.1442401\\
0.301	-0.1692601\\
0.401	-0.1962801\\
0.501	-0.2253001\\
0.601	-0.2563201\\
0.701	-0.2893401\\
0.801	-0.3243601\\
0.901	-0.3613801\\
1.001	-0.4004001\\
1.101	-0.4414201\\
1.201	-0.4844401\\
1.301	-0.5294601\\
1.401	-0.5764801\\
1.501	-0.6255001\\
1.601	-0.6765201\\
1.701	-0.729540099999999\\
1.801	-0.7845601\\
1.901	-0.8415801\\
2.001	-0.9006001\\
2.101	-0.9616201\\
2.201	-1.0246401\\
2.301	-1.0896601\\
2.401	-1.1566801\\
2.501	-1.2257001\\
2.601	-1.2967201\\
2.701	-1.3697401\\
2.801	-1.4447601\\
2.901	-1.5217801\\
3.001	-1.6008001\\
};

\end{axis}
\end{tikzpicture}%
    \end{minipage}\hfill
    \begin{minipage}[r]{.45\columnwidth}
    \centering
%
%
\begin{tikzpicture}

\begin{axis}[%
width=0.951\figurewidth,
height=\figureheight,
at={(0\figurewidth,0\figureheight)},
scale only axis,
xmin=-2,
xmax=3.99,
xlabel style={font=\color{white!15!black}},
xlabel={$x$},
xtick={-2,-1,1,2,3},
xticklabels={-2,-1,1,2,$\frac{2\alpha + 1/5}{\alpha -1/5}$},
ytick={-1},
extra y ticks={0.2},
extra y tick style={
    yticklabel style={xshift=-3ex, yshift=.8ex, anchor=west}
},
extra y tick labels={0.2},
ymin=-2,
ymax=1.99,
ylabel style={font=\color{white!15!black}},
ylabel={$\nabla_x l$},
axis background/.style={fill=white},
axis x line=middle,
axis y line=middle,
legend style={legend cell align=left, align=left, draw=white!15!black},
legend style={at={(-0.1,1.0)},anchor=north west}
]
\addplot [color=black,thick]
  table[row sep=crcr]{%
-2	-1.6\\
-1.9	-1.52\\
-1.8	-1.44\\
-1.7	-1.36\\
-1.6	-1.28\\
-1.5	-1.2\\
-1.4	-1.12\\
-1.3	-1.04\\
-1.2	-0.96\\
-1.1	-0.88\\
-1	-0.8\\
-0.9	-0.72\\
-0.8	-0.64\\
-0.7	-0.56\\
-0.6	-0.48\\
-0.5	-0.4\\
-0.4	-0.32\\
-0.3	-0.24\\
-0.2	-0.16\\
-0.1	-0.08\\
0	0\\
};
\addlegendentry{$\nabla_x l$}

\addplot [color=black,thick,forget plot]
  table[row sep=crcr]{%
0.001	0.2002\\
0.101	0.2202\\
0.201	0.2402\\
0.301	0.2602\\
0.401	0.2802\\
0.501	0.3002\\
0.601	0.3202\\
0.701	0.3402\\
0.801	0.3602\\
0.901	0.3802\\
1.001	0.4002\\
1.101	0.4202\\
1.201	0.4402\\
1.301	0.4602\\
1.401	0.4802\\
1.501	0.5002\\
1.601	0.5202\\
1.701	0.5402\\
1.801	0.5602\\
1.901	0.5802\\
2.001	0.6002\\
2.101	0.6202\\
2.201	0.6402\\
2.301	0.6602\\
2.401	0.6802\\
2.501	0.7002\\
2.601	0.7202\\
2.701	0.7402\\
2.801	0.7602\\
2.901	0.7802\\
3.001	0.8008\\
3.101	0.8808\\
3.201	0.9608\\
3.301	1.0408\\
3.401	1.1208\\
3.501	1.2008\\
3.601	1.2808\\
3.701	1.3608\\
3.801	1.4408\\
3.901	1.5208\\
4.001	1.6008\\
};

\addplot [color=black,only marks, mark size=1.5pt,forget plot]
 table[row sep=crcr]{%
 0 0\\
 };

\addplot [color=blue, line width=1pt, densely dotted]
  table[row sep=crcr]{%
-3	-0.4\\
-2.9	-0.38\\
-2.8	-0.36\\
-2.7	-0.34\\
-2.6	-0.32\\
-2.5	-0.3\\
-2.4	-0.28\\
-2.3	-0.26\\
-2.2	-0.24\\
-2.1	-0.22\\
-2	-0.2\\
-1.9	-0.18\\
-1.8	-0.16\\
-1.7	-0.14\\
-1.6	-0.12\\
-1.5	-0.1\\
-1.4	-0.08\\
-1.3	-0.06\\
-1.2	-0.04\\
-1.1	-0.02\\
-1	0\\
-0.9	0.02\\
-0.8	0.04\\
-0.7	0.06\\
-0.6	0.08\\
-0.5	0.1\\
-0.4	0.12\\
-0.3	0.14\\
-0.2	0.16\\
-0.1	0.18\\
0	0.2\\
};
\addlegendentry{$\nabla f$}

\addplot [color=blue, line width=1pt, densely dotted,forget plot]
  table[row sep=crcr]{%
3	0.8\\
3.1	0.82\\
3.2	0.84\\
3.3	0.86\\
3.4	0.88\\
3.5	0.9\\
3.6	0.92\\
3.7	0.94\\
3.8	0.96\\
3.9	0.98\\
4	1\\
};

\addplot [color=red, dashed]
  table[row sep=crcr]{%
2	0\\
2.1	0.0800000000000001\\
2.2	0.16\\
2.3	0.24\\
2.4	0.32\\
2.5	0.4\\
2.6	0.48\\
2.7	0.559999999999999\\
2.8	0.64\\
2.9	0.72\\
3	0.8\\
};
\addlegendentry{$\alpha (x-2)$}

\end{axis}
\end{tikzpicture}%
\\[12pt]
%
%
\begin{tikzpicture}

\begin{axis}[%
width=0.951\figurewidth,
height=\figureheight,
at={(0\figurewidth,0\figureheight)},
scale only axis,
xmin=-2,
xmax=3.99,
xlabel style={font=\color{white!15!black}},
xlabel={$x$},
axis x line=middle,
axis y line=middle,
ymin=-.1,
ymax=2.99,
ylabel style={font=\color{white!15!black}},
ylabel={$F_\alpha$},
axis background/.style={fill=white},
legend style={legend cell align=left, align=left, draw=white!15!black},
legend style={at={(0.65,.95)},anchor=north west},
]
\addplot [color=black,thick]
  table[row sep=crcr]{%
-2	1.7\\
-1.9	1.544\\
-1.8	1.396\\
-1.7	1.256\\
-1.6	1.124\\
-1.5	1\\
-1.4	0.884\\
-1.3	0.776\\
-1.2	0.676\\
-1.1	0.584\\
-1	0.5\\
-0.9	0.424\\
-0.8	0.356\\
-0.7	0.296\\
-0.6	0.244\\
-0.5	0.2\\
-0.4	0.164\\
-0.3	0.136\\
-0.2	0.116\\
-0.0999999999999999	0.104\\
0	0.1\\
0.1	0.121\\
0.2	0.144\\
0.3	0.169\\
0.4	0.196\\
0.5	0.225\\
0.6	0.256\\
0.7	0.289\\
0.8	0.324\\
0.9	0.361\\
1	0.4\\
1.1	0.441\\
1.2	0.484\\
1.3	0.529\\
1.4	0.576\\
1.5	0.625\\
1.6	0.676\\
1.7	0.729\\
1.8	0.784\\
1.9	0.841\\
2	0.9\\
2.1	0.961\\
2.2	1.024\\
2.3	1.089\\
2.4	1.156\\
2.5	1.225\\
2.6	1.296\\
2.7	1.369\\
2.8	1.444\\
2.9	1.521\\
3	1.6\\
3.1	1.684\\
3.2	1.776\\
3.3	1.876\\
3.4	1.984\\
3.5	2.1\\
3.6	2.224\\
3.7	2.356\\
3.8	2.496\\
3.9	2.644\\
4	2.8\\
};
\addlegendentry{$F_\alpha$}

\addplot [color=black, dashed]
  table[row sep=crcr]{%
3	1.6\\
3.1	1.681\\
3.2	1.764\\
3.3	1.849\\
3.4	1.936\\
3.5	2.025\\
3.6	2.116\\
3.7	2.209\\
3.8	2.304\\
3.9	2.401\\
4	2.5\\
};
\addlegendentry{$f$}

\addplot [color=black, dashed,forget plot]
  table[row sep=crcr]{%
-2	0.1\\
-1.9	0.081\\
-1.8	0.064\\
-1.7	0.049\\
-1.6	0.036\\
-1.5	0.025\\
-1.4	0.016\\
-1.3	0.00899999999999999\\
-1.2	0.004\\
-1.1	0.001\\
-1	0\\
-0.9	0.001\\
-0.8	0.004\\
-0.7	0.009\\
-0.6	0.016\\
-0.5	0.025\\
-0.4	0.036\\
-0.3	0.049\\
-0.2	0.064\\
-0.1	0.081\\
0	0.1\\
};

\end{axis}
\end{tikzpicture}%
\\
%
%
\begin{tikzpicture}

\begin{axis}[%
width=0.951\figurewidth,
height=\figureheight,
at={(0\figurewidth,0\figureheight)},
scale only axis,
xmin=-2,
xmax=3.99,
xlabel style={font=\color{white!15!black}},
xlabel={$x$},
ymin=0,
ymax=2.99,
ylabel style={font=\color{white!15!black}},
ylabel={$d$},
axis background/.style={fill=white},
axis x line=middle,
axis y line=middle,
legend style={legend cell align=left, align=left, draw=white!15!black}
]
\addplot [color=black,thick,forget plot]
  table[row sep=crcr]{%
-3	2.8\\
-2.9	2.623\\
-2.8	2.452\\
-2.7	2.287\\
-2.6	2.128\\
-2.5	1.975\\
-2.4	1.828\\
-2.3	1.687\\
-2.2	1.552\\
-2.1	1.423\\
-2	1.3\\
-1.9	1.183\\
-1.8	1.072\\
-1.7	0.967\\
-1.6	0.868\\
-1.5	0.775\\
-1.4	0.688\\
-1.3	0.607\\
-1.2	0.532\\
-1.1	0.463\\
-1	0.4\\
-0.9	0.343\\
-0.8	0.292\\
-0.7	0.247\\
-0.6	0.208\\
-0.5	0.175\\
-0.4	0.148\\
-0.3	0.127\\
-0.2	0.112\\
-0.1	0.103\\
0	0.1\\
};

\addplot [color=black,only marks, mark size=1.5pt,forget plot]
 table[row sep=crcr]{%
 0 0.1\\
 };

\addplot [color=black,thick,forget plot]
  table[row sep=crcr]{%
0.001	0.075150075\\
0.101	0.090915075\\
0.201	0.108180075\\
0.301	0.126945075\\
0.401	0.147210075\\
0.501	0.168975075\\
0.601	0.192240075\\
0.701	0.217005075\\
0.801	0.243270075\\
0.901	0.271035075\\
1.001	0.300300075\\
1.101	0.331065075\\
1.201	0.363330075\\
1.301	0.397095075\\
1.401	0.432360075\\
1.501	0.469125075\\
1.601	0.507390075\\
1.701	0.547155075\\
1.801	0.588420075\\
1.901	0.631185075\\
2.001	0.675450075\\
2.101	0.721215075\\
2.201	0.768480075\\
2.301	0.817245075\\
2.401	0.867510075\\
2.501	0.919275075\\
2.601	0.972540075\\
2.701	1.027305075\\
2.801	1.083570075\\
2.901	1.141335075\\
3.001	1.2006003\\
3.101	1.2636603\\
3.201	1.3327203\\
3.301	1.4077803\\
3.401	1.4888403\\
3.501	1.5759003\\
3.601	1.6689603\\
3.701	1.7680203\\
3.801	1.8730803\\
3.901	1.9841403\\
4.001	2.1012003\\
4.101	2.2242603\\
4.201	2.3533203\\
4.301	2.4883803\\
4.401	2.6294403\\
4.501	2.7765003\\
4.601	2.9295603\\
4.701	3.0886203\\
4.801	3.2536803\\
4.901	3.4247403\\
5.001	3.6018003\\
};

\end{axis}
\end{tikzpicture}%
    \end{minipage}
    \caption{This figure shows the values of $\nabla_x l$, $F_\alpha$, and $d$ for $\alpha=1/10$ (left column) and $\alpha=4/5$ (right column). Top row: The solid thick black line represents $\nabla_x l$, which is discontinuous at the origin, where it takes the value zero (the origin is the minimizer of \eqref{eq:simpEx}). For values $x\leq 0$, $\nabla_x l$ is given by $\min \{ \nabla f(x), \alpha x\}$ and for values $x\geq 2$, $\nabla_x l$ is given by $\max \{ \nabla f(x), \alpha (x-2) \}$, which is represented by the lines in blue and in red. Middle row: The solid thick black line represents $F_\alpha$, which is continuous and has its minimum at the origin (the origin is the minimizer of \eqref{eq:simpEx}). The objective function $f$ is indicated with dashed lines. Last row: The function $d$ is discontinuous at the origin for $\alpha \neq 1/5$, unbounded below for $\alpha < 1/5$, and unbounded above for $\alpha > 1/5$. For $\alpha < 1/5$, $d(x)$ is upper bounded by $f_{I_{x}}^*$, that is, $d(x) \leq f^* =f^*_{\{1\}}= 0.1$ for $x\leq 0$ and $d(x)\leq f^*_{\{ \}} = f^*_{\{2\}}=0$ for $x>0$, where $g_1(x)=x$ and $g_2(x)=2-x$. As we will show in Section~\ref{Sec:Discrete}, this holds more generally provided that $f$ and $C$ are convex.}
    \label{fig:simpExample}
\end{figure}

The plots are shown in Figure~\ref{fig:simpExample} for two different $\alpha$. The left column is prototypical for $\alpha \leq 1/5$, the right column for $\alpha >1/5$, where $1/5$ amounts to the Hessian of $f$. It is important to note that $\nabla_x l$ is discontinuous at the origin, but nonetheless unique. In the continuous-time case, the discontinuity at the origin is less of an issue, since the solutions to $\dot{x}(t)^+=-\nabla_x l (x(t),\lambda(t))$ approach the origin either from $x(t)>0$ or from $x(t)<0$ and never cross the origin. When the solution approaches the origin from negative values, $x(t)<0$, the velocity $\dot{x}(t)$ continuously reduces to zero for $t\rightarrow \infty$. If the solutions approach the origin from positive values, $x(t)>0$, the velocity continuously reduces to $\dot{x}(t)^-=-0.2$ at which point it instantly drops to zero. Hence, if $x(t)$ approaches the origin from positive values, the convergence is in finite time. The origin is therefore a stable and attractive equilibrium in the sense of Lyapunov.

In discrete time, the situation changes drastically. Starting from a generic initial condition, $x_0>0$, the solution to $x_{k+1}=x_{k} -T \nabla_x l (x_k,\lambda_k)$ crosses the origin and eventually always approaches the origin from $x_k<0$ (provided that $\alpha$ and $T$ are small enough). For small $\alpha$ and $T$, the origin can therefore be viewed as a semi-permeable membrane; solutions cross from $x_k>0$ to $x_{k+1}<0$, but not vice versa. The origin is \emph{not} a stable equilibrium, since trajectories starting arbitrarily close to the origin will jump to a negative $x_1$, such that $|x_1|\geq |0.2 T - (1-0.2 T) x_0|\approx 0.2 T$. (Hence, no matter how small we choose $\delta>0$, there exists an initial condition $x_0$ with $|x_0|<\delta$ such that $|x_k|\geq 0.1T$ for some $k\geq 0$.) We therefore conclude that any attempt to find a continuous Lyapunov function for proving convergence in discrete time is doomed to fail. Indeed, as we will show in the following, proving convergence of \eqref{eq:disAlgorithm} hinges on the analysis of the \emph{discontinuous} function $d(x)$, which can be shown to be monotonically increasing along trajectories $x_k$ for small enough $\alpha$ and $T$. The analysis can also be interpreted as choosing an appropriate sequence of nested invariant sets, which generalizes the above discussion of the origin acting as a semi-permeable membrane. Each of these invariant sets can then be shown to be attractive, whereby trajectories traverse most of these invariant sets in finite time.

We would like to emphasize that even though the origin is \emph{not} stable in the sense of Lyapunov (in discrete time), it is still \emph{attractive}; that is, $x_k$ converges to origin for small enough $\alpha$ and $T$. From Figure~\ref{fig:simpExample}, it follows that $\alpha T \leq 1$ is necessary for ensuring that trajectories approach the origin from $x_k<0$ for large $k$. If $\alpha T>1$, we observe oscillations about the origin. We further note that already the analysis of a two-dimensional problem with multiple linear constraints appears to be very challenging due to the discontinuity of $\nabla_x l$ and the discrete nature of \eqref{eq:disAlgorithm}, which results in a multitude of different constraints that may or may not become active over the course of the optimization.

\section{The Discrete-Time Case}\label{Sec:Discrete}
This section analyzes the convergence of algorithm \eqref{eq:disAlgorithm} to stationary points of \eqref{eq:fundProb}. In contrast to the continuous-time setting, where a trajectory starting from $x(0)\in C$ is guaranteed to remain feasible, a discrete trajectory $x_k$ may become infeasible in the course of the optimization, even if $x_0\in C$. This is due to the finite length of each step of the discrete algorithm and the fact that only the active constraints $I_{x_k}$ are taken into account. While this potentially saves computation and distinguishes our algorithm from other methods, it also complicates the analysis. As we discussed in the previous section, while trajectories still converge to the minimizer of \eqref{eq:fundProb} (assuming convexity and appropriately chosen parameters $T$ and $\alpha$), the minimizer may not correspond to a stable equilibrium in the sense of Lyapunov.

In Section~\ref{Sec:Example}, we saw that for $\alpha T \leq 1$, the solutions $x_k$ of algorithm \eqref{eq:disAlgorithm} cross the origin from $x_k>0$ to $x_{k+1}<0$, but not vice versa. The property is crucial for guaranteeing convergence, as it excludes oscillations about the origin. We can therefore visualize the boundary of the feasible set as a semi-permeable membrane; trajectories can pass from the feasible to the infeasible region, but not the other way. The following lemma will be the first step in making this observation precise.

\begin{lemma}\label{Lem:discrete}
Let $C$ be convex. Provided that $\alpha T \leq 1$, the inequality constraints at time $k$ for which the corresponding $\lambda_{ki}$ is nonzero will remain active at time $k+1$. In other words, $\lambda_{ki}>0$ implies $g_i(x_{k+1})\leq 0$.
\end{lemma}
\begin{proof}
The stationarity condition \eqref{eq:statDual}, which applies in the same way to the discrete algorithm \eqref{eq:disAlgorithm} (it suffices to replace $x(t)$ with $x_k$, $\lambda(t)$ by $\lambda_k$, and $\dot{x}(t)^+=(x_{k+1}-x_k)/T$), implies that $\lambda_{ki} \nabla \bar{g}_i(x_{k})\T (x_{k+1}-x_k) = -\alpha T \bar{g}_i(x_{k}) \lambda_{ki}$ for all $i\in \{1,2,\dots,n_\text{h}+n_\text{g}\}$ (complementary slackness). \hchange{Due to the fact that $C$ is convex, there exist linear functions $h$ and concave functions $g$ which describe $C$.} Thus, it follows that
\begin{align*}
h(x_{k+1})&= h(x_k) + \nabla h(x_k)\T (x_{k+1}-x_k),\\
g(x_{k+1})&\leq g(x_k) + \nabla g(x_k)\T (x_{k+1}-x_k).
\end{align*}
Combined with the fact that $\lambda_k\in \mathbb{R}^{n_\text{h}} \times \mathbb{R}^{n_\text{g}}_{\geq 0}$, this implies
\begin{equation*}
\lambda_{ki} \bar{g}_i(x_{k+1}) \leq (1-\alpha T) \lambda_{ki} \bar{g}_i(x_k),
\end{equation*}
for any $i \in \{1,2,\dots,n_\text{h}+n_\text{g}\}$. The result follows by noting that $\lambda_{ki} \bar{g}_i(x_k) \leq 0$ and $1-\alpha T\geq 0$.
\end{proof}
Lemma~\ref{Lem:discrete} implies that $\lambda_k \in D_{x_{k+1}}$, ensuring that $\lambda_k$ is a feasible candidate for \eqref{eq:dual2}, or \eqref{eq:dual}, at time $k+1$. Lemma~\ref{Lem:discrete} can therefore be viewed as the discrete-time version of Lemma~\ref{Lemma:cont1}. As in the continuous-time case, Lemma~\ref{Lem:discrete} will be of paramount importance for proving convergence.


We are now ready to prove Proposition~\ref{Prop:GD}. We will divide the proof into several smaller claims:

\begin{claim}\label{Claim:1}
Let the assumption of Proposition~\ref{Prop:GD} be satisfied. Then, 
the sequence $d(x_k)$ is monotonically increasing and bounded above by $f^*$.
\end{claim}
\begin{proof}
The fact that $d(x_k)$ is bounded above by $f^*$ follows from Lemma~\ref{Lem:discrete2} (see Appendix~\ref{App:d}). We note that due to Lemma~\ref{Lem:discrete}, the multiplier $\lambda_{k}$ is a feasible candidate for the dual \eqref{eq:dual} (or \eqref{eq:dual2}) at time $k+1$; that is, $\lambda_{k} \in D_{x_{k+1}}$. This means that
\begin{align*}
d(x_{k+1}) \geq l(x_{k+1},\lambda_k) - \frac{1}{2\alpha} |\nabla_x l(x_{k+1},\lambda_k)|^2.
\end{align*}
Due to the strong convexity of $l(\cdot, \lambda_k)$, for a fixed $\lambda_k$, it follows that
\begin{equation*}
l(x_{k+1},\lambda_k) \geq l(x_k,\lambda_k) + T \nabla_x l(x_k,\lambda_k)\T v_k + \frac{\mu}{2} T^2 |v_k|^2= l(x_k,\lambda_k) -T |v_k|^2 + \frac{\mu}{2} T^2 |v_k|^2.
\end{equation*}
Moreover, by using Taylor's theorem, we can relate the gradient $\nabla_x l(x_{k+1},\lambda_k)$ to the gradient $\nabla_x l(x_k,\lambda_k)$ in the following way:
\begin{equation*}
\nabla_x l(x_{k+1},\lambda_k)=\nabla_x l(x_k,\lambda_k) + T \Delta_x l(\xi_k,\lambda_k)~~ v_k,
\end{equation*}
where $\Delta_x l$ denotes the second derivative of $l$ with respect to $x$, and $\xi_k$ lies between $x_k$ and $x_{k+1}$. Hence, we obtain the following lower bound for $d(x_{k+1})$:
\begin{align*}
d(x_{k+1})\geq d(x_k) + \frac{T}{\alpha} v_k\T \Delta_x l~ v_k - \frac{T^2}{2\alpha} v_k\T (\Delta_x l)^2~~ v_k - T |v_k|^2 + \frac{\mu}{2} T^2 |v_k|^2,
\end{align*}
where the arguments of the Hessian $\Delta_x l(\xi_k,x_k)$ have been omitted to simplify notation. We note that the Hessian $\Delta_x l$ is positive definite due to the convexity of $l(\cdot, \lambda_k)$ and has eigenvalues that are lower bounded by $\mu$ and upper bounded by $L_l$. Moreover, the matrix $(\Delta_x l)^2$ has the same eigenvectors as $\Delta_x l$, which means that
\begin{equation*}
v_k\T \left( \Delta_x l T - \frac{1}{2} \Delta_x l^2 T^2 \right) v_k \geq |v_k|^2 \min_{s \in [\mu T, L_l T]} s - s^2/2.
\end{equation*}
It can be shown that this minimum is lower bounded by $\mu T ( 1- \mu T/2)$ as long as $T\leq 2/(L_l+\mu)$.\footnote{The choice $T=2/(L_l+\mu)$ corresponds to the maximizer of $\min_{s \in [\mu T, L_l T]} s-s^2/2$ with respect to $T$.} This yields
\begin{equation}
d(x_{k+1})\geq d(x_k) + \underbrace{T \left(1-\frac{\mu T}{2}\right) \left(\frac{\mu}{\alpha} - 1\right)}_{=c_1} |v_k|^2. \label{eq:claim1}
\end{equation} 
From $T\leq 2/(L_l+\mu)$ and $\alpha < \mu$ we conclude that $c_1>0$, which proves the claim.
\end{proof}

\begin{claim}\label{Claim:convV}
Let the assumptions of Proposition~\ref{Prop:GD} be satisfied. The velocity $(x_{k+1}-x_k)/T$ is guaranteed to converge \hchange{and satisfies}
\begin{align*}
    \min_{j\in \{0,1,\dots,k\}} |-\nabla f(x_j)+R_j|^2 \leq \frac{f^*-d(x_0)}{c_1 (k+1)}, \quad \forall k\geq 0, \quad \forall x_0\in \mathbb{R}^n,
\end{align*}
where $c_1=T(\mu/\alpha -1) (1-\mu T/2)>0$ is constant.
\end{claim}
\begin{proof}
The result follows from Claim~\ref{Claim:1} by expanding $d(x_{k+1})$ as a telescoping sum,
\begin{align*}
f^*\geq d(x_{k+1}) &= d(x_0) + \sum_{j=0}^{k} d(x_{j+1})-d(x_j) \\
&\geq d(x_0) + c_1\sum_{j=0}^{k} |v_j|^2,
\end{align*}
where \eqref{eq:claim1} has been used for the last step. \hchange{The fact that the sum of squares of $|v_k|$ is bounded implies convergence of $v_k$ to zero for large $k$. We further obtain}
\begin{align*}
    f^* &\geq d(x_0) + c_1 (k+1) \min_{j\in \{0,1,\dots,k\}} |-\nabla f(x_j) + R_j|^2,
\end{align*}
\hchange{which implies the desired inequality.}
\end{proof}

\hchange{In order to prove convergence of $x_k$ to $x^*$, we will consider modifications of \eqref{eq:fundProb}, where some inequality constraints are removed. The resulting optimal costs are denoted by}
\begin{equation}
f^*_I:=\min_{x\in \mathbb{R}^n} f(x) \quad \text{s.t.} \quad  h(x)=0, \quad g_i(x)\geq 0, \quad i\in I,\label{eq:fseq}
\end{equation}
\hchange{where $I$ is any subset of $\{1,\dots,n_\text{g}\}$. The minimum in \eqref{eq:fseq} is guaranteed to be attained, due to the assumptions on $f$ and $C$. It is clear that $f^*_{\{\}} \leq f^*_I \leq f^*$ and we will use $x_I^*$ to denote any minimizer of \eqref{eq:fseq} with $\lambda_I^*$ the corresponding multipliers that satisfy the Karush-Kuhn-Tucker conditions of \eqref{eq:fseq}.}

\begin{claim}\label{Claim:3}
Let the assumptions of Proposition~\ref{Prop:GD} be satisfied. Each level set $\{x\in \mathbb{R}^n~|~d(x)\geq f_{I}^*\}$, where $I$ is any subset of $\{1,2,\dots,n_\text{g}\}$, is closed, invariant and attractive. 

\hchange{Let these level sets be labelled in the order $S_0 \supset S_1 \supset \dots \supset S_q$, where $q\leq 2^{n_\text{g}}$ and where $S_0=\mathbb{R}^n$, $S_1$ corresponds to $d(x)\geq f_{\{ \} }^*$ and $S_q$ to $d(x)\geq f^*$. We further denote the $f_I^*$ corresponding to $S_j$ by $f_j^*$ for $j=1,\dots, q$, and therefore $f_1^*=f_{\{ \} }^* < f_2^* < \dots < f_q^*=f^*$. On each of these level sets, the velocity converges at a linear rate, that is, for any integer $j$ with $0\leq j < q$,}
\begin{equation*}
|v(x_k)|^2 \leq \frac{1}{c_1} (1-c_2 T)^{k-k_0} (f_{j+1}^*-d(x_{k_0})), \quad \forall k\geq k_0:~x_k, x_{k+1}\in S_j\setminus S_{j+1},
\end{equation*}
\hchange{where $c_1=T(\mu/\alpha -1) (1-\mu T/2)>0$ and $c_2=2 \alpha (1-\mu T/2) (\mu-\alpha)/(L_l-\alpha)$ are constant, and $0<c_2 T<1$.}
\end{claim}
\begin{proof}
We conclude from \citet[Theorem~1.17, p.~16]{RockafellarWets} that $d$ is upper semi-continuous, which means that the level sets $\{ x\in \mathbb{R}^n ~|~d(x)\geq f_I^*\}$ are closed. \hchange{The fact that these are invariant follows directly from Claim~\ref{Claim:1}. For proving attractiveness and obtaining the linear rate, we start from \eqref{eq:claim1} and apply the lower bound on $d(x_k)$ provided by Lemma~\ref{Lem:discrete2} (see Appendix~\ref{App:d}).} This yields
\begin{equation}
d(x_{k+1})\geq d(x_k) + \frac{c_1}{L_l/(2\alpha^2) (1-\alpha/L_l)} (f^*_{I_{x_k}}-d(x_k)). \label{eq:claimtmp}
\end{equation}
\hchange{We consider the dynamics on one of the level sets $S_j$, that is, $x_k \in S_j \setminus S_{j+1}$, where $0\leq j<q$. From Lemma~\ref{Lem:discrete2} and the fact that $v(x_k) \neq 0$ we infer that }
\begin{equation*}
    f_j^* \leq d(x_k) < f_{I_{x_k}}^* \qquad \Rightarrow \qquad f_{I_{x_k}}^* \geq f_{j+1}^*,
\end{equation*}
\hchange{ as long as $x_k$ remains on $S_j \setminus S_{j+1}$, where $f^*_0$ is defined as $-\infty$. This follows from the fact that there are only finitely many $f_i^*$ and therefore $f_{I_{x_k}}^*$ can only take on a finite number of values.
As a result, we obtain from \eqref{eq:claimtmp} that}
\begin{equation*}
    d(x_{k+1})\geq d(x_k) + \frac{c_1}{L_l/(2\alpha^2) (1-\alpha/L_l)} (f_{j+1}^*-d(x_k)) = d(x_k) + c_2 T (f_{j+1}^* - d(x_k))
\end{equation*}
\hchange{as long as $x_k \in S_j \setminus S_{j+1}$, where we have used the definition of $c_2$ in the second step. Subtracting $f_{j+1}^*$ on both sides and rearranging terms results in}
\begin{equation*}
    (f_{j+1}^*-d(x_{k})) \leq (1-c_2 T)^{k-k_0} (f_{j+1}^*-d(x_{k_0})),
\end{equation*}
\hchange{where $k_0$ refers to the first time instant for which $x_k \in S_j$.} \hchange{We verify that $c_2 T<1$  by noting that $c_2 T$ is monotonically increasing for $0< T \leq 2/(L_l+\mu)$ and therefore}
\begin{align*}
    c_2 T = 2\alpha T (1-\mu T/2) \frac{\mu-\alpha}{L_l-\alpha} \leq \frac{4 \alpha}{L_l+\mu} \frac{L_l}{L_l+\mu} \frac{\mu-\alpha}{L_l-\alpha} \leq \frac{4 \alpha \mu}{(L_l+\mu)^2} < \frac{4 \mu^2}{(L_l+\mu)^2} \leq 1,
\end{align*}
\hchange{where we have repeatedly used the fact that $0<T\leq 2/(L_l+\mu)$, $0<\alpha < \mu$, and $0<\mu \leq L_l$.} \hchange{This shows attractivity of the set $S_{j+1}$. In addition, we conclude from \eqref{eq:claim1}}
\begin{equation*}
    c_1 |v(x_k)|^2 \leq d(x_{k+1})-d(x_k) \leq d(x_{k+1})-f_{j+1}^* + f_{j+1}^*-d(x_k) \leq f_{j+1}^*-d(x_k),
\end{equation*}
\hchange{(as long as $x_k, x_{k+1} \in S_j\setminus S_{j+1}$) which, in view of the exponential convergence of $f_{j+1}^*-d(x_k)$, implies the desired result.}

\end{proof}

The last claim provides a geometrical picture of the convergence of \eqref{eq:disAlgorithm}. At each iteration $k$ the iterate $x_k$ is contained in one of the level sets $S_j=\{x \in \mathbb{R}^n ~|~d(x) \geq f_j^*\}$ and converges to the next smaller level set, $S_{j+1}$. \hchange{Claim~\ref{Claim:3} already guarantees that the convergence happens at least at the linear rate $1-c_2 T$. The next claim ensures that except for the level set $S_q=\{x\in \mathbb{R}^n~|~d(x)\geq f^*\}$ (which contains only of the single point $x^*$), the convergence in fact happens in finite time.}
\begin{claim} \label{Claim:disLast}
\hchange{Provided that the assumptions of Proposition~\ref{Prop:GD} are satisfied, the iterates $x_k$ converge to the minimizer of \eqref{eq:fundProb}. Moreover, there exists an integer $N$, large enough, such that}
\begin{equation*}
    \min_{j\in \{0,1,\dots,k\} } |x^*-x_k|^2 \leq \frac{L_l/\alpha -1}{c_1 \alpha (\mu-\alpha)} ~\frac{f^*-d(x_0)}{k+1}, \quad \forall k\geq N.
\end{equation*}
 \label{Claim:conv}
\end{claim}
\begin{proof}
\hchange{As in the proof of Claim~\ref{Claim:3} we order the level sets corresponding to $f_I^*$ as follows $S_1\supset S_2 \supset \dots \supset S_q$, where $q\leq 2^{n_\text{g}}$ and where $S_1$ corresponds to $d(x)\geq f_{\{ \} }^*$ and $S_q$ to $d(x)\geq f^*$. We start by proving that for any $j<q-1$, $x_k$ traverses $S_j\setminus S_{j+1}$ in finite time.}

\hchange{For the sake of contradiction, we assume that $x_i \in S_j$, but $x_i\not \in S_{j+1}$ for all $i>k$. This implies $d(x_i) < f_{j+1}^*$ and $d(x_i) \geq f^*_{j}$ for all $i \geq k$, where $f_{j}^*$ and $f_{j+1}^*$ are defined in Claim~\ref{Claim:3}. According to Claim~\ref{Claim:3}, $|v(x_i)|$ converges to zero at an exponential rate for all $i>k$, since $x_i$ stays in $S_{j}\setminus S_{j+1}$ for all $i>k$. Hence, $x_i$ is a Cauchy sequence and has therefore a limit in $\mathbb{R}^n$, which we call $\bar{x}$. Since $v(x_i) \in V_\alpha(x_i)$ for all $i\geq 0$ and $v(x_i)\rightarrow 0$, we conclude by continuity of $g$ and $h$ that $\bar{x}\in C$. The same reasoning as in the proof of Claim~\ref{Claim:CTconv} (continuous-time case) implies by upper semi-continuity of $d$ that $f(\bar{x})=\lim_{i\rightarrow \infty} d(x_i) \leq d(\bar{x}) = f(\bar{x})-|v(\bar{x})|^2/(2\alpha)$. This means that $v(\bar{x})=0$, $\bar{x}=x^*$, and therefore $d(x_i)\rightarrow d(x^*)=f^*$, which leads to the desired contradiction.}

\hchange{Thus, there exists a finite time instant $N$ where $x_k$ enters $S_{q-1}$, that is $x_k \in S_{q-1}$ for all $k\geq N$. According to Claim~\ref{Claim:3}, the velocity converges at an exponential rate for $k>N$. We infer that $x_k$ is a Cauchy sequence, repeat the same arguments as above, and conclude that $x_k$ converges to $x^*$.}

\hchange{We further note that Lemma~\ref{Lem:discrete2} and Lemma~\ref{Lemma:propD} (see Appendix~\ref{App:d}) imply the following bound}
\begin{equation*}
    |x_k-x_{I_{x_k}}^*|^2 \leq \frac{2}{\mu-\alpha} (f^*_{I_{x_k}}-d(x_k)) \leq \frac{2}{\mu - \alpha} \frac{L_l/\alpha -1}{2 \alpha} |v_k|^2,
\end{equation*}
\hchange{which implies that $x_{I_{x_k}}^*$ converges to $x^*$. However, $I_{x_k}$ can only take on a finite number of values and therefore $I_{x_k} \rightarrow I_{x^*}$ in finite time. This means that}
\begin{equation*}
    |x_k-x^*|^2 \leq  \frac{L_l/\alpha -1}{\alpha (\mu - \alpha) } ~|v_k|^2, \quad \forall k\geq N,
\end{equation*}
\hchange{where $N$ is a sufficiently large integer. Applying the result from Claim~\ref{Claim:convV} concludes the proof.}
\end{proof}

\hchange{We note that Lemma~\ref{Lem:discrete2} and Lemma~\ref{Lemma:propD} (see Appendix~\ref{App:d}) relate the velocity $|v_k|$ of the iterates to the distance $|x_k-x^*|$. As a result, a similar argument as used for Claim~\ref{Claim:conv} ensures that the convergence of $|x_k-x^*|$ occurs asymptotically at a linear rate.} \hchangeII{The details are included in Appendix~\ref{App:AsymptoticRate}. Furthermore, we believe that with a more careful analysis the dependence of the integer $N$ in Claim~\ref{Claim:disLast} on problem specific parameters, such as $f^*_j$, $j=1,2,\dots,q$, can be explicitly quantified.}

With this geometrical picture in mind, we will discuss two extensions of \eqref{eq:disAlgorithm}. The resulting trajectories can be shown to converge to the minimizer of \eqref{eq:fundProb} with the same arguments as used for Claim~\ref{Claim:1}~-~Claim~\ref{Claim:conv}.

\subsection{Extensions}\label{Sec:Extension}
The convergence proof hinges on the following two properties of \eqref{eq:disAlgorithm}: (i) the multiplier $\lambda_k$ is feasible for the dual \eqref{eq:dual} at time $k+1$, and (ii) the function $l(x_{k},\lambda_{k})-|\nabla_x l(x_k,\lambda_k)|^2/2\alpha$ increases sufficiently from $x_k$ to $x_{k+1}$ (for a fixed $\lambda_k$). We can therefore extend \eqref{eq:disAlgorithm} by including the following line-search mechanism:
\begin{equation*}
T:=\argmax_{\tau>0, \alpha \tau \leq 1} l(x_k + \tau v_k,\lambda_k) - \frac{1}{2\alpha} |\nabla_x l(x_k + \tau v_k,\lambda_k)|^2,\qquad x_{k+1} = x_k + T v_k,
\end{equation*}
where the velocity $v_k$ is determined by solving \eqref{eq:vdef}, as before. As an alternative, we can alternate between updating $\lambda_k$ via \eqref{eq:dual} and applying gradient steps (with $\lambda_k$ fixed):
\begin{equation*}
x_{j+1} = x_j - T \nabla_x l(x_j,\lambda_k),\quad j=k,k+1,\dots,
\end{equation*}
as long as $g_i(x_{j+1})\leq 0$ for all $i\in I_{x_k}$ with corresponding multipliers $\lambda_{ik} > 0$ (constraints that were active and had a nonzero multiplier $\lambda_k$ at time $k$ are not allowed to open up). As is immediate from the arguments of Claim~\ref{Claim:1}, each of these gradient steps increases $l(x,\lambda_k)-|\nabla_x l(x,\lambda_k)|^2/(2\alpha)$ by $c_1 |\nabla l_x (x_j,\lambda_k)|^2$. Evaluating $\nabla_x l$ for a fixed $\lambda_k$ is computationally cheap and requires only the evaluation of $\nabla f$ and $\nabla \bar{g}(x)$. 
\section{\hchange{Motivation and Background}}\label{Sec:Motivation}
The continuous-time formulation given in Proposition~\ref{Prop:GF} can be motivated by drawing analogies to non-smooth mechanics. \hchange{This not only provides additional intuition for the algorithms that are discussed herein, but also allows for generalizations to accelerated first-order methods or Newton-type methods when constraints are incorporated on a velocity level.} We will start by viewing the stationarity conditions of \eqref{eq:fundProb} as the static equilibrium of a mechanical system. We will then apply d'Alembert's principle~\citep[see, e.g.,][]{Lanczos}, which relates this variational characterization of equilibria to the variational characterization of motion. In the context of optimization, this leads to the algorithm \eqref{eq:posLeveltmp}. We further note that the equivalence between \eqref{eq:posLeveltmp} and \eqref{eq:velLeveltmp} can be related to the equivalence between the principle of virtual work and the principle of virtual power in the context of mechanics.

We consider a mechanical system that consists of a point mass located at $x\in \mathbb{R}^n$ on which the external force $F:=-\nabla f(x)$ acts. The point mass is constrained to the set $C$.\footnote{From a physical perspective the constraint can be thought of as a second rigid body with infinite mass that consists of all points $\mathbb{R}^n\setminus C$. We seek to model the interaction between the point mass and the constraint.} For a given $\bar{x}\in C$ we start by investigating whether the point mass is in static equilibrium; i.e., it does not move under the influence of the external force and the constraint $x\in C$. In order to do so, we isolate the point mass and replace the interaction with the constraint by a (constraint) force, $-R \in N_C(\bar{x})$. The corresponding graphical procedure, often referred to as free-body diagram, is illustrated in Figure~\ref{Fig:fbd}. The principle of virtual work, which is the fundamental postulate of classical mechanics, can now be stated.
\begin{postulate}\label{pos:virtWork}
The point mass is in static equilibrium if and only if the virtual work vanishes for any virtual displacement $\delta x\in \mathbb{R}^n$. The \emph{virtual work} is defined as $(F+R)\T \delta x$, where $F$ is the external force and $R\in -N_C(x)$ the constraint force.
\end{postulate}
Due to the fact that arbitrary virtual displacements are allowed, Postulate~\ref{pos:virtWork} concludes that the point mass is in static equilibrium at $\bar{x}\in C$ if the following conditions are fulfilled:
\begin{equation}
-\nabla f(\bar{x})+R=0, \quad -R\in N_C(\bar{x}). \label{eq:KKTtmp}
\end{equation}
By virtue of the constraint qualification, these are equivalent to the Karush-Kuhn-Tucker conditions of \eqref{eq:fundProb}. Thus, with our choice $F:=-\nabla f(x)$, we can relate the stationarity conditions of \eqref{eq:fundProb} to the static equilibrium of a mechanical system, as characterized by the principle of virtual work. 

\begin{figure}
\def\svgwidth{\linewidth}
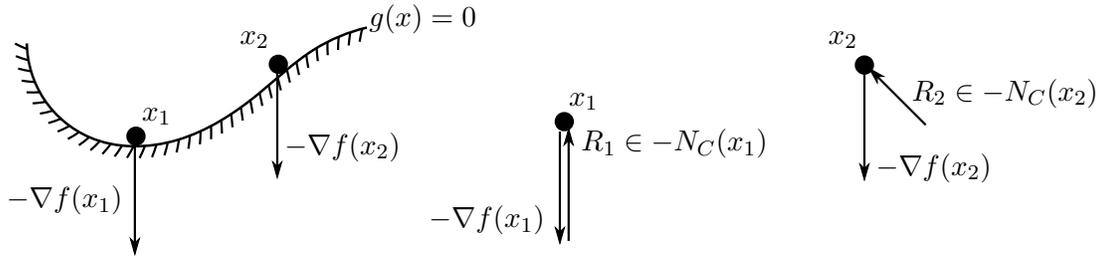
\caption{The figure illustrates the concept of a free-body diagram, where the geometric boundary condition $g(x)\geq 0$, as shown on the left, is replaced by the constraint forces, $-R\in N_C(x)$, as shown on the right. We note that $x_1$ is in static equilibrium, since $-\nabla f(x_1)$ and $R_1$ cancel, whereas $x_2$ is not.}
\label{Fig:fbd}
\end{figure}

The connections to optimization are even more explicit when restricting ourselves to admissible virtual displacements; i.e., $\delta x\in T_C(\bar{x})$. By definition, constraint forces satisfy $-R\T \delta x \leq 0$ for all $\delta x \in T_C(\bar{x})$ or, in the language of classical mechanics, constraint forces are such that their contribution to the virtual work is nonnegative.\footnote{In most classical textbooks only equality constraints are considered. In that case, constraint forces perform no virtual work.} This leads to the principle of d'Alembert-Lagrange, which represents the cornerstone of Lagrangian mechanics.
\begin{corollary}\label{Cor:DAlembertLagrange}
If the point mass located at $\bar{x}\in C$ is in static equilibrium, the virtual work of the external forces satisfies $F\T \delta x \leq 0$ for all admissible variations $\delta x\in T_C(\bar{x})$.
\end{corollary}
Through the lens of optimization, this means that $-\delta f=-\nabla f(\bar{x})\T \delta x \leq 0$ for all admissible variations $\delta x\in T_C(\bar{x})$, or equivalently, $f(\bar{x})\leq f(x)$ for all $x$ in an open neighborhood of $\bar{x}$ with $x\in C$. The relations are summarized in Figure~\ref{Fig:relation} (left).

\begin{figure}
\begin{minipage}[l]{.5\columnwidth} \scalebox{0.75}{\begin{tikzpicture}
  \matrix (m) [matrix of nodes,row sep=3em,column sep=3em,minimum width=2em,
 column 1/.style={nodes={text width=8em,align=center}}
  ]
  {
  	 optimization & mechanics (static eq.)\\[-3em]
     stationarity & p. of d'Alembert-Lagrange \\
     {$-\nabla f(\bar{x}) + R =0$\\ $-R\in N_C(\bar{x})$} & p. of virtual work \\
     KKT \\ };
   \path[>=stealth]
    (m-2-1.east|-m-2-2) edge [<->]  (m-2-2)
    (m-2-1) edge [<->]  (m-3-1)
    (m-3-1.east|-m-3-2) edge [<->] (m-3-2)
    (m-3-2) edge [<->] (m-2-2)
	(m-4-1) edge [<->] (m-3-1);    
\end{tikzpicture}}
\end{minipage}\hfill
\begin{minipage}[r]{.5\columnwidth} \scalebox{0.75}{\begin{tikzpicture}
  \matrix (m) [matrix of nodes,row sep=3em,column sep=3em,minimum width=2em,
  column 1/.style={nodes={text width=10em,align=center}}
  ]
  {
  	 optimization & mechanics (static eq.)\\[-3em]
     stationarity & p. of d'Alembert-Lagrange \\
     {$-\nabla f(\bar{x}) + R=0$ \\ $\displaystyle{-R \in \limsup_{x\rightarrow_C \bar{x}}N_C(x)}$ } & p. of virtual work \\
     KKT \\ };
   \path[>=stealth]
    (m-2-1) edge [->] (m-3-1)
    (m-3-1.east|-m-3-2) edge [<->] (m-3-2)
    (m-4-1) edge [->] (m-3-1);
   \draw (m-2-2.north west) -- (m-2-2.south east);
   \draw (m-2-2.south west) -- (m-2-2.north east);
\end{tikzpicture}}
\end{minipage}
\caption{The figure summarizes the analogies between constrained optimization and non-smooth mechanics. On the left, constraint qualifications are assumed to hold ensuring that the set $C$ is regular in the sense of Clarke. On the right, the set $C$ fails to be regular, for example due to a reintrant (inward facing) corner. In that case, the notion of equilibrium needs to be extended by an appropriate closure of $N_C(x)$; see, for example, \citet[Ch.~6]{RockafellarWets}. The resulting equilibrium condition is no longer sufficient for stationarity and its equivalence to the Karush-Kuhn-Tucker conditions breaks down~\citep[Thm.~6.14]{RockafellarWets}. Moreover, the principle of d'Alembert-Lagrange is no longer a consequence of the principle of virtual work and therefore fails to characterize static equilibria when $C$ is not regular \citep{Panagiotopoulous}. There are important examples of mechanical systems where $C$ fails to be regular; see, for example, \citet[Ch.~11]{Glocker}. }
\label{Fig:relation}
\end{figure}
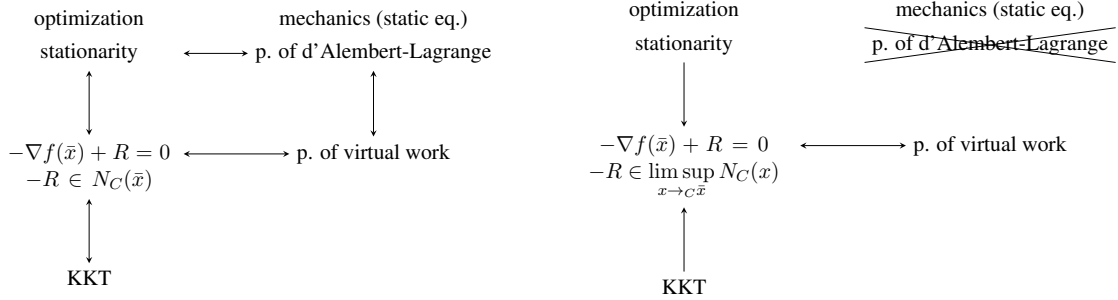

The important insight from classical mechanics (essentially due to d'Alembert) is that the principle of virtual work, Postulate~\ref{pos:virtWork}, and the principle of d'Alembert-Lagrange, Corollary~\ref{Cor:DAlembertLagrange}, naturally extend from the static equilibrium to the dynamic equilibrium that characterizes the motion of a mechanical system. It suffices to add the ``forces of inertia,'' which for the point mass amounts to adding $-m\ddot{x}$ to the external forces $F$ \citep[Ch.~4]{Lanczos}. We will apply these ideas to gradient flow, where the ``forces of inertia" are given by $-\dot{x}$. This yields \eqref{eq:posLeveltmp}, which we restate as follows:
\begin{equation*}
\dot{x}(t)=-\nabla f(x(t)) + R(t), \quad -R(t)\in N_C(x(t)), \quad \forall t\in [0,\infty) ~~\text{a.e.} 
\end{equation*}


The condition \eqref{eq:posLeveltmp} can still be viewed as a force balance between $\dot{x}(t)+\nabla f(x(t))$ and $R(t)$, whereby the reaction force $R(t)$ ensures that $x(t)$ remains feasible. Moreover, when the system is at rest, $\dot{x}$ vanishes and \eqref{eq:posLeveltmp} reduces to the Karush-Kuhn-Tucker conditions of \eqref{eq:fundProb} (see \eqref{eq:KKTtmp}). If $x(t)$ happens to be in the interior of $C$, the reaction force $R(t)$ vanishes, and $x(t)$ evolves according to unconstrained gradient flow. The almost everywhere quantifier is clearly needed---if $x(t)$ approaches the boundary of the set $C$, an instantaneous velocity jump might be required for ensuring that $x(t)$ remains in $C$ (at the time instant of the velocity jump, $\dot{x}$ is no longer defined).

\hchange{Thus, the analogy to non-smooth mechanics not only motivates \eqref{eq:posLeveltmp}, which is used as a starting point for all the derivations in this article, but also enables different choices for the ``forces of inertia." As a result, by applying the principle of virtual work, as stated in Postulate~\ref{pos:virtWork}, one could derive momentum-based or Newton-type algorithms that include constraints on velocity level. A thorough exploration of these extensions is an important topic for future work.}


\section{Computational Aspects}\label{Sec:FixedPoint}
This section highlights two important aspects of the implementation of the discrete-time algorithm \eqref{eq:disAlgorithm}: (i) the computation of the constraint force $R_k=\nabla \bar{g}(x_k)\lambda_k$, and (ii) control of round-off errors and inaccuracies.

\subsection{Computing the constraint force $R_k$}
The constraint forces are determined by the dual problem \eqref{eq:dual}, which can be solved with various algorithms. The simple nature of the set $D_{x_k}$ makes (accelerated) projected gradient descent schemes appealing. In the following, we present a procedure that is inspired by the method of successive over-relaxation which solves \eqref{eq:dual} very efficiently. The procedure is useful for solving large linear complementary problems and is commonly used in the non-smooth mechanics community \citep[see, e.g.,][]{Studer}. For completeness, we give a rough overview of the main points and refer the reader to the work of \citet{Cottle} for further details. The stationarity conditions of \eqref{eq:dual} are given by
\begin{equation}\label{eq:statdual}
W_k\T W_k \lambda_k - W_k\T \nabla f(x_k) + \alpha \bar{g}(x_k) + \partial \psi_{D_{x_k}}(\lambda_k) \ni 0,
\end{equation}
where we use the notation introduced in Algorithm~\ref{Alg:one}.\footnote{Compared to the notation in \eqref{eq:dual}, for example, we exclude all multipliers $\lambda_i$ that correspond to inactive inequality constraints; that is, $i\not\in I_{x_k}$.} The underlying idea relies on a suitable splitting of the matrix $W_k\T W_k$ that enables fixed-point iteration. We introduce $\lambda^j$ as the approximation of $\lambda_k$ at iteration $j$, $j=0,1,\dots$ and further suppress the subscript $k$ for ease of notation. We denote the strictly upper triangular part of $W\T W$ by $U$ and the diagonal by $D$. The matrix $W\T W$ is therefore given by $U\T+D+U$, where the diagonal elements are guaranteed to be strictly positive.\footnote{\hchange{The diagonal elements of $D$ are given by $|\nabla h_i(x)|^2$, $i=1,\dots n_\text{h}$ and $|\nabla g_i(x)|^2$, $i\in I_x$. Due to the constraint qualification, these are are guaranteed to be strictly positive. }} We can split the matrix $W\T W$ into $U\T+\omega^{-1} D$ and $U + (1-\omega^{-1})D$, where $\omega \in (0,2)$ is fixed, leading to
\begin{equation}
    (U\T + \omega^{-1} D) \lambda^{j+1} + \partial \psi_{D_{x}}(\lambda^{j+1}) + (U + (1-\omega^{-1}) D) \lambda^{j} - W\T \nabla f(x) + \alpha \bar{g}(x)\ni 0, \label{eq:statdual2}
\end{equation}
where we have omitted the subscript $k$; hence, $x=x_k$, $W=W_k$, etc. The role of the variable $\omega$ as a tuning parameter will become apparent below. We note that \eqref{eq:statdual2} reduces to \eqref{eq:statdual} for $\lambda^{j+1}=\lambda^j$. As a result of the fact that $U\T$ is strictly lower triangular, \eqref{eq:statdual2} reduces to the following inclusion for a single component: $\lambda_i^{j+1}$
\begin{equation}
\omega^{-1} D_{ii} \lambda_i^{j+1} + \partial \psi_{\mathbb{R}}(\lambda_i^{j+1}) + *_i \ni 0, \quad \text{or} \quad 
\omega^{-1} D_{ii} \lambda_i^{j+1} + \partial \psi_{\mathbb{R}\geq 0}(\lambda_i^{j+1}) + *_i \ni 0, \label{eq:statSingle}
\end{equation}
depending on whether $i\leq n_\text{h}$ or $i>n_\text{h}$, where $*_i$ is a placeholder for all remaining terms that are constant or only depend on $\lambda^j$ and $\lambda^{j+1}_1, \dots \lambda^{j+1}_{i-1}$. The inclusion in \eqref{eq:statSingle} can be seen as a stationarity condition for $\lambda_i^{j+1}$, which uniquely determines $\lambda_i^{j+1}$ from $\lambda^{j}$ and $\lambda^{j+1}_1,\dots,\lambda^{j+1}_{i-1}$. We can therefore express \eqref{eq:statdual2} as
\begin{equation}
\lambda^{j+1} = \text{prox}_{D_{x}}\left(\lambda^j -\omega D^{-1} (U\T \lambda^{j+1} + (D+U) \lambda^j - W\T \nabla f(x) + \alpha \bar{g}(x))\right), \label{eq:statdual3}
\end{equation}
where $\text{prox}_{D_{x}}: \mathbb{R}^{n_\text{h}}\times \mathbb{R}^{|I_{x}|} \rightarrow \mathbb{R}^{n_\text{h}} \times \mathbb{R}^{|I_{x}|}_{\geq 0}$ is defined as
\begin{align*}
    (\text{prox}_{D_{x}} (\xi))_i &= \xi_i, \qquad\qquad~~~ i=1,\dots, n_\text{h}, \\ (\text{prox}_{D_{x}} (\xi))_i &= \max\{ \xi_i, 0\}, \quad i=n_\text{h}+1, \dots n_\text{h} + |I_{x}|,
\end{align*}
and where we have used the fact that $\omega^{-1} D_{ii}>0$. It is important to note that \eqref{eq:statdual3} provides an explicit rule for computing $\lambda^{j+1}$ from $\lambda^{j}$, since $U\T$ is strictly lower triangular. In particular, by substituting the newly computed elements $\lambda^{j+1}$ directly in the right-hand side of \eqref{eq:statdual3}, i.e., overwriting $\lambda_i^j$ with $\lambda_i^{j+1}$ as soon as it becomes available, the expression on the right-hand side of \eqref{eq:statdual3} reduces to
\begin{equation*}
\text{prox}_{D_{x}} \left(\lambda^j - \omega D^{-1} (W\T W \lambda^j - W\T \nabla f(x) + \alpha \bar{g}(x)) \right),
\end{equation*} 
which becomes very convenient for algorithmic implementation. The expression \eqref{eq:statdual3} can therefore be viewed as an extension of the method of successive over-relaxation that accounts for the complementary slackness induced by the inequality constraints. The method reduces to a variant of the Gauss-Seidel method for $\omega=1$. The following proposition due to \citet[p.~400]{Cottle} ensures convergence of the $\lambda^{j} \rightarrow \lambda_k$ as long as $\omega\in (0,2)$. The proof follows \citet[p.~400]{Cottle} and is included in Appendix~\ref{App:ProofLambda} for completeness.

\begin{proposition}\label{Prop:ConvergenceLambda}
\citet[p.~400]{Cottle} The sequence $\lambda^j$, defined according to \eqref{eq:statdual3}, converges for $\omega\in (0,2)$. The resulting multiplier $\lim_{j\rightarrow\infty} \lambda^j=\lambda_k$ satisfies \eqref{eq:statdual} and therefore maximizes \eqref{eq:dual}.
\end{proposition}
In our numerical experiments, the choice $\omega=1$ (the Gauss-Seidel variant) yielded good results.

\subsection{Dealing with round-off errors and inaccuracies in the computation of $R_k$}
In Section~\ref{Sec:Example} and Section~\ref{Sec:Discrete} we noted that the minimizer of \eqref{eq:fundProb} is typically not a stable equilibrium in the sense of Lyapunov for \eqref{eq:disAlgorithm}. If we revisit the example of Section~\ref{Sec:Example} we realize that a trajectory initialized at $x_0=\epsilon>0$, where $\epsilon>0$ is arbitrarily small, will make a relatively large step to $x_1<0$ before approaching the origin from $x_k<0$. Thus, if we set the constraint force $R_k$ to be slightly too large by mistake, when approaching the origin from $x_k<0$, this might push $x_k$ again to positive values ($x_k>0$), at which point the cycle would start again. For a practical implementation of \eqref{eq:disAlgorithm}, it is therefore important to address and discuss the effect of round-off errors and inexact computations of $R_k$. 

We can address the problem with a combination of the following two strategies:
\paragraph{(i) Slightly extending the infeasible set:} We extend the set $I_{x}$ to $\{i\in \mathbb{Z}~|~g_i(x)\leq \epsilon_\text{g}\}$, where $\epsilon_\text{g}>0$ is a user-specified tolerance for constraint satisfaction. Provided that $x^*$, the minimizer of \eqref{eq:fundProb}, lies on the boundary of the feasible set, this has the effect that in a neighborhood about $x^*$ inequality constraints are treated as equality constraints, which prevents $x_k$ from cycling even in the presence of round-off errors and inexact computations of $R_k$. We illustrate the situation with the example of Section~\ref{Sec:Example}, where Figure~\ref{fig:simpExample2} shows the gradient $\nabla_x l$. The introduction of the parameter $\epsilon_\text{g}$ slightly extends the infeasible region, and moves the discontinuity of $\nabla_x l$ from $x^*$ to $x^*+\epsilon_\text{g}$. This renders the origin stable in the sense of Lyapunov and therefore mitigates the effect of small round-off errors and slight inaccuracies in the computation of $R_k$.

\paragraph{(ii) Adapting the stopping criteria of \eqref{eq:statdual3}:} In continuous time, the complementary slackness states that $\lambda_i>0$ implies $\diff  g_i(x(t))/\dt +\alpha g_i(x(t))=0$ (constraint $i$ remains active), whereas $\diff g_i(x(t))/\dt +\alpha g_i(x(t)) \geq 0$ for $\lambda_i=0$ (constraint $i$ might open up). Since we are solving the complementary slackness conditions only approximately, it might happen that even for $\lambda_i>0$, $\diff g(x(t))_i /\dt$ becomes too large such that the constraint incorrectly opens up in the next iteration of our discrete approximation. This can be avoided by stopping the iteration \eqref{eq:statdual3} only if for each inequality constraint $i$ with $\lambda_i>0$, we have
\begin{equation}
   \underbrace{(W_k\T W_k \lambda -W_k\T \nabla f(x_k))_i}_{\approx \diff g_i(x(t)) / \dt} + \alpha g_i(x_k) \leq \epsilon_\text{g} \alpha/2.  \label{eq:stopping2}
\end{equation}
For convex constraints ($g$ is concave) this  inequality ensures that
\begin{equation*}
    g_i(x_{k+1}) \leq (1 - \alpha T) g_i(x_k) + \epsilon_\text{g} \alpha T/2,
\end{equation*}
for all constraints where the corresponding multiplier $\lambda_i$ is strictly positive. The fact that $g_i(x_k)\leq \epsilon_\text{g}$ (see point (i) above) and $0 < \alpha T \leq 1$ guarantees that $g_i(x_{k+1})< \epsilon_\text{g}$, which means that the constraint remains active.

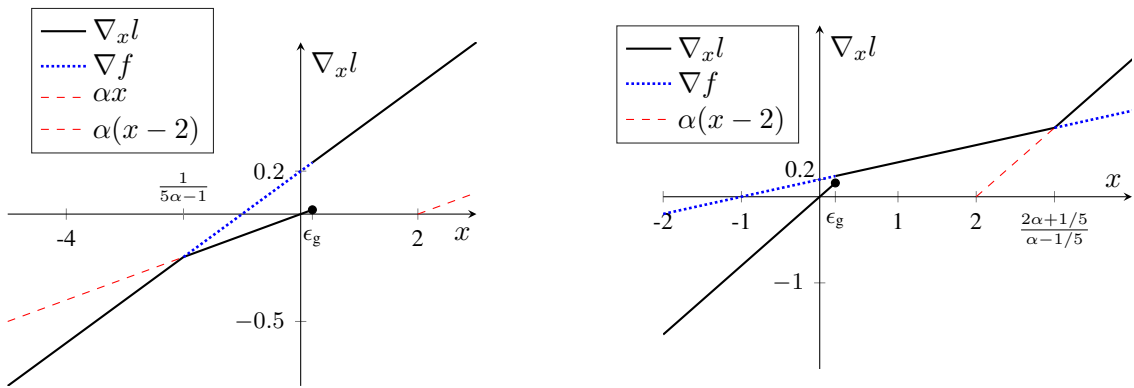
\begin{figure}
    
    \setlength{\figurewidth}{.43\columnwidth}
    \setlength{\figureheight}{.3\columnwidth}
    
    \begin{minipage}[l]{.45\columnwidth}
    \centering
%
%
\begin{tikzpicture}

\begin{axis}[%
width=0.951\figurewidth,
height=\figureheight,
at={(0\figurewidth,0\figureheight)},
scale only axis,
xmin=-5,
xmax=3,
xlabel style={font=\color{white!15!black}},
xtick={-4,2,0.2},
xticklabels={-4,2,$\epsilon_\text{g}$},
extra x ticks={-2},
extra x tick style={
    xticklabel style={yshift=0.5ex, anchor=south}
},
extra x tick labels={$\frac{1}{5\alpha -1}$},
ytick={-0.5,0.2},
xlabel={$x$},
xlabel near ticks,
ymin=-0.8,
ymax=0.8,
ylabel style={font=\color{white!15!black}},
ylabel={$\nabla_x l$},
axis background/.style={fill=white},
axis x line=middle,
axis y line=middle,
xlabel style={at={(axis description cs:0.97,0.49)},anchor=north},
legend style={legend cell align=left, align=left, draw=white!15!black},
legend style={at={(0.05,1.1)},anchor=north west}
]
\addplot [color=black,thick]
  table[row sep=crcr]{%
-5	-0.8\\
-4.9	-0.78\\
-4.8	-0.76\\
-4.7	-0.74\\
-4.6	-0.72\\
-4.5	-0.7\\
-4.4	-0.68\\
-4.3	-0.66\\
-4.2	-0.64\\
-4.1	-0.62\\
-4	-0.6\\
-3.9	-0.58\\
-3.8	-0.56\\
-3.7	-0.54\\
-3.6	-0.52\\
-3.5	-0.5\\
-3.4	-0.48\\
-3.3	-0.46\\
-3.2	-0.44\\
-3.1	-0.42\\
-3	-0.4\\
-2.9	-0.38\\
-2.8	-0.36\\
-2.7	-0.34\\
-2.6	-0.32\\
-2.5	-0.3\\
-2.4	-0.28\\
-2.3	-0.26\\
-2.2	-0.24\\
-2.1	-0.22\\
-2	-0.2\\
-1.9	-0.19\\
-1.8	-0.18\\
-1.7	-0.17\\
-1.6	-0.16\\
-1.5	-0.15\\
-1.4	-0.14\\
-1.3	-0.13\\
-1.2	-0.12\\
-1.1	-0.11\\
-1	-0.1\\
-0.9	-0.09\\
-0.8	-0.08\\
-0.7	-0.07\\
-0.6	-0.06\\
-0.5	-0.05\\
-0.4	-0.04\\
-0.3	-0.03\\
-0.2	-0.02\\
-0.1	-0.01\\
0	0\\
0.1	0.01\\
0.2	0.02\\
};
\addlegendentry{$\nabla_x l$}

\addplot [color=black,thick,forget plot]
  table[row sep=crcr]{%
0.201	0.2402\\
0.301	0.2602\\
0.401	0.2802\\
0.501	0.3002\\
0.601	0.3202\\
0.701	0.3402\\
0.801	0.3602\\
0.901	0.3802\\
1.001	0.4002\\
1.101	0.4202\\
1.201	0.4402\\
1.301	0.4602\\
1.401	0.4802\\
1.501	0.5002\\
1.601	0.5202\\
1.701	0.5402\\
1.801	0.5602\\
1.901	0.5802\\
2.001	0.6002\\
2.101	0.6202\\
2.201	0.6402\\
2.301	0.6602\\
2.401	0.6802\\
2.501	0.7002\\
2.601	0.7202\\
2.701	0.7402\\
2.801	0.7602\\
2.901	0.7802\\
3.001	0.8002\\
};

\addplot [color=black,only marks, mark size=1.5pt,forget plot]
 table[row sep=crcr]{%
 0.2 0.02\\
 };

\addplot [color=blue, line width=1pt, densely dotted]
  table[row sep=crcr]{%
-2	-0.2\\
-1.9	-0.18\\
-1.8	-0.16\\
-1.7	-0.14\\
-1.6	-0.12\\
-1.5	-0.1\\
-1.4	-0.08\\
-1.3	-0.06\\
-1.2	-0.04\\
-1.1	-0.02\\
-1	0\\
-0.9	0.02\\
-0.8	0.04\\
-0.7	0.06\\
-0.6	0.08\\
-0.5	0.1\\
-0.4	0.12\\
-0.3	0.14\\
-0.2	0.16\\
-0.1	0.18\\
0	0.2\\
0.1	0.22\\
0.2	0.24\\
};

\addlegendentry{$\nabla f$}
\addplot [color=red, dashed]
  table[row sep=crcr]{%
-5	-0.5\\
-4.9	-0.49\\
-4.8	-0.48\\
-4.7	-0.47\\
-4.6	-0.46\\
-4.5	-0.45\\
-4.4	-0.44\\
-4.3	-0.43\\
-4.2	-0.42\\
-4.1	-0.41\\
-4	-0.4\\
-3.9	-0.39\\
-3.8	-0.38\\
-3.7	-0.37\\
-3.6	-0.36\\
-3.5	-0.35\\
-3.4	-0.34\\
-3.3	-0.33\\
-3.2	-0.32\\
-3.1	-0.31\\
-3	-0.3\\
-2.9	-0.29\\
-2.8	-0.28\\
-2.7	-0.27\\
-2.6	-0.26\\
-2.5	-0.25\\
-2.4	-0.24\\
-2.3	-0.23\\
-2.2	-0.22\\
-2.1	-0.21\\
-2	-0.2\\
};
\addlegendentry{$\alpha x$}

\addplot [color=red, dashed]
  table[row sep=crcr]{%
2	0\\
2.1	0.01\\
2.2	0.02\\
2.3	0.03\\
2.4	0.04\\
2.5	0.05\\
2.6	0.06\\
2.7	0.07\\
2.8	0.08\\
2.9	0.09\\
3	0.1\\
3.1	0.11\\
};
\addlegendentry{$\alpha (x-2)$}

\end{axis}
\end{tikzpicture}%
    \end{minipage}\hfill
    \begin{minipage}[r]{.45\columnwidth}
    \centering
%
%
\begin{tikzpicture}

\begin{axis}[%
width=0.951\figurewidth,
height=\figureheight,
at={(0\figurewidth,0\figureheight)},
scale only axis,
xmin=-2,
xmax=3.99,
xlabel style={font=\color{white!15!black}},
xlabel={$x$},
xtick={-2,-1,0.2,1,2,3},
xticklabels={-2,-1,$\epsilon_\text{g}$,1,2,$\frac{2\alpha + 1/5}{\alpha -1/5}$},
ytick={-1},
extra y ticks={0.2},
extra y tick style={
    yticklabel style={xshift=-3ex, yshift=.8ex, anchor=west}
},
extra y tick labels={0.2},
ymin=-2,
ymax=1.99,
ylabel style={font=\color{white!15!black}},
ylabel={$\nabla_x l$},
axis background/.style={fill=white},
axis x line=middle,
axis y line=middle,
legend style={legend cell align=left, align=left, draw=white!15!black},
legend style={at={(-0.1,1.0)},anchor=north west}
]
\addplot [color=black,thick]
  table[row sep=crcr]{%
-2	-1.6\\
-1.9	-1.52\\
-1.8	-1.44\\
-1.7	-1.36\\
-1.6	-1.28\\
-1.5	-1.2\\
-1.4	-1.12\\
-1.3	-1.04\\
-1.2	-0.96\\
-1.1	-0.88\\
-1	-0.8\\
-0.9	-0.72\\
-0.8	-0.64\\
-0.7	-0.56\\
-0.6	-0.48\\
-0.5	-0.4\\
-0.4	-0.32\\
-0.3	-0.24\\
-0.2	-0.16\\
-0.1	-0.08\\
0	0\\
0.1	0.08\\
0.2	0.16\\
};
\addlegendentry{$\nabla_x l$}

\addplot [color=black,thick,forget plot]
  table[row sep=crcr]{%
0.201	0.2402\\
0.301	0.2602\\
0.401	0.2802\\
0.501	0.3002\\
0.601	0.3202\\
0.701	0.3402\\
0.801	0.3602\\
0.901	0.3802\\
1.001	0.4002\\
1.101	0.4202\\
1.201	0.4402\\
1.301	0.4602\\
1.401	0.4802\\
1.501	0.5002\\
1.601	0.5202\\
1.701	0.5402\\
1.801	0.5602\\
1.901	0.5802\\
2.001	0.6002\\
2.101	0.6202\\
2.201	0.6402\\
2.301	0.6602\\
2.401	0.6802\\
2.501	0.7002\\
2.601	0.7202\\
2.701	0.7402\\
2.801	0.7602\\
2.901	0.7802\\
3.001	0.8008\\
3.101	0.8808\\
3.201	0.9608\\
3.301	1.0408\\
3.401	1.1208\\
3.501	1.2008\\
3.601	1.2808\\
3.701	1.3608\\
3.801	1.4408\\
3.901	1.5208\\
4.001	1.6008\\
4.101	1.6808\\
4.201	1.7608\\
4.301	1.8408\\
4.401	1.9208\\
4.501	2.0008\\
4.601	2.0808\\
4.701	2.1608\\
4.801	2.2408\\
4.901	2.3208\\
5.001	2.4008\\
};

\addplot [color=black,only marks, mark size=1.5pt,forget plot]
 table[row sep=crcr]{%
 0.2 0.16\\
 };

\addplot [color=blue, line width=1pt, densely dotted]
  table[row sep=crcr]{%
-2	-0.2\\
-1.9	-0.18\\
-1.8	-0.16\\
-1.7	-0.14\\
-1.6	-0.12\\
-1.5	-0.1\\
-1.4	-0.08\\
-1.3	-0.06\\
-1.2	-0.04\\
-1.1	-0.02\\
-1	0\\
-0.9	0.02\\
-0.8	0.04\\
-0.7	0.06\\
-0.6	0.08\\
-0.5	0.1\\
-0.4	0.12\\
-0.3	0.14\\
-0.2	0.16\\
-0.1	0.18\\
0	0.2\\
0.1	0.22\\
0.2	0.24\\
};
\addlegendentry{$\nabla f$}

\addplot [color=blue, line width=1pt, densely dotted,forget plot]
  table[row sep=crcr]{%
3	0.8\\
3.1	0.82\\
3.2	0.84\\
3.3	0.86\\
3.4	0.88\\
3.5	0.9\\
3.6	0.92\\
3.7	0.94\\
3.8	0.96\\
3.9	0.98\\
4	1\\
4.1	1.02\\
4.2	1.04\\
4.3	1.06\\
4.4	1.08\\
4.5	1.1\\
4.6	1.12\\
4.7	1.14\\
4.8	1.16\\
4.9	1.18\\
5	1.2\\
};

\addplot [color=red, dashed]
  table[row sep=crcr]{%
2	0\\
2.1	0.0800000000000001\\
2.2	0.16\\
2.3	0.24\\
2.4	0.32\\
2.5	0.4\\
2.6	0.48\\
2.7	0.559999999999999\\
2.8	0.64\\
2.9	0.72\\
3	0.8\\
};
\addlegendentry{$\alpha(x-2)$}

\end{axis}
\end{tikzpicture}%
    \end{minipage}
    \caption{This figure shows the values of $\nabla_x l$ (solid thick line) for $\alpha=1/10$ (left) and $\alpha=4/5$ (right), where $\epsilon_\text{g}=0.2$. We note that the discontinuity of $\nabla_x l$ is now at $\epsilon_\text{g}>0$, which means that the origin is an asymptotically stable equilibrium in the sense of Lyapunov. The parameter $\epsilon_\text{g}$ has no effect on the constraint $x\geq 2$. The original gradient $\nabla f$ is again shown in blue (dotted) and the functions $\alpha x$ and $\alpha (x-2)$ are shown in red (dashed).}
    \label{fig:simpExample2}
\end{figure}

Algorithm~\ref{Alg:two} summarizes the discussions of the two previous sections. The next section will be concerned with the empirical evaluation of Algorithm~\ref{Alg:two} on various examples. The exact implementation in Python and C++ is available as supplementary material.

\begin{algorithm}
\begin{algorithmic}
\Require $x_0 \in \mathbb{R}^n$, $T>0$, $\alpha T \in (0,1]$, $\epsilon_\text{g}>0$, $\omega\in (0,2)$,\\
\qquad TOL, MAXITER, MAXITER\_PROX, TOL\_PROX
\State $k=0$
\While{$k<\text{MAXITER}$}
\State Determine the set of closed constraints $I_{k}=\{i\in \mathbb{Z}~|~g_i(x_k)\leq \epsilon_\text{g}\}$ 
\vspace{-2pt}\State Define $W_k:=(\nabla h(x_k), \nabla g_i(x_k)_{i\in I_{k}})$ and $D_k:=\mathbb{R}^{n_\text{h}} \times \mathbb{R}_{\geq 0}^{|I_k|}$
\State Define $\bar{g}_k:=(h(x_k),g_i(x_k)_{i\in I_{k}})$ \medskip
\State $j=0$, $\lambda^0=0$ \Comment{initialization with $\lambda_{k-1}$ is also possible}
\While{$j<\text{MAXITER\_PROX}$}
\State $\lambda^{j+1} = \text{prox}_{D_{k}}\left(\lambda^j -\omega D^{-1} (U\T \lambda^{j+1} + \lambda^j - W_k\T \nabla f(x_k) + \alpha \bar{g}_k)\right)$
\If{$|\lambda^{j+1}-\lambda^{j}| \leq \text{TOL\_PROX}$ \textbf{and} $\forall i>n_\text{h}: \lambda_i>0,$\\ \hspace{3cm} $ (W_k\T W_k \lambda^{j+1} -W_k\T \nabla f(x_k))_i + \alpha \bar{g}_{ki} \leq \epsilon_\text{g} \alpha T/2,$}
\State \textbf{break}
\EndIf
\EndWhile
\State $\lambda_k=\lambda^{j+1}$ \medskip
\State Perform the update $x_{k+1} = x_k - T\nabla f(x_k) + T W_k \lambda_k$
\If{$|x_{k+1}-x_k| \leq T\cdot \text{TOL}$} 
\State \Return $x_{k+1}$
\EndIf
\State $k\leftarrow k+1$
\EndWhile
\end{algorithmic}
\caption{Implementation of the gradient descent scheme \eqref{eq:disAlgorithm}.}
\label{Alg:two}
\end{algorithm}

\section{Numerical Examples}\label{Sec:NumExamples}
The following section illustrates the application of Algorithm~\ref{Alg:two} to the following problems: (i) randomly generated quadratic programs, (ii) trust-region optimization, (iii) $\nu$-support vector machines, and (iv) the computation of a catenary subject to nonlinear constraints. The examples (i)-(iii) lead to convex quadratic programs or convex second-order cone programs, whereas the last example is a nonconvex problem. Algorithm~\ref{Alg:two} is implemented in C++ and we use pybind11 \citep{pybind} as a Python interface. The experiments were conducted on a Dell Precision Tower 3620 that runs Ubuntu 20.04LTS and is equipped with an Intel Core i7-6700 processor (8x3.4GHz) and 64GB of random access memory. All matrices were stored in compressed row storage for exploiting sparsity. The parameters of Algorithm~\ref{Alg:two} which were used for the experiments are summarized in Table~\ref{Tab:Params}.

\begin{table}
    \centering
    \begin{tabular}{l|c|c|c|c}
        Parameters & 1) Rand.QP & 2) Trust region & 3) $\nu$-SVM & 4) Catenary  %
         \\ \hline 
         $T$ &  $2/(L+\mu)$ &$2/(\bar{L}_l+\mu)$ &$2/(L+\mu)$ & $2/n$\\
         $\alpha T$ & 0.4 & 0.4 & 0.4 & 0.8\\
         $\epsilon_\text{g}$ & 1e-6 & 1e-6 & 1e-6 & 1e-6\\
         $\omega$ & 1 & 1 & 1 & 1\\
         TOL & 1e-6 & 1e-6 & 1e-6 & 1e-6\\
         MAXITER & 1000 & 1000 & 1000 & 10000\\
         MAXITER\_PROX & 200 & 200 & 200 & 10000\\
         TOL\_PROX & 1e-6 & 1e-6 & 1e-6 & 1e-8\\
    \end{tabular}
    \caption{Parameters of Algorithm~\ref{Alg:two} used for the experiments, where $L$ and $\mu$ refer to the smoothness and strong convexity constants of $f$. The variable $n$ denotes the number of chain links of the catenary, as defined in Section~\ref{Subsec:Cat}.}
    \label{Tab:Params}
\end{table}

We compare Algorithm~\ref{Alg:two} with the state-of-the-art interior-point solver CVXOPT \citep{CVXOPT} for larger problem instances. \hchange{We also show comparisons to projected gradients and Frank-Wolfe, where the projections and Frank-Wolfe updates are computed with the standard optimization library in Python's scientific computing library~\citep[scipy.optimize,][]{2020SciPy-NMeth}. This is motivated by the fact that the scipy optimization library is standard in Python and can solve optimization problems with nonlinear constraints, which parallels the capabilities of Algorithm~\ref{Alg:two}. For the projected gradients implementation we used the same step size as for Algorithm~\ref{Alg:two}, see Table~\ref{Tab:Params}, whereas the Frank-Wolfe implementation follows \cite{Jaeggi}, Algorithm~1.} \hchangeII{We further note that our stopping criterion ensures $|-\nabla f(x_k)+R_k| \leq \text{TOL}$; analogous stopping criteria are used for Frank-Wolfe and projected gradients. We also found that the resulting function values for the different algorithms agree with each other.}

\subsection{Randomly generated quadratic programs}\label{subsec:RandQP}
We generate quadratic programs of the following form:
\begin{align*}
    \min_{x \in \mathbb{R}^n} \quad &\frac{1}{2} x\T Q x + c\T x, \qquad
    \text{s.t.} \quad A_1 x + b_1 \geq 0, \quad A_2 x + b_2 =0,
\end{align*}
where the entries of $A_1\in \mathbb{R}^{n/2 \times n}, A_2\in \mathbb{R}^{n/4 \times n}$ and $b_1\in \mathbb{R}^{n/2}, b_2\in \mathbb{R}^{n/4}$ are independent samples from a normal distribution with zero mean and unit variance, the entries of $c$ are independent samples of a uniform distribution supported on $[-1,1]$, and $Q$ is a diagonal matrix. The first two diagonal elements of $Q$ are set to $1/20$ and $1$, respectively, whereas the remaining elements are independent samples of a uniform distribution in $[1/20,1]$. The condition number of $Q$ is therefore fixed to $20$. The problem dimension $n$ is chosen such that $n/4$ (the number of equality constraints) and $n/2$ the number (of inequality constraints) are integers. We initialize Algorithm~\ref{Alg:two} with $x_0=0$, $\lambda_0=0$.

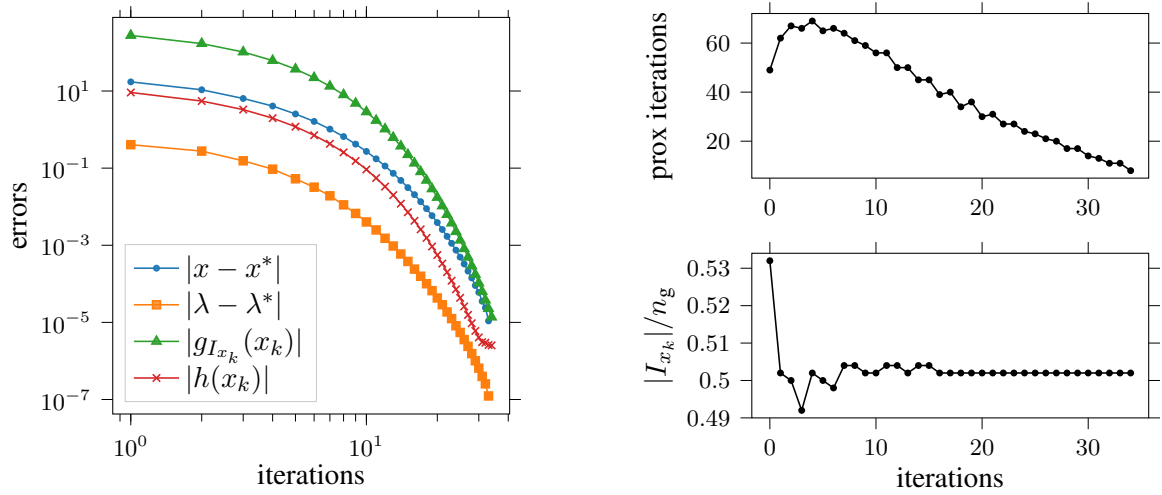
\begin{figure}

    \begin{minipage}[l]{.45\columnwidth}
    \centering
    
    \setlength{\figurewidth}{\columnwidth}
    \setlength{\figureheight}{\columnwidth}

\begin{tikzpicture}

\definecolor{color0}{rgb}{0.12156862745098,0.466666666666667,0.705882352941177}
\definecolor{color1}{rgb}{1,0.498039215686275,0.0549019607843137}
\definecolor{color2}{rgb}{0.172549019607843,0.627450980392157,0.172549019607843}
\definecolor{color3}{rgb}{0.83921568627451,0.152941176470588,0.156862745098039}

\begin{axis}[
height=\figureheight,
legend cell align={left},
legend style={fill opacity=0.8, draw opacity=1, text opacity=1, draw=white!80!black},
log basis x={10},
log basis y={10},
tick align=outside,
tick pos=both,
width=\figurewidth,
x grid style={white!69.0196078431373!black},
xlabel={iterations},
xmin=0.838351323225448, xmax=40.5557897483711,
xmode=log,
xtick style={color=black},
y grid style={white!69.0196078431373!black},
ylabel={errors},
ymin=4.20907009421747e-08, ymax=810.165310620172,
ymode=log,
ytick style={color=black},
legend pos={south west}
]
\addplot [semithick, color0, mark=*, mark size=1, mark options={solid}]
table {%
0 31.0265220102627
1 17.2530443515249
2 10.7657934316667
3 6.40082221098293
4 4.0755616485093
5 2.53513644215474
6 1.62686114501889
7 1.02734331777638
8 0.659283061849395
9 0.419841920073207
10 0.271194892547387
11 0.174357195098491
12 0.113297202853936
13 0.0734031916214155
14 0.0479821620470284
15 0.0312913954649227
16 0.0205568827165657
17 0.0134800553693251
18 0.00889246277553813
19 0.00585624151166773
20 0.00387612634621819
21 0.0025601101180018
22 0.00169872057139716
23 0.00112302772640322
24 0.000746355603362562
25 0.00049217647117076
26 0.00032728115070492
27 0.000213734539455894
28 0.000142054852722284
29 9.03106435725064e-05
30 6.00323501641202e-05
31 3.54169994385322e-05
32 2.39914576386057e-05
33 1.09383567250107e-05
};
\addlegendentry{$|x-x^*|$}
\addplot [semithick, color1, mark=square*, mark size=1.5, mark options={solid}]
table {%
0 0.596322741806995
1 0.405339502322486
2 0.276669711235918
3 0.154792611305209
4 0.0939888398788842
5 0.0529468725514675
6 0.0320842611844387
7 0.0191621326164341
8 0.0111601226642254
9 0.00668693504502981
10 0.0040204863463647
11 0.00249534355525171
12 0.00151795442811278
13 0.000955786213803625
14 0.000592964010832414
15 0.000381124208610429
16 0.000240309510206339
17 0.000157138977383077
18 0.000100465923531948
19 6.6535265071573e-05
20 4.29598399181375e-05
21 2.87111023362621e-05
22 1.8652286071888e-05
23 1.25322300230612e-05
24 8.16508305473128e-06
25 5.49017422465999e-06
26 3.57607284902733e-06
27 2.38865252445688e-06
28 1.54951014446405e-06
29 1.0120889918159e-06
30 6.50843691369651e-07
31 3.97946876273505e-07
32 2.54575370537973e-07
33 1.23497542793685e-07
};
\addlegendentry{$|\lambda-\lambda^*|$}
\addplot [semithick, color2, mark=triangle*, mark size=2, mark options={solid}]
table {%
0 16.2506427055427
1 276.122301963577
2 169.717770950808
3 102.66945402644
4 61.742505446589
5 36.9713709059114
6 22.1727600584395
7 13.3105858294424
8 7.98360216809787
9 4.78986908542081
10 2.87392149043118
11 1.72496641166184
12 1.03497987643154
13 0.620970999225241
14 0.372607866368284
15 0.223564770711565
16 0.134138671463889
17 0.0804832638371083
18 0.0482899980803167
19 0.0289740594668585
20 0.0173844722954231
21 0.0104307327858987
22 0.00625848382830052
23 0.00375512247790631
24 0.00225311755418778
25 0.00135189184076143
26 0.000811180060559176
27 0.000486722526323327
28 0.000292070289252904
29 0.000175257445250658
30 0.000105183936850851
31 6.31260786040887e-05
32 3.79019200585118e-05
33 2.27643161438483e-05
34 1.36816424664831e-05
};
\addlegendentry{$|g_{I_{x_k}}(x_k)|$}
\addplot [semithick, color3, mark=x, mark size=2, mark options={solid}]
table {%
0 15.2607061591034
1 9.15642377989377
2 5.49385431157712
3 3.29631241894607
4 1.97778737737766
5 1.18667243688457
6 0.712003454305976
7 0.427202066353079
8 0.256321240107445
9 0.15379271699484
10 0.0922756179707469
11 0.0553653432816735
12 0.0332191531397718
13 0.0199314256131441
14 0.0119587578584723
15 0.00717519530583707
16 0.00430500006987689
17 0.00258292428571761
18 0.00154960966546919
19 0.000929710068994906
20 0.000557681022634419
21 0.000334590386562211
22 0.000200573353637165
23 0.000120361060682421
24 7.20515821040426e-05
25 4.32802955985258e-05
26 2.58775183874011e-05
27 1.56234997403694e-05
28 9.4979279156772e-06
29 5.9947558984802e-06
30 4.08081775689113e-06
31 3.07871182206253e-06
32 2.9707368843249e-06
33 2.59039897321955e-06
34 2.52827662485235e-06
};
\addlegendentry{$|h(x_k)|$}
\end{axis}

\end{tikzpicture}
    \end{minipage}\hfill
    \begin{minipage}[r]{.45\columnwidth}

    \setlength{\figurewidth}{\columnwidth}
    \setlength{\figureheight}{.55\columnwidth}

    \centering
\begin{tikzpicture}

\begin{groupplot}[group style={group size=1 by 2}]
\nextgroupplot[
height=\figureheight,
tick align=outside,
tick pos=both,
width=\figurewidth,
x grid style={white!69.0196078431373!black},
xmin=-1.7, xmax=35.7,
xtick style={color=black},
y grid style={white!69.0196078431373!black},
ylabel={prox iterations},
ymin=4.95, ymax=72.05,
ytick style={color=black}
]
\addplot [semithick, black, mark=*, mark size=1, mark options={solid}]
table {%
0 49
1 62
2 67
3 66
4 69
5 65
6 66
7 64
8 61
9 59
10 56
11 56
12 50
13 50
14 45
15 45
16 39
17 40
18 34
19 36
20 30
21 31
22 27
23 27
24 24
25 23
26 21
27 20
28 17
29 17
30 14
31 13
32 11
33 11
34 8
};

\nextgroupplot[
height=\figureheight,
tick align=outside,
tick pos=both,
width=\figurewidth,
x grid style={white!69.0196078431373!black},
xlabel={iterations},
xmin=-1.7, xmax=35.7,
xtick style={color=black},
y grid style={white!69.0196078431373!black},
ylabel={$|I_{x_k}|/n_\text{g}$},
ymin=0.49, ymax=0.534,
ytick style={color=black}
]
\addplot [semithick, black, mark=*, mark size=1, mark options={solid}]
table {%
0 0.532
1 0.502
2 0.5
3 0.492
4 0.502
5 0.5
6 0.498
7 0.504
8 0.504
9 0.502
10 0.502
11 0.504
12 0.504
13 0.502
14 0.504
15 0.504
16 0.502
17 0.502
18 0.502
19 0.502
20 0.502
21 0.502
22 0.502
23 0.502
24 0.502
25 0.502
26 0.502
27 0.502
28 0.502
29 0.502
30 0.502
31 0.502
32 0.502
33 0.502
34 0.502
};
\end{groupplot}

\end{tikzpicture}
    \end{minipage}
    \caption{Trajectories for a single randomly generated convex quadratic program with $n=1000$. The figure on the left indicates linear convergence of the iterate $x_k$, the multiplier $\lambda_k$, and the constraint violations. The figures on the right display the number of iterations of the inner loop of Algorithm~\ref{Alg:two} (top) and the ratio of constraints that enter $|I_{x_k}|$ (bottom). \hchange{The solution $x^*$ is computed (approximately) with CVXOPT and a tolerance of $1e-8$.} }
    \label{Fig:randomQP1}

\end{figure}

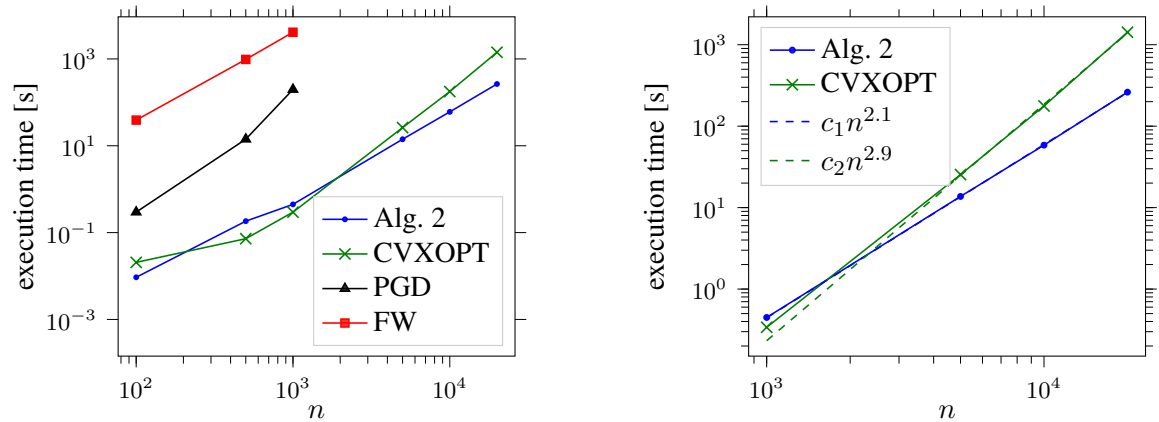
\begin{figure}

    \setlength{\figurewidth}{.45\columnwidth}
    \setlength{\figureheight}{.4\columnwidth}
    
    \begin{minipage}[l]{.45\columnwidth}
    \centering
\begin{tikzpicture}

\begin{axis}[
height=\figureheight,
legend cell align={left},
legend style={fill opacity=0.8, draw opacity=1, text opacity=1, at={(0.97,0.03)}, anchor=south east, draw=white!80!black},
log basis x={10},
log basis y={10},
tick align=outside,
tick pos=both,
width=\figurewidth,
x grid style={white!69.0196078431373!black},
xlabel={$n$},
xmin=76.7270499010926, xmax=26066.4264112612,
xmode=log,
xtick style={color=black},
y grid style={white!69.0196078431373!black},
ylabel={execution time [s]},
ymin=0.000142958995663009, ymax=9330.74552842946,
ymode=log,
ytick style={color=black}
]
\addplot [semithick, blue, mark=*, mark size=.8, mark options={solid}]
table {%
100 0.00934828806202859
500 0.18339656596072
1000 0.446982859051786
5000 14.1292053840589
10000 60.3724751899717
20000 264.774826788926
};
\addlegendentry{Alg.~\ref{Alg:two}}
\addplot [semithick, green!50!black, mark=x, mark size=3, mark options={solid}]
table {%
100 0.020619775983505
500 0.0723790290066972
1000 0.295663678087294
5000 26.1862691019196
10000 177.176193838008
20000 1421.70829877094
};
\addlegendentry{CVXOPT}
\addplot [semithick, black, mark=triangle*, mark size=2, mark options={solid}]
table {%
100 0.296458181925118
500 14.1646726869512
1000 197.883506325074
};
\addlegendentry{PGD}
\addplot [semithick, red, mark=square*, mark size=1.5, mark options={solid}]
table {%
100 39.031354442006
500 973.922650860972
1000 4118.15309932793
};
\addlegendentry{FW}
\end{axis}

\end{tikzpicture}
    \end{minipage}\hfill
    \begin{minipage}[r]{.45\columnwidth}
    \centering
\begin{tikzpicture}

\begin{axis}[
height=\figureheight,
legend cell align={left},
legend style={fill opacity=0.8, draw opacity=1, text opacity=1, at={(0.03,0.97)}, anchor=north west, draw=white!80!black},
log basis x={10},
log basis y={10},
tick align=outside,
tick pos=both,
width=\figurewidth,
x grid style={white!69.0196078431373!black},
xlabel={$n$},
xmin=860.891659331734, xmax=23231.7269928308,
xmode=log,
xtick style={color=black},
y grid style={white!69.0196078431373!black},
ylabel={execution time [s]},
ymin=0.150232198654723, ymax=2201.23094673941,
ymode=log,
ytick style={color=black}
]
\addplot [semithick, blue, mark=*, mark size=1, mark options={solid}]
table {%
1000 0.447985931998119
5000 13.7157339040423
10000 58.5082810309832
20000 261.522347006947
};
\addlegendentry{Alg.~\ref{Alg:two}}
\addplot [semithick, green!50!black, mark=x, mark size=3, mark options={solid}]
table {%
1000 0.339570223994087
5000 25.4404051849851
10000 177.234124645009
20000 1423.33261581301
};
\addlegendentry{CVXOPT}
\addplot [semithick, blue, dashed]
table {%
1000 0.444085383805538
5000 13.6093334298396
10000 59.4267160490632
20000 259.493574654642
};
\addlegendentry{$c_1 n^{2.1}$}
\addplot [semithick, green!50!black, dashed]
table {%
1000 0.232339062002444
5000 24.8447559265647
10000 185.834320204169
20000 1390.00739905923
};
\addlegendentry{$c_2 n^{2.9}$}
\end{axis}

\end{tikzpicture}
    \end{minipage}
    \caption{\hchange{This plot shows the results obtained for randomly generated quadratic programs. The figure on the left includes runtime comparisons to CVXOPT, projected gradients, and Frank-Wolfe. The figure on the right includes a detailed comparison between CVXOPT and Algorithm~\ref{Alg:two} for $n\geq 1000$.} \hchange{Algorithm~\ref{Alg:two} seems to achieve a better scaling with respect to the problem size $n$ (an exponent of 2.1 instead of 2.9), leading to a speedup of a factor of roughly 5.5 for $n=2 \cdot 10^{4}$.}}
    \label{Fig:randomQP2}
\end{figure}

The results for a randomly generated quadratic program of size $n=1000$ are shown in Figure~\ref{Fig:randomQP1}. We observe very little difference between different randomly generated programs. We also observe little change when increasing $n$; even though the computational complexity increases, the number of iterations required for convergence remains at about $35$, the maximum number of iterations that are required for computing $\lambda_k$ remains at about $70$, and only about 50\% of the inequality constraints are active. Figure~\ref{Fig:randomQP2} compares the runtime of Algorithm~\ref{Alg:two} to the interior-point solver CVXOPT, \hchange{projected gradients and Frank-Wolfe.}\footnote{We ran CVXOPT by exploiting sparsity of the Hessian and standard tolerance settings. \hchange{We also used the default settings in scipy.optimize.}} \hchange{The runtime of both Frank-Wolfe and projected gradients is comparably large and as a result, we ran these methods on smaller sized problems with $n\leq 1000$. Moreover, we stopped Frank-Wolfe after 1000 iterations even though the specified tolerance settings of $1e-6$ were not reached.} Compared to CVXOPT, the execution time of Algorithm~\ref{Alg:two} scales favorably in the problem dimension $n$ (see Figure~\ref{Fig:randomQP2}, right panel). For $n=20,000$ the execution time is roughly reduced by a factor of five \hchange{compared to CVXOPT}; even larger improvements seem possible when increasing $n$ further.

\subsection{Trust-region optimization}
In order to demonstrate that Algorithm~\ref{Alg:two} can efficiently handle nonlinear constraints, we extended the example of the previous section and considered the trust-region optimization
\begin{align*}
    \min_{x \in \mathbb{R}^n} \quad &\frac{1}{2} x\T Q x + c\T x, \qquad
    \text{s.t.} \quad A_1 x \geq 0, \quad A_2 x = 0,  \quad |x|^2\leq 1,
\end{align*}
where the matrices $Q$ and $A_1\in \mathbb{R}^{n/2 \times n}, A_2 \in \mathbb{R}^{n/4 \times n}$ and the vector $c$ are generated as in Section~\ref{subsec:RandQP}. According to Appendix~\ref{App:ProofPropCG}, the constant $L_l$ can be upper bounded as 
\begin{equation*}
    L_l \leq \bar{L}_l:= \alpha + L (2+ |Q^{-1} c| \sqrt{2}/2).
\end{equation*}
We choose $T=2/(\bar{L}_l+\mu)$ and $\alpha T=0.4$, which parallels the previous section. Figure~\ref{Fig:TR} (left) shows the number of iterations needed for computing $\lambda_k$ and the evolution of $|x_k|$ on an example with $n=1000$. The iterations of the inner loop are comparable to Section~\ref{subsec:RandQP}. The constraint $|x_k|\leq 1$ is initially not active leading to a violation at the fourth iteration. At this point, the constraint enters the set $I_{x_k}$ and its violation decreases linearly over the remaining iterations. Figure~\ref{Fig:TR} (right) shows how the execution time scales with the problem size $n$. Compared to CVXOPT, a speedup of up to two orders of magnitude is achieved.

\begin{figure}
   
    \setlength{\figurewidth}{.5\columnwidth}
    \setlength{\figureheight}{.25\columnwidth}
    
    \begin{minipage}[l]{.45\columnwidth}
    
    \setlength{\figurewidth}{\columnwidth}
    \setlength{\figureheight}{.55\columnwidth}
    \centering
\begin{tikzpicture}

\begin{groupplot}[group style={group size=1 by 2}]
\nextgroupplot[
height=\figureheight,
tick align=outside,
tick pos=both,
width=\figurewidth,
x grid style={white!69.0196078431373!black},
xmin=-1.35, xmax=28.35,
xtick style={color=black},
y grid style={white!69.0196078431373!black},
ylabel={prox iterations},
ymin=-2.05, ymax=87.05,
ytick style={color=black}
]
\addplot [semithick, black, mark=*, mark size=1, mark options={solid}]
table {%
0 83
1 57
2 56
3 56
4 46
5 43
6 41
7 39
8 37
9 34
10 32
11 30
12 27
13 25
14 23
15 21
16 19
17 17
18 15
19 13
20 10
21 8
22 7
23 6
24 4
25 4
26 3
27 2
};

\nextgroupplot[
height=\figureheight,
tick align=outside,
tick pos=both,
width=\figurewidth,
x grid style={white!69.0196078431373!black},
xlabel={iterations},
xmin=-1.4, xmax=29.4,
xtick style={color=black},
y grid style={white!69.0196078431373!black},
ylabel={\(\displaystyle |x_k|\)},
ymin=-0.0556931306175385, ymax=1.16955574296831,
ytick style={color=black}
]
\addplot [semithick, black, mark=*, mark size=1, mark options={solid}]
table {%
0 0
1 0.377352045522239
2 0.748599781472815
3 1.11386261235077
4 1.06832554015468
5 1.04099812311009
6 1.02459981100196
7 1.0147601891595
8 1.00885620824764
9 1.0053137537869
10 1.00318826079098
11 1.00191295890642
12 1.00114777600954
13 1.00068866577833
14 1.00041319950861
15 1.00024791971424
16 1.0001487518303
17 1.00008925109851
18 1.00005355065922
19 1.00003213039561
20 1.00001927823742
21 1.00001156694249
22 1.00000694016552
23 1.00000416409932
24 1.0000024984596
25 1.00000149907576
26 1.00000089944546
27 1.00000053966728
28 1.00000032380037
};
\end{groupplot}

\end{tikzpicture}
    
    \end{minipage}\hfill
    \begin{minipage}[r]{.45\columnwidth}
    
    \setlength{\figurewidth}{\columnwidth}
    \setlength{\figureheight}{1.05\columnwidth}
    \centering
\begin{tikzpicture}

\begin{axis}[
height=\figureheight,
legend cell align={left},
legend style={fill opacity=0.8, draw opacity=1, text opacity=1, at={(0.03,0.97)}, anchor=north west, draw=white!80!black},
log basis x={10},
log basis y={10},
tick align=outside,
tick pos=both,
width=\figurewidth,
x grid style={white!69.0196078431373!black},
xlabel={$n$},
xmin=76.7270499010926, xmax=26066.4264112612,
xmode=log,
xtick style={color=black},
y grid style={white!69.0196078431373!black},
ylabel={execution time [s]},
ymin=0.00302363478571102, ymax=23262.3388143668,
ymode=log,
ytick style={color=black},
legend pos=south east
]
\addplot [semithick, blue, mark=*, mark size=0.8, mark options={solid}]
table {%
100 0.00621634302660823
500 0.100896168034524
1000 0.318254444980994
5000 5.92189826397225
10000 27.6736904249992
20000 122.667250350001
};
\addlegendentry{Alg.~\ref{Alg:two}}
\addplot [semithick, green!50!black, mark=x, mark size=3, mark options={solid}]
table {%
100 0.0804486649576575
500 0.926680776057765
1000 3.65211584803183
5000 173.950378904003
10000 1683.93262361002
20000 11314.8223215881
};
\addlegendentry{CVXOPT}
\addplot [semithick, black, mark=triangle*, mark size=2, mark options={solid}]
table {%
100 0.5999885071069
500 34.0745108139236
1000 486.360369032016
};
\addlegendentry{PGD}
\addplot [semithick, red, mark=square*, mark size=1.5, mark options={solid}]
table {%
100 0.363529664929956
500 48.9326473691035
1000 1014.05425189296
};
\addlegendentry{FW}
\end{axis}

\end{tikzpicture}
    
    \end{minipage}
    \caption{The left panel shows a trajectory of Algorithm~\ref{Alg:two} for the trust-region problem. The top left indicates the number of iterations required in the inner loop for computing the multiplier $\lambda_k$, which decreases steadily. The constraint $|x_k|\leq 1$ is initially not active, leading to a violation at the fourth iteration. The violation then decreases at a linear rate, which parallels the continuous-time case. \hchange{The right graph shows the execution times for the trust-region problem, where we again ran projected gradients and Frank-Wolfe only on smaller sized problems with $n\leq 1000$.} Compared to CVXOPT, Algorithm~\ref{Alg:two} achieves a speedup of roughly two orders of magnitude for large $n$; the scaling with $n$ seems similar.}
    \label{Fig:TR}
\end{figure}
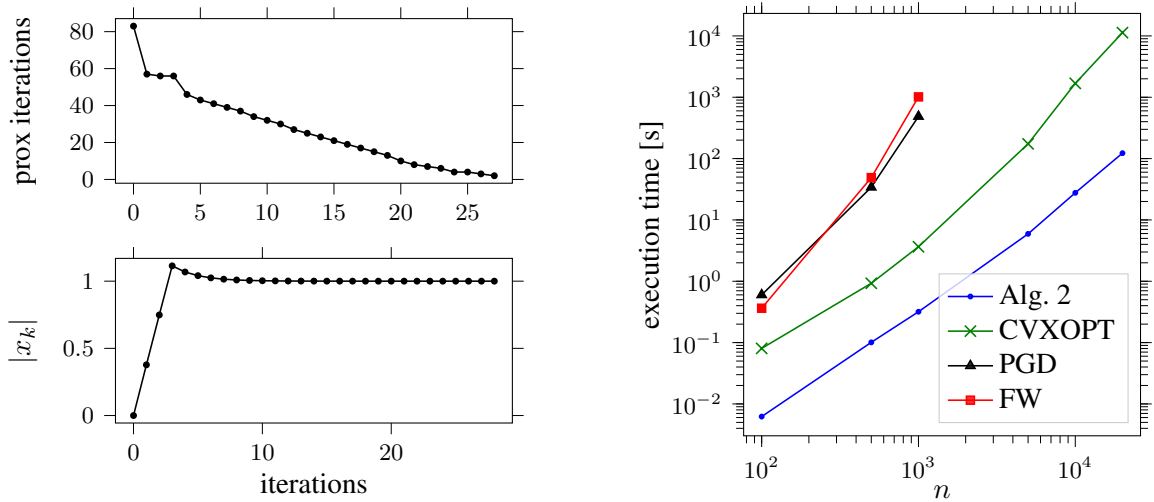

\subsection{$\nu$-support vector machine}
We used the support vector machine formulation suggested by \citet{Schoelkopf}, which leads to the following quadratic program:
\begin{align*}
    \min_{x\in \mathbb{R}^{n_\text{s}}} \quad &\frac{1}{2} \sum_{i=1}^{n_\text{s}} \sum_{j=1}^{n_\text{s}} x_i x_j l_i l_j k(r_i,r_j) + \frac{\nu_1}{2} |x|^2 \\
    \text{s.t.} \quad &0 \leq x_i \leq 1/n_{\text{s}}, \quad \sum_{n=1}^{n_\text{s}} x_i l_i = 0, \quad \sum_{i=0}^{n_\text{s}} x_i \geq \nu_2,
\end{align*}
where $r_i \in \mathbb{R}^2$ are the training samples with labels $l_i \in \{-1,1\}$, $i=1,\dots,n_\text{s}$, the integer $n_\text{s}>0$ denotes the number of training samples, $\nu_1$ and $\nu_2$ are regularization parameters, and $k:\mathbb{R}^2 \rightarrow \mathbb{R}^2$ is the kernel function. The kernel is chosen to be a radial basis function kernel with unit standard deviation. We set $\nu_1=0.1 \mu_{\text{k}}$ and $\nu_2=0.1$, where $\mu_{\text{k}}$ denotes the smallest eigenvalue of the kernel matrix $k(r_i,r_j)$. The parameter $\nu_1$ improves the conditioning of the Hessian, whereas the parameter $\nu_2$ can be interpreted as an upper bound on the fraction of margin errors; i.e., the training samples which lie on the ``wrong'' side of the boundary. It is clear that Algorithm~\ref{Alg:two}, which is based on gradient descent, has difficulties with ill-conditioned objective functions (its rate scales with $\mu/L$). The purpose of the regularization with $\nu_1$ is to reduce these effects.

We generated the training samples in the following way: The points with label +1 were generated in polar coordinates where the radius is sampled from a normal distribution with mean two and standard deviation 0.5, and the angle was uniformly sampled in $[0,2\pi)$. The points with label -1 were generated in polar coordinates where the radius was sampled from a normal distribution with mean zero and standard deviation 0.5, and the angle was uniformly sampled in $[0,2\pi)$. As an example, the training data and the resulting classifier are shown in Figure~\ref{Fig:nuSVM1} for $n_\text{s}=1000$. Due to the nature of the problem, only very few inequality constraints tend to be active at the optimum (in this case just one). The numerical results indicate that Algorithm~\ref{Alg:two} can indeed take advantage of this fact and identifies the correct active inequality constraint after very few iterations (in this case just one). The number of constraints that enter the computation of the reaction force $R_k$ is therefore significantly reduced after the first iterations enabling rapid convergence of the inner loop of Algorithm~\ref{Alg:two}. 

Figure~\ref{Fig:nuSVM1Exec} shows how the execution time scales with the problem dimension $n_\text{s}$. Compared to CVXOPT, we observe a speedup of a factor of five across all problem instances. The scaling with $n_\text{s}$ seems similar for the two methods. \hchange{Compared to projected gradients and the Frank-Wolfe implementation a speed-up of several orders of magnitude is obtained. Furthermore, Algorithm~\ref{Alg:two} also works out-of-the-box for nonconvex constraints, as is highlighted with the next example.} 

\begin{figure}

    \begin{minipage}[l]{.45\columnwidth}
    \centering
    
    \setlength{\figurewidth}{1.05\columnwidth}
    \setlength{\figureheight}{1.05\columnwidth}

    \input{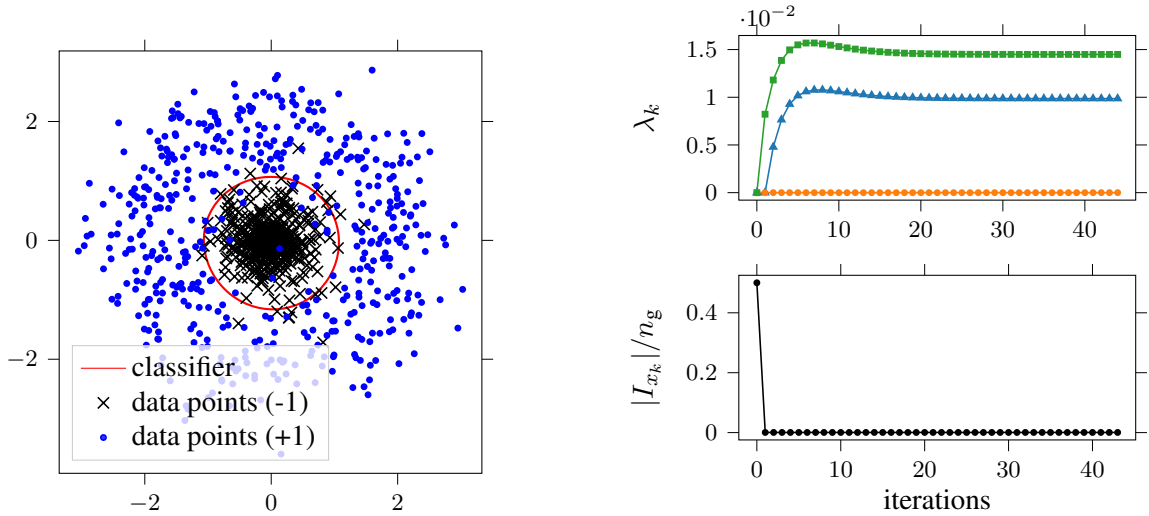}
    \end{minipage}\hfill
    \begin{minipage}[r]{.45\columnwidth}

    \setlength{\figurewidth}{\columnwidth}
    \setlength{\figureheight}{.55\columnwidth}

    \centering
\begin{tikzpicture}

\definecolor{color0}{rgb}{0.12156862745098,0.466666666666667,0.705882352941177}
\definecolor{color1}{rgb}{1,0.498039215686275,0.0549019607843137}
\definecolor{color2}{rgb}{0.172549019607843,0.627450980392157,0.172549019607843}

\begin{groupplot}[group style={group size=1 by 2}]
\nextgroupplot[
height=\figureheight,
tick align=outside,
tick pos=both,
width=\figurewidth,
x grid style={white!69.0196078431373!black},
xmin=-2.2, xmax=46.2,
xtick style={color=black},
y grid style={white!69.0196078431373!black},
ylabel={\(\displaystyle \lambda_k\)},
ymin=-0.000784466420187006, ymax=0.0164737948239271,
ytick style={color=black}
]
\addplot [semithick, color0, mark=triangle*, mark size=1.5, mark options={solid}]
table {%
0 0
1 0
2 0.00478720152818964
3 0.0076447761760491
4 0.00928496193272565
5 0.010170299283715
6 0.0105982355346979
7 0.0107577689175972
8 0.0107671875441554
9 0.0106990803069012
10 0.0105968149744664
11 0.0104853213460473
12 0.0103780917832319
13 0.0102816827811659
14 0.0101985757587641
15 0.0101289682698665
16 0.0100718739568699
17 0.0100257804127125
18 0.00998902798218086
19 0.00996001534886979
20 0.00993729998475176
21 0.00991963674215101
22 0.00990598170199304
23 0.009895477942035
24 0.00988743320300506
25 0.00988129520813993
26 0.00987662776883077
27 0.00987308921551421
28 0.00987041375399803
29 0.0098683958188401
30 0.00986687721925849
31 0.00986573674824674
32 0.00986488188821808
33 0.00986424225689241
34 0.00986376447111803
35 0.00986340814963042
36 0.00986314282035365
37 0.00986294553952604
38 0.00986279906675841
39 0.00986269047151756
40 0.00986261007260227
41 0.00986255063344078
42 0.0098625067531268
43 0.00986247440668987
44 0.00986245059878619
};
\addplot [semithick, color1, mark=*, mark size=1, mark options={solid}]
table {%
0 0
1 0
2 0
3 0
4 0
5 0
6 0
7 0
8 0
9 0
10 0
11 0
12 0
13 0
14 0
15 0
16 0
17 0
18 0
19 0
20 0
21 0
22 0
23 0
24 0
25 0
26 0
27 0
28 0
29 0
30 0
31 0
32 0
33 0
34 0
35 0
36 0
37 0
38 0
39 0
40 0
41 0
42 0
43 0
44 0
};
\addplot [semithick, color2, mark=square*, mark size=1., mark options={solid}]
table {%
0 0
1 0.00822085070448518
2 0.0118077030540806
3 0.0138611515407122
4 0.0149649162911507
5 0.0154937774042507
6 0.0156856410262824
7 0.0156893284037401
8 0.0155961460198436
9 0.0154606416027554
10 0.0153141696175571
11 0.015173692698211
12 0.0150474406616955
13 0.0149385080067657
14 0.0148471079524494
15 0.0147719580607017
16 0.0147111101477365
17 0.0146624290825034
18 0.0146238533569649
19 0.0145935229599634
20 0.0145698289922197
21 0.0145514191779523
22 0.0145371803060471
23 0.0145262102278766
24 0.0145177867215481
25 0.014511337222637
26 0.0145064114050689
27 0.0145026574034188
28 0.0144998017931415
29 0.0144976330897616
30 0.0144959883646639
31 0.0144947425232881
32 0.0144937998010629
33 0.0144930870726594
34 0.0144925486229587
35 0.0144921420831772
36 0.0144918352874228
37 0.0144916038509947
38 0.0144914293111614
39 0.0144912977040311
40 0.0144911984780458
41 0.0144911236663393
42 0.0144910672575205
43 0.0144910247181257
44 0.0144909926307205
};

\nextgroupplot[
height=\figureheight,
tick align=outside,
tick pos=both,
width=\figurewidth,
x grid style={white!69.0196078431373!black},
xlabel={iterations},
xmin=-2.15, xmax=45.15,
xtick style={color=black},
y grid style={white!69.0196078431373!black},
ylabel={\(\displaystyle |I_{x_k}|/n_\text{g}\)},
ymin=-0.024487756121939, ymax=0.525237381309345,
ytick style={color=black}
]
\addplot [semithick, black, mark=*, mark size=1, mark options={solid}]
table {%
0 0.500249875062469
1 0.000499750124937531
2 0.000499750124937531
3 0.000499750124937531
4 0.000499750124937531
5 0.000499750124937531
6 0.000499750124937531
7 0.000499750124937531
8 0.000499750124937531
9 0.000499750124937531
10 0.000499750124937531
11 0.000499750124937531
12 0.000499750124937531
13 0.000499750124937531
14 0.000499750124937531
15 0.000499750124937531
16 0.000499750124937531
17 0.000499750124937531
18 0.000499750124937531
19 0.000499750124937531
20 0.000499750124937531
21 0.000499750124937531
22 0.000499750124937531
23 0.000499750124937531
24 0.000499750124937531
25 0.000499750124937531
26 0.000499750124937531
27 0.000499750124937531
28 0.000499750124937531
29 0.000499750124937531
30 0.000499750124937531
31 0.000499750124937531
32 0.000499750124937531
33 0.000499750124937531
34 0.000499750124937531
35 0.000499750124937531
36 0.000499750124937531
37 0.000499750124937531
38 0.000499750124937531
39 0.000499750124937531
40 0.000499750124937531
41 0.000499750124937531
42 0.000499750124937531
43 0.000499750124937531
};
\end{groupplot}

\end{tikzpicture}
    \end{minipage}
    \caption{This figure shows the training data and the resulting classifier (left), as well as the dual variables $\lambda_k$ (top right) and the percentage of constraints that are active (bottom right). Only two dual variables are nonzero: The first one corresponds to the equality constraints (blue line, triangular marks), whereas the second one corresponds to the support constraint (green line, square marks). All remaining constraints are inactive.}
    \label{Fig:nuSVM1}

\end{figure}

\begin{figure}
\setlength{\figurewidth}{.5\columnwidth}
\setlength{\figureheight}{.4\columnwidth}
\centering
\begin{tikzpicture}

\begin{axis}[
height=\figureheight,
legend cell align={left},
legend style={fill opacity=0.8, draw opacity=1, text opacity=1, at={(0.97,0.03)}, anchor=south east, draw=white!80!black},
log basis x={10},
log basis y={10},
tick align=outside,
tick pos=both,
width=\figurewidth,
x grid style={white!69.0196078431373!black},
xlabel={$n$},
xmin=79.4328234724281, xmax=12589.2541179417,
xmode=log,
xtick style={color=black},
y grid style={white!69.0196078431373!black},
ylabel={execution time [s]},
ymin=0.00498447385200959, ymax=5214.71935284061,
ymode=log,
ytick style={color=black}
]
\addplot [semithick, blue, mark=*, mark size=.8, mark options={solid}]
table {%
100 0.00935918500181288
500 0.0596687720390037
1000 0.202956121996976
5000 6.30003201006912
10000 39.5290712949354
};
\addlegendentry{Alg.~\ref{Alg:two}}
\addplot [semithick, green!50!black, mark=x, mark size=3, mark options={solid}]
table {%
100 0.0190279979724437
500 0.270598810049705
1000 1.07230346393771
5000 37.2783174070064
10000 163.624870460015
};
\addlegendentry{CVXOPT}
\addplot [semithick, black, mark=triangle*, mark size=2, mark options={solid}]
table {%
100 0.321677375002764
500 17.8351045510499
1000 169.808370316983
};
\addlegendentry{PGD}
\addplot [semithick, red, mark=square*, mark size=1.5, mark options={solid}]
table {%
100 29.10206392908
500 538.948184055043
1000 2777.23244649696
};
\addlegendentry{FW}
\end{axis}

\end{tikzpicture}
\caption{\hchange{This figure shows the execution times for the different optimization algorithms on the $\nu$-support vector machine problem.}}
\label{Fig:nuSVM1Exec}
\end{figure}
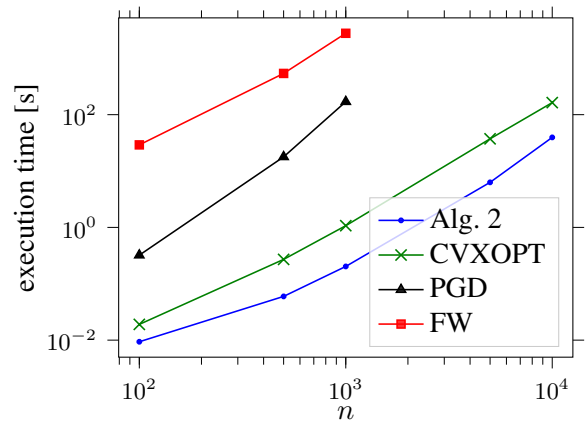

\subsection{Catenary}\label{Subsec:Cat}
We consider an idealized chain of length two, which has $n$ chain links and is suspended at the points $(0,0)$ and $(1,0)$ (in a two-dimensional coordinate system). The aim is to solve the following problem:
\begin{align*}
    \min_{(x,y) \in \mathbb{R}^{n+1} \times \mathbb{R}^{n+1}} \quad &\frac{9.81}{n+1} \sum_{i=2}^{n}  y_i\\
    \text{s.t.} \quad &|x_i-x_{i+1}|^2 + |y_i-y_{i+1}|^2 = 4/n^2, \\ &|x_i-0.5|^2 + |y_i+0.8|^2 \geq 0.5^2, \quad i=1,\dots,n\\
    &(x_1,y_1)=(0,0), \quad (x_{n+1},y_{n+1})=(1,0).
\end{align*}
The position of the $i$th joint is described by the tuple $(x_i,y_i)$ and we have included the nonlinear constraint that each joint is required to lie outside a circle centered at $(0.5,-0.8)$ with radius $0.5$ (the chain therefore lies on a circular object). The cost function captures the potential energy of the chain. We found that a time step of $T=2/n$ works well, which can be motivated by the fact that $L_l$ is roughly $\mathcal{O}(n)$ (considering the continuous limit of the chain). The results for a chain of length $n=40$ can be found in Figure~\ref{Fig:catenary}. Starting from a random initialization that violates the equality constraints (see Figure~\ref{Fig:catenary} (left, black)) the solution evolves and finds a local minimum that satisfies all the constraints. We note that the cost has a plateau at about iteration 1000, which corresponds to a symmetric shape, where the chain lies on top of the round object (see Figure~\ref{Fig:catenary} (left, green crosses)). This corresponds to an unstable equilibrium, since the slightest deviation will cause the chain to slide down either to the left or the right. This is precisely what we observe in our numerical results, leading to the final solution shown in red (square marks).

\begin{figure}
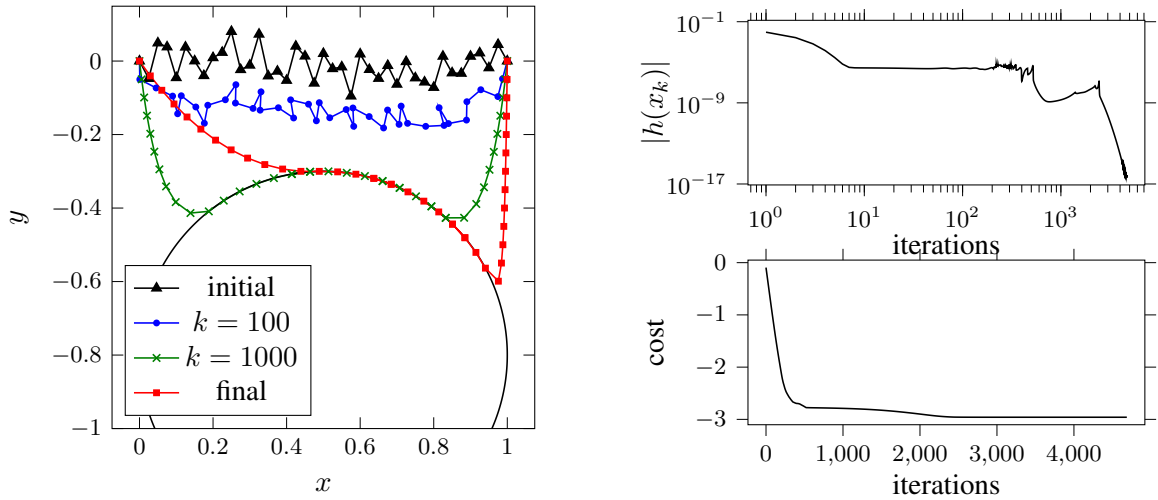


    \begin{minipage}[l]{.45\columnwidth}
    \centering
    
    \setlength{\figurewidth}{1.05\columnwidth}
    \setlength{\figureheight}{1.05\columnwidth}

    \input{media/catenary_1.tikz}
    \end{minipage}\hfill
    \begin{minipage}[r]{.45\columnwidth}

    \setlength{\figurewidth}{\columnwidth}
    \setlength{\figureheight}{.55\columnwidth}

    \centering
    \input{media/catenary_2.tikz}
    \end{minipage}
    \caption{This figure shows the evolution of the solution of the catenary problem (left) as well as the violation of the equality constraints and the evolution of the cost (right). We can clearly see that the symmetric shape (roughly corresponding to the solution at time $k=1000$) is suboptimal, and unstable from a physics perspective. Thus, the chain slides to the right and reaches a lower energy state. (We suspect that the random initialization and the finite precision breaks the symmetry.)}
    \label{Fig:catenary}

\end{figure}

\section{Conclusions}\label{Sec:Conclusion}
We have presented a new class of primal first-order algorithms for smooth constrained optimization. The key feature of these algorithms is that at each iteration, a low-dimensional, local, and convex approximation of the feasible set is constructed and used for computing the next iterate. The local approximation is a natural generalization of the tangent cone (in the sense of Clarke) to include infeasible points. It can be motivated by drawing analogies to non-smooth mechanical systems and can be viewed as a reformulation of constraint optimization on the velocity level. That is, the algorithm imposes a constraint on $x_{k+1}-x_k$ rather than on $x_k$. While in our continuous-time formulation constraints on the position level and the velocity level are equivalent, this is no longer true for the resulting discrete-time algorithms. We found that a formulation of constraints on the velocity level leads to efficient first-order algorithms that avoid projection or optimization over the entire feasible set at each iteration. This simplification requires a more complex theoretical analysis, but, as we have shown, that analysis can be carried out with a blend of ideas from dynamical systems and mathematical optimization. 

We have aimed to highlight and explicate the philosophical and conceptual novelty of our approach to constrained optimization. Many aspects of the general approach deserve a more thorough treatment. For example, we have not discussed existence of solutions to the non-smooth differential equations or the differential inclusions that we have introduced. Similarly, the strong convexity assumptions on the objective function for proving convergence of our discrete algorithm can most likely be relaxed, and the numerical experiments do not include an extensive comparison to different state-of-the-art solvers. We also acknowledge that there are software packages that are tailored to, for example, support vector machines, which would most likely outperform our method. 

There are many opportunities for further research in this vein. In particular, the analogies to non-smooth mechanical systems that are made throughout the article enable extensions to Newton-type methods or accelerated first-order methods. We hope that our perspective helps to trigger further developments at the intersection between non-smooth dynamics, constrained optimization, and machine learning.


\acks{We thank the German Research Foundation and the Branco Weiss Fellowship, administered by ETH Zurich, for the generous support. We also thank the Office of Naval Research under grant number N00014-18-1-2764.}


\appendix
\section{Properties of $d$}\label{App:d}

\hchange{We recall from Section~\ref{Sec:Discrete} that $f_I^*$ denotes the optimal function value that arises when considering modifications of \eqref{eq:fundProb}, where some inequality constraints are removed. We recall the definition of $f_I^*$ by restating \eqref{eq:fseq}}:
\begin{equation*}
f^*_I:=\min_{x\in \mathbb{R}^n} f(x) \quad \text{s.t.} \quad  h(x)=0, \quad g_i(x)\geq 0, \quad i\in I,
\end{equation*}
\hchange{where $I$ is any subset of $\{1,\dots,n_\text{g}\}$. We further note that $f^*_{\{\}} \leq f^*_I \leq f^*$ and that $x_I^*$ denotes any minimizer of \eqref{eq:fseq} with $\lambda_I^*$ the corresponding multipliers that satisfy the Karush-Kuhn-Tucker conditions of \eqref{eq:fseq}.}

\begin{lemma}\label{Lem:discrete2}
\hchange{Let Assumption~\ref{Ass:Conv} (convexity) be satisfied.} For $0\leq \alpha \leq \mu$ and any $x\in \mathbb{R}^n$ the following upper and lower bounds on $d(x)$ hold
\begin{equation*}
f_{I_x}^* - \frac{1}{2\alpha} \frac{L_l}{\alpha} \left( 1 -\frac{\alpha}{L_l} \right) |v(x)|^2 \quad \leq \quad d(x) \quad \leq \quad f_{I_x}^* -\frac{1}{2 \alpha} \left(1-\frac{\alpha}{\mu}\right) |v(x)|^2,
\end{equation*}
where $L_l \geq \bar{L}_l(\lambda_{I_x}^*)$.
\end{lemma}
\begin{proof}
We start by deriving the upper bound. We conclude from the strong convexity of $l$ that
\begin{align}
f^* \geq f^*_{I_x} \geq \inf_{z\in \mathbb{R}^n} l(z,\lambda) &\geq l(x,\lambda) - \frac{1}{2\mu} |\nabla_x l(x,\lambda)|^2, \quad \forall \lambda \in D_x. \label{eq:importantUBd2}
\end{align}
By rearranging terms we obtain
\begin{align*}
f_{I_x}^* - \frac{1}{2 \alpha} \left(1-\frac{\alpha}{\mu}\right) |\nabla_x l(x,\lambda)|^2 &\geq l(x,\lambda) - \frac{1}{2 \alpha} |\nabla_x l(x,\lambda)|^2, \quad \forall \lambda \in D_x,\nonumber
\end{align*}
which yields the desired upper bound for $\lambda=\lambda(x)$. 

In order to obtain the lower bound, we first note that the smoothness of $l(\cdot, \lambda^*_{I_x})$ (where $\lambda_{I_x}^*$ is fixed) implies
\begin{equation}
f^*_{I_x}=\inf_{z \in \mathbb{R}^n} l(z,\lambda^*_{I_x})\leq l(x,\lambda^*_{I_x})-\frac{1}{2L_l} |\nabla_x l(x,\lambda_{I_x}^*)|^2. \label{eq:prooftmp3331}
\end{equation}
We further consider the modified primal and dual problems (where $\alpha$ is replaced by $L_l$):
\begin{equation*}
v_\text{m}(x) = \argmin_{v\in V_{L_l}(x)} \frac{1}{2} |v+\nabla f(x)|^2, \quad \lambda_\text{m}(x) \in  \argmax_{\lambda \in D_x} l(x,\lambda) - \frac{1}{2L_l} |\nabla_x l(x,\lambda)|^2,
\end{equation*}
and note that $v(x) L_l/\alpha \in V_{L_l}(x)$, hence $v(x) L_l/\alpha$ is a feasible candidate for minimization over $V_{L_l}(x)$. This means that
\begin{align*}
\frac{1}{2} |v_m(x)|^2 + v_m(x)\T \nabla f(x) \leq \frac{1}{2} \frac{L_l^2}{\alpha^2} |v(x)|^2 + \frac{L_l}{\alpha} v(x)\T \nabla f(x).
\end{align*}
Complementary slackness implies that $v(x)\T\nabla f(x)=-|v(x)|^2 - \alpha \bar{g}(x)\T \lambda(x)$ and similarly $v_\text{m}(x)\T \nabla f(x)=-|v_\text{m}(x)|^2 - L_l \bar{g}(x)\T \lambda_\text{m}(x)$, and yields therefore
\begin{align*}
-\frac{1}{2} |v_\text{m}(x)|^2 - L_l \bar{g}(x)\T \lambda_\text{m}(x)\leq \frac{1}{2} \frac{L_l}{\alpha} \left(\frac{L_l}{\alpha} -2\right) |v(x)|^2 - L_l \bar{g}(x)\T \lambda(x).
\end{align*}
Dividing by $L_l$ and adding $f(x)$ on both sides implies that
\begin{equation*}
\max_{\lambda \in D_x} l(x,\lambda) - \frac{1}{2L_l} |\nabla_x l(x,\lambda)|^2 \leq d(x) + \frac{1}{2\alpha} \left( \frac{L_l}{\alpha} -1 \right) |v(x)|^2.
\end{equation*}
The left-hand side includes a maximum over $\lambda$, which means that for $\lambda_{I_x}^*\in D_x$, we have:
\begin{equation*}
l(x,\lambda_{I_x}^*) - \frac{1}{2L_l} |\nabla_x l(x,\lambda_{I_x}^*)|^2 \leq d(x) +\frac{1}{2\alpha} \left( \frac{L_l}{\alpha} -1 \right) |v(x)|^2.
\end{equation*}
Combining the previous inequality with \eqref{eq:prooftmp3331} yields the desired lower bound.
\end{proof}

\begin{lemma}\label{Lemma:propD}
\hchange{Let Assumption~\ref{Ass:Conv} (convexity) be satisfied. For $0 \leq \alpha<\mu$ and any $x\in \mathbb{R}^n$ the following holds
\begin{equation*}
   d(x) \leq  f_{I_x}^*-\frac{\mu-\alpha}{2} |x-x_{I_x}^*|^2, \qquad d(x) \leq f^* - \frac{\mu-\alpha}{2} |x-x^*|^2.
\end{equation*}
Moreover, $d(x)=f^*$ if and only if $x=x^*$.}
\end{lemma}
\begin{proof}
\hchange{We conclude from strong duality that
\begin{equation*}
    d(x) = f(x) - \frac{1}{2\alpha} |\nabla f(x)|^2 + \frac{1}{\alpha} \min_{v\in V_\alpha(x)} \frac{1}{2} |v+\nabla f(x)|^2.
\end{equation*}
Furthermore, $\alpha (x_{I_x}^*-x) \in V_\alpha(x)$ due to the fact that $C$ is convex. Hence, $\alpha (x_{I_x}^*-x)$ is a feasible candidate in the above minimization, which yields
\begin{equation*}
    d(x)\leq f(x)+\frac{\alpha}{2} |x_{I_x}^*-x|^2 + \nabla f(x)\T (x_{I_x}^*-x)\leq f_{I_x}^*+\frac{\alpha - \mu}{2} |x_{I_x}^*-x|^2,
\end{equation*}
where the strong convexity of $f$ has been used in the second step. Rearranging terms yields the first inequality. The second inequality follows from the same reasoning (we simply replace $x_{I_x}^*$ with $x^*$). This also implies $d(x)=f^*$ if and only if $x=x^*$.}
\end{proof}

\section{Barbalat's lemma}\label{App:Barbalat}
\begin{lemma}\label{Lemma:cont2}
(Variant of Barbalat's lemma) Let $\xi: \mathbb{R}_{\geq 0} \rightarrow\mathbb{R}$ be piecewise continuous, such that \begin{equation*}
-\infty < \int_{0}^{\infty} \xi(\tau) \diff \tau , \quad \xi(t)^+\geq \xi(t)^-, \quad \xi(t_2)-  \xi(t_1) \geq - \bar{r} (t_2-t_1), 
\end{equation*} 
for any $t_2\geq t_1>0$ such that $\xi$ is continuous on $(t_1,t_2)$ and any $t\geq 0$.  If $\bar{\xi}: \mathbb{R}_{\geq 0} \rightarrow \mathbb{R}_{\geq 0}$, with $\bar{\xi}(x):=\max\{\xi(x),0\}$, is integrable and such that $\lim_{t\rightarrow \infty} \bar{\xi}(t) = 0$, then $\lim_{t\rightarrow \infty} \xi(t)=0$ holds.
\end{lemma}
\begin{proof}
The proof follows a standard argument, which is also used for proving Barbalat's lemma~\citep[see, e.g.,][p.~204]{Sastry}. \hchange{We start by assuming that $\xi(t)$ does not converge to zero and show that this leads to a contradiction.} This means that there exists an $\epsilon>0$ and a sequence $t_k\geq 0$, such that $\xi(t_k)< -\epsilon$ for all $k>0$ (taking into account that $\lim_{t\rightarrow \infty} \bar{\xi}(t) =0$). However, since $\xi(t)^+ \geq \xi(t)^-$ at every $t$ where $\xi$ is discontinuous, we conclude that $\xi(t) \leq \xi(t_1) + \bar{r} (t_1-t)$ for all $t\leq t_1$, where $t_1>0$ is arbitrary (looking backwards in time, the function increases by a slope of at most $\bar{r}$). For each $t_k$, we thus conclude $\xi(t) < - \epsilon/2$ as long as $t\in (t_k-\epsilon/(2 \bar{r}), t_k)$. This means that for any subsequence $t_{kj}$, $j=1,2,\dots$ such that $t_{k(j+1)}>t_{kj}+\epsilon/(2\bar{r})$,
\begin{equation*}
\int_{0}^{\infty} \xi(\tau) \diff \tau = \sum_{j=1}^{\infty} \int_{t_{k(j-1)}}^{t_{kj}-\epsilon/(2\bar{r})} \xi(\tau) \diff \tau + \int_{t_{kj}-\epsilon/(2\bar{r})}^{t_{kj}} \xi(\tau) \diff \tau \leq   \int_{0}^{\infty} \bar{\xi}(\tau) \diff \tau - \sum_{j=1}^{\infty} \epsilon^2/(4\bar{r}),
\end{equation*}
where $t_{k0}$ is defined as $t_{k0}=0$ for notational convenience. The right-hand side is unbounded below leading to the desired contradiction.
\end{proof}

\section{Nonlinear constraints}\label{App:ProofPropCG}
When estimating the constant $L_l$, a bound on $\lambda$ is often useful. The following proposition, which can be generalized to multiple constraints by a similar argument, establishes such a bound.

\begin{proposition}\label{Prop:BoundCg}
Let $g: \mathbb{R}^n \rightarrow \mathbb{R}$ be a scalar $L_\text{g}$-smooth and $\mu_\text{g}$-strongly concave function. Then, in the absence of any other constraints, the corresponding multiplier $\lambda>0$ is bounded by
\begin{equation*}
    \lambda \leq \frac{1}{\mu_\text{g}} \left( \alpha + L (1+ |x_\text{f}-x_\text{g}| \sqrt{L_\text{g}/(2g(x_\text{g}))})\right),
\end{equation*}
where $x_\text{f}$ is the (unconstrained) minimizer of $f$, $x_\text{g}$ the (unconstrained) maximizer of $g$, and $L$ the smoothness constant of $f$.
\end{proposition}

\begin{proof}
It follows from \eqref{eq:velLeveltmp2} that
\begin{equation*}
    \frac{1}{2} |W(x) \lambda|^2 = \frac{1}{2} |v(x) + \nabla f(x)|^2 \leq \frac{1}{2} |v_\text{f} +\nabla f(x)|^2,
\end{equation*}
for any $v_\text{f} \in V_\alpha(x)$. In particular, we can set $v_\text{f}=\alpha (x_\text{g}-x)$, which satisfies $v_\text{f} \in V_\alpha(x)$, due to the concavity of $g$. Moreover, from $W(x)=\nabla g(x)$ and the strong concavity of $g$ we conclude
\begin{equation*}
    \mu_\text{g} |x-x_\text{g}| \lambda \leq |\alpha (x_\text{g}-x) + \nabla f(x)| \leq \alpha |x-x_\text{g}| + L |x-x_\text{f}|,
\end{equation*}
where $\lambda>0$ by definition of $\lambda$.
This yields the following bound on the dual variable
\begin{equation*}
    \lambda \leq \sup_{g(x)\leq 0} \frac{1}{\mu_\text{g}} \left( \alpha + \frac{ L |x-x_\text{f}|}{|x-x_\text{g}|} \right),
\end{equation*}
which can be further simplified to
\begin{equation*}
\lambda \leq \frac{1}{\mu_\text{g}} \left( \alpha + L + L |x_\text{g}-x_\text{f}| \sup_{g(x)\leq 0} \frac{1}{|x-x_\text{g}|} \right).
\end{equation*}
Due to the strong concavity of $g$, it follows that $g(x) \geq g(x_\text{g}) - L_\text{g} |x-x_\text{g}|^2/2$ for all $x\in \mathbb{R}^n$. As a consequence, $L_\text{g} |x-x_\text{g}|^2/2 \geq g(x_\text{g})$, for all $x\in \mathbb{R}^n$ such that $g(x)\leq 0$, which means that the last supremum is bounded by $\sqrt{L_\text{g}/(2 g(x_\text{g}))}$.
\end{proof}

\section{Proof of Proposition~\ref{Prop:ConvergenceLambda}}\label{App:ProofLambda}
\begin{proof}
The proof follows the presentation of \citet[p.~400]{Cottle}. In order to simplify the notation we define $G:=W_k\T W_k$, $q:=-W_k\T \nabla f(x_k) + \alpha \bar{g}(x_k)$, $B:=U\T+\omega^{-1} D$, $C:=U + (1-\omega^{-1})D$, and omit the subscript $k$. We can therefore express \eqref{eq:statdual} concisely as
\begin{equation*}
    G \lambda + q + \partial \psi_{D_{x}}(\lambda) \ni 0.
\end{equation*}
Furthermore, by virtue of the conjugate subgradient theorem, \eqref{eq:statdual2} is equivalent to
\begin{equation}
    \lambda^{j+1} \in D_{x}, \quad -B \lambda^{j+1} - C \lambda^j - q \in D_{x}^*, \quad {\lambda^{j+1}}\T(-B \lambda^{j+1} - C \lambda^j - q)=0, \label{eq:prooftmp331}
\end{equation}
where $D_{x}^*:=\{0\}^{n_\text{h}} \times \mathbb{R}_{\leq 0}^{|I_{x}|}$ is the polar cone of $D_{x}$. We further introduce the function $\tilde{d}: D_{x} \rightarrow \mathbb{R}$, $\tilde{d}(\lambda)=\lambda\T G \lambda/2 + \lambda\T q$. Due to the fact that $G$ is positive semi-definite, $\tilde{d}$ is convex and can be shown to be bounded below for $\lambda \in D_x$. We further have that
\begin{equation*}
    \tilde{d}(\lambda^{j})-\tilde{d}(\lambda^{j+1}) = (\lambda^{j}-\lambda^{j+1})\T (q+G \lambda^{j+1}) + \frac{1}{2} (\lambda^{j}-\lambda^{j+1})\T G (\lambda^{j}-\lambda^{j+1}).
\end{equation*}
As a consequence of \eqref{eq:prooftmp331} and some elementary manipulations, the decrease in $\tilde{d}$ can be expressed as
\begin{align*}
    \tilde{d}(\lambda^{j})-\tilde{d}(\lambda^{j+1}) &= {\lambda^j}\T (q+B\lambda^{j+1}+C\lambda^j) + \frac{1}{2} (\lambda^{j}-\lambda^{j+1})\T (B-C) (\lambda^{j}-\lambda^{j+1})\\
    &\geq \frac{1}{2} (\lambda^{j}-\lambda^{j+1})\T (B-C) (\lambda^{j}-\lambda^{j+1}).
\end{align*}
For the last inequality we have used $-B \lambda^{j+1} -C \lambda^j -q \in D_{x}^*$ and $\lambda^j\in D_{x}$, which ensures that $ (q+B\lambda^{j+1}+C\lambda^j)\T\lambda^j\geq 0$. The symmetric part of $B-C$ is given by $(2\omega^{-1} -1) D$, which is guaranteed to be positive definite for $\omega \in (0,2)$ (the elements of $D$ are given by $|\nabla g_i(x)|^2>0$). This concludes that $\tilde{d}(\lambda^{j})$ is a monotonically decreasing sequence, which therefore converges. Thus, the above inequality implies, in the limit as $j\rightarrow \infty$,
\begin{equation*}
    0=\lim_{j\rightarrow \infty} \frac{1}{2} (\lambda^{j}-\lambda^{j+1})\T (B-C) (\lambda^j-\lambda^{j+1}),
\end{equation*}
which, due to the positive definiteness of the symmetric part of $B-C$, implies that $\lambda^j$ converges. Moreover, $\lim_{j\rightarrow \infty} \lambda^j$ satisfies \eqref{eq:statdual} by construction.
\end{proof}

\section{Asymptotic rate of convergence}\label{App:AsymptoticRate}
\begin{proposition}\label{Prop:AsymptoticRate}
\hchange{Let Assumption~\ref{Ass:Conv} (convexity) be satisfied and let $T\leq 2/(L_\text{l} + \mu)$, $\alpha < \mu$, where $L_\text{l}\geq \bar{L}_\text{l}(\lambda(x))$ for all $x\in \mathbb{R}^n$. Then, for every $x_0\in \mathbb{R}^n$ there exists constants $N>0$ and $c_3>0$, such that}
\begin{equation}
\frac{1}{2} |v(x_{k+1})|^2 \leq (1-2\mu T (1-\mu T/2)) \frac{1}{2} |v(x_k)|^2 + c_3 (1-\alpha T)^{k-N}, \quad \forall k\geq N. \label{eq:lptmp1}
\end{equation}
\hchange{In particular, for $T=1/L_\text{l}$, $\alpha=\mu/2$, there exists a constant $c_4$ such that }
\begin{equation}
    |x_k-x^*|^2 \leq c_4 (1-\mu/(2 L_\text{l}))^{k}, \quad \forall k\geq N. \label{eq:lptmp2}
\end{equation}
\end{proposition}
\begin{proof}
\hchange{We first note that the assumptions of Proposition~\ref{Prop:GD} are satisfied and that as a consequence, $x_k$ converges to $x^*$, $I_{x_k}$ is constant for large $k$, and $x_k$ and $\lambda_k$ are bounded. By Lemma~\ref{Lem:discrete2} we conclude that $\lambda_k$ is a feasible candidate for the dual \eqref{eq:dual} at iteration $k+1$. As a result,
\begin{align*}
    \frac{1}{2} |v(x_{k+1})|^2 + \alpha \lambda_{k+1}\T \bar{g}(x_{k+1}) &\leq \frac{1}{2} |-\nabla f(x_{k+1}) + \nabla \bar{g}(x_{k+1}) \lambda_{k}|^2 + \alpha \lambda_k\T \bar{g}(x_{k+1})\\
    & \leq \frac{1}{2} |v(x_k) - \Delta_x l(\xi_k,\lambda_k) T v(x_k)|^2 + \alpha \lambda_k\T \bar{g}(x_{k+1}),
\end{align*}
where $\Delta_x l$ denotes the second derivative of $l$ with respect to $x$, and $\xi_k$ lies between $x_k$ and $x_{k+1}$ (we applied Taylor's theorem in the second step). By the same reasoning as in the proof of Claim~\ref{Claim:3} we conclude
\begin{align*}
    \frac{1}{2} |v(x_{k+1})|^2 + \alpha \lambda_{k+1}\T \bar{g}(x_{k+1}) & \leq (1-2\mu T (1-\mu T/2)) \frac{1}{2} |v(x_k)|^2 + \alpha \lambda_k\T \bar{g}(x_k).
\end{align*}
We further note that $\lambda_k\T \bar{g}(x_k)$ is negative, which leads to 
\begin{align*}
    \frac{1}{2} |v(x_{k+1})|^2 & \leq (1-2\mu T (1-\mu T/2)) \frac{1}{2} |v(x_k)|^2 - \alpha \lambda_{k+1}\T \bar{g}(x_{k+1}).
\end{align*}
For large $k$, $I_{x_k}$ is constant, see Claim~\ref{Claim:disLast} in Section~\ref{Sec:Discrete}, and $|\bar{g}(x_{k})|$ is therefore decaying at the rate $1-\alpha T$ (see Lemma~\ref{Lem:discrete}). This yields \eqref{eq:lptmp1}.}

\hchange{The second part follows from unrolling the recursion in \eqref{eq:lptmp1} and noting that by Lemma~\ref{Lem:discrete2} and Lemma~\ref{Lemma:propD}, $|x_k-x^*|^2$ is bounded by a multiple of $|v(x_k)|^2$.}
\end{proof}

\bibliography{literature}

\end{document}